\documentclass[twoside,11pt]{article}

\usepackage{jmlr2e}

\usepackage{soul}
\usepackage{microtype}
\usepackage{graphicx}
\usepackage{subfigure}
\usepackage{booktabs} 
\usepackage{makecell}
\usepackage{fancybox}
\usepackage{soul} 
\usepackage{graphicx} 
\usepackage{amsmath}
\usepackage{booktabs}
\usepackage{algorithm}
\usepackage{algorithmic}
\usepackage{todonotes}

\usepackage{amsfonts} 
\usepackage{verbatim} 
\usepackage{amssymb}
\usepackage{verbatim}
\usepackage{enumitem}
\usepackage{sidecap,color} 

\usepackage{multirow} 
\def \S {\mathbf{S}}

\def \O {\mathcal{O}}
\def \Y {\mathcal{Y}}

\def \R {\mathbb{R}}

\def \w {\mathbf{w}}
\def \v {\mathbf{v}}

\def \x {\mathbf{x}}

\def \E {\mathrm{E}}

\def \x {\mathbf{x}}

\def \e {\mathbf{e}}

\def \1 {\mathbf{1}}

\def \z {\mathbf{z}}
\def \s {\mathbf{s}}

\def \y {\mathbf{y}}
\def \u {\mathbf{u}}

\def \y {\mathbf{y}}
\def \E {\mathrm{E}}
\def \x {\mathbf{x}}

\def \z {\mathbf{z}}
\def \u {\mathbf{u}}

\def \w {\mathbf{w}}

\def \R {\mathbb{R}}
\def \S {\mathcal{S}}

\def \v {\mathbf{v}}

\def \h {\mathbf{h}}

\def \s {\mathbf{s}}

\def \teta {\tilde{\eta}} 

\def \State {\STATE}
\newtheorem{assumption}{Assumption} 

\newlength\myindent
\jmlrheading{1}{2024}{1-56}{4/00}{10/00}{meila00a}{Zhishuai Guo,  Yi Xu, Wotao Yin, Rong Jin, Tianbao Yang} 

\ShortHeadings{A Novel Convergence Analysis for Algorithms of 
 the Adam Family and Beyond}{Guo, Xu, Yin, Jin, Yang} 
\firstpageno{1}

\begin{document}

\title{Unified Convergence Analysis for Adaptive Optimization with Moving Average Estimator}

        
\author{\name{Zhishuai Guo$^*$}\email{zhishuai-guo@uiowa.edu}\\
\name{Yi Xu$^\dagger$}\email{yixu@alibaba-inc.com}\\
\name{Wotao Yin$^\ddagger$}\email{wotao.yin@alibaba-inc.com}\\
\name{Rong Jin$^\ddagger$}\email{jinrong.jr@alibaba-inc.com}\\
\name{Tianbao Yang$^*$}\email{tianbao-yang@uiowa.edu}\\
  \addr$^*$Texas A\&M University, College Station, TX 77840, USA\\
   \addr$^\dagger$Dalian University of Technology, Dalian, 116024, Liaoning, China \\ 
    \addr$^\ddagger$Alibaba Group, Bellevue, WA 98004, USA \\
}


\maketitle

\begin{center}
    V1: Apr 30th, 2021 \\
    This Version (V7): Apr 4th,  2025\footnote{V3 (Jan 13, 2022)  simplified the algorithms for min-max problem and bilevel problem. And for the min-max problem, V3 improved the dependence on condition number, i.e., from $O(\kappa^{4.5} \epsilon^3)$ to $O(\kappa^3/\epsilon^3)$. 
    V4 showed experimental results and extended to problems under PL condition for both minimization problems and min-max problems. The results for bilevel problems are presented in the appendix. Accepted to Machine Learning. 
    }  
\end{center}
\begin{abstract}
Although adaptive optimization algorithms have been successful in many applications, there are still some mysteries in terms of convergence analysis that have not been unraveled. 
This paper provides a novel non-convex analysis of adaptive optimization to uncover some of these mysteries. 
Our contributions are three-fold.
First, we show that an increasing or large enough momentum parameter for the first-order moment used in practice is sufficient to ensure the convergence of adaptive algorithms whose adaptive scaling factors of the step size are bounded. 
Second, our analysis gives insights for practical implementations, e.g., increasing the momentum parameter in a stage-wise manner in accordance with stagewise decreasing step size would help improve the convergence. 
Third, the modular nature of our analysis allows its extension to solving other optimization problems, e.g., compositional, min-max and bilevel problems.  
As an interesting yet non-trivial use case, we present algorithms for solving non-convex min-max optimization and bilevel optimization that do not require using large batches of data to estimate gradients or double loops as the literature do.
Our empirical studies corroborate our theoretical results.
\end{abstract}


\section{Introduction}\label{sec1} 
Since its invention in 2014, the Adam optimizer~\citep{kingma2014adam} has received tremendous attention and has been widely used in practice for training deep neural networks.  Many variants of Adam were proposed for improving its performance, e.g.~\citep{DBLP:conf/nips/ZaheerRSKK18,luo2018adaptive,Liu2020On}.
Its analysis for non-convex optimization has also received a lot of attention~\citep{DBLP:conf/iclr/ChenLSH19}. 
For more generality, we consider a family of adaptive algorithms. The update for minimizing$F: \R^n\rightarrow \R$ is given by:
 \begin{equation}
\begin{aligned}\label{eqn:adam}
&\v_{t+1} = \beta_t\v_t + (1-\beta_t) \O_F(\x_t),\\
&\x_{t+1}  = \x_t - \eta\s_t \circ \v_{t+1}, t=0, \ldots, T,
\end{aligned}
\end{equation}
where $\x_t$ denotes the model parameter and $\O_F(\x_t)$ denotes an unbiased stochastic gradient estimator, $\beta_t$ is known as the momentum parameter of the first-order moment, $s_t$ denotes an appropriate possibly coordinate-wise adaptive step size scaling factor and $\eta$ is the standard learning rate parameter.  

One criticism of Adam is that it might not converge for some problems with some momentum parameters following its original analysis. In particular, the authors of AMSGrad~\citep{reddi2019convergence} show that Adam with small momentum parameters can diverge for some problems. However, we notice that the failure of Adam shown in \citep{reddi2019convergence} and the practical success of Adam come from an inconsistent setting of the momentum parameter for the first-order moment. In practice, this momentum parameter (i.e., $\beta_t$ in (\ref{eqn:adam})) is usually set to a large value (e.g., 0.9) close to its limit value $1$.  However, in the failure case analysis of Adam~\citep{reddi2019convergence,kingma2014adam} and existing (unsuccessful) analysis of Adam~\citep{DBLP:conf/iclr/ChenLSH19} and its variants~\citep{luo2018adaptive,DBLP:conf/nips/ZaheerRSKK18,naichenrmsprop,savarese2019convergence}, such momentum parameter is set as a small value or a decreasing sequence. 

To address the gap between theory and practice of adaptive algorithms, we provide a generic convergence analysis with an increasing or large momentum parameter for the first-order moment and \textbf{a bounded} second order coordinate step size $s_t$.Our analysis is simple and intuitive. It only requires a bounded assumption on the gradient which usually holds or can be easily enforced in adaptive algorithms. By employing an increasing/large momentum parameter, we place more emphasis on historical gradients during the later stages of optimization, ensuring a stable convergence process. At the early iterations the information in the current stochastic gradient is more valuable, hence using a relatively small $\beta$ is helpful for improving the convergence speed. 
As the algorithm approaches the stationary point, the differences between successive models decrease, making older gradients more relevant for current updates. Either utilizing past stochastic gradients and or relying solely on current stochastic gradients introduce inaccuracies. Given the assumption of bounded variance (e.g., Assumption \ref{ass:1}: $\E[\|\O_F(\x)- \nabla F(\x)\|^2]\leq \sigma^2(1 + c\|\nabla F(\x)\|^2)$), the error in the current stochastic gradient remains at a constant level (i.e., $\sigma^2$). This is why we assign less weight to the current stochastic gradients and assign more weight to the momentum term, which becomes a progressively more accurate estimator as the algorithm progresses.
Our analysis covers a family of Adam family algorithms such as AMSGrad, Adabound, AdaFom, etc. A surprising yet natural result is that when adaptive scaling factors of the step size are bounded, the Adam family with an increasing or large enough momentum parameter for the first-order moment indeed converges at the same rate as SGD.
In our analysis, $1 - \beta_t$ is decreasing or as small as the step size in the order, that is, an increasing or sufficiently large momentum parameter $\beta_t$. This increasing (or large) momentum parameter is more natural than the decreasing (or small) momentum parameter, which is indeed the reason that makes Adam diverge on some examples~\citep{reddi2019convergence}. The increasing/large momentum parameter $\beta_t$ is also consistent with the large value close to 1 used in practice and suggested in~\citep{kingma2014adam} for practical purposes.

A key in the analysis is to carefully leverage the design of stochastic estimator of the gradient, i.e., $\v_{t+1}$. Simply using unbiased stochastic gradients for model updates tends to result in convergence issues in various problems, particularly in (non-convex) min-max scenarios \citep{chavdarova2019reducing,hsieh2020explore}. One would usually need to employ impractically large mini-batch sizes to mitigate this in non-convex min-max problems \citep{lin2020gradient,boct2023alternating}. This phenomenon largely stems from their enhanced sensitivity to the variance of the unbiased stochastic gradients, as opposed to more straightforward minimization problems. Recent studies in stochastic non-convex optimization have proposed better stochastic estimators of the gradient based on variance reduction technique (e.g., SPIDER, SARAH, STORM)~\citep{fang2018spider,wang2019spiderboost,pham2020proxsarah,cutkosky2019momentum}. However, these estimators sacrifice generality as they require that the unbiased stochastic oracle is Lipschitz continuous with respect to the input, which prohibits many useful tricks in machine learning for improving {generalization} and {efficiency} (e.g., adding random noise to the stochastic gradient~\citep{DBLP:journals/corr/NeelakantanVLSK15},  gradient compression~\citep{NIPS2017_6c340f25,pmlr-v70-zhang17e,10.5555/3326943.3327063}). In addition, they also require computing stochastic gradients at two points per-iteration, making them further restrictive.  Instead, we directly analyze the stochastic estimator based on moving average (SEMA), i.e., the first equation in~(\ref{eqn:adam}).  We prove that for non-convex problems, averaged variance of the stochastic estimator $\mathbf{v}_{t+1}$ decreases over time, i.e., $\mathbb{E}\left[\frac{1}{T+1}\sum_{t=0}^T\|\mathbf{v}_t - F(\mathbf{x}_t)\|^2\right] \leq O(1/\sqrt{T})$. Thus it provides gradient estimators accurate enough to ensure that in expectation, the uniformly sampled output of adaptive algorithms with moving average estimator converge to a stationary point.   

This variance-reduction property of adaptive algorithms with moving average estimator is also helpful for us to design new algorithms and improve analysis for other non-convex optimization problems, e.g., compositional optimization, non-convex min-max optimization, and non-convex bilevel optimization. As an interesting yet non-trivial use case, we consider  non-convex strongly-concave  min-max optimization (or concave but satisfying a dual-side Polyak-Lojasiewicz [PL] condition).  We propose primal-dual stochastic momentum and adaptive methods based on the SEMA estimator {\it without requiring  a large mini-batch size and a Lipschitz continuous stochastic oracle}, and establish the state-of-the-art complexity, i.e., $O(\kappa^3/\epsilon^4)$ for finding an $\epsilon$-stationary solution, where $\kappa$ is a condition number.  This work establishes the convergence of primal-dual stochastic momentum and adaptive methods  with various kinds of adaptive step sizes for updating the primal variable.  
In addition, our result addresses a gap in the literature of non-convex strongly-concave min-max optimization~\citep{lin2020gradient,yan2020sharp}, which either requires a large mini-batch size or a double-loop for achieving the same complexity.

We present a comparison of our theoretical results with existing literature in Table~\ref{tab:0}. 

\begin{table*}[t] 
	\caption{Comparison with previous results. ``mom. para." is short for momentum parameter. ``-" denotes no strict requirements and applicable to a range of updates. $\uparrow$ represents increasing as iterations and $\downarrow$ represents decreasing as iterations. $\epsilon$ denotes the target accuracy level for the objective gradient norm, i.e., $\E[\|\nabla F(\x)\|^2]\leq \epsilon^2$.  
}\label{tab:0} 
	\centering
	\label{tab:1} 
	\scalebox{0.59}{\begin{tabular}{lllccc} 
			\toprule 
			\multirow{2}{*}{Problem} & \multirow{2}{*}{Method}	& \multirow{2}{*}{batch size} 
   & $\uparrow$ or $\downarrow$ & $\uparrow$ or $\downarrow$ &  \multirow{2}{*}{Converge?} \\
   &&& 1st mom. para. & 2nd mom. para. \\ 
			\hline  
			&This work&$O(1)$&$\uparrow$&-&Yes \\
			&\citep{kingma2014adam}&$O(1)$&$\downarrow$&constant&No\\
	Non-convex		&\citep{DBLP:conf/iclr/ChenLSH19}&$O(1)$&Non-$\uparrow$&-&No\\
	(Adam-family)	&\citep{DBLP:conf/nips/ZaheerRSKK18}&$O(1/\epsilon^2)$&constant&$\uparrow$&Yes\\
			&\citep{DBLP:conf/cvpr/ZouSJZL19}&$O(1)$&constant&$\uparrow$&Yes\\
			&\citep{2020arXiv200302395D}&$O(1)$&constant&$\uparrow$&Yes\\
			\midrule
			\hline 
		Problem & Method	&batch size&\# of loops&\# of samples for (\ref{eqn:mm})&Oracle\\
		\hline 
			&This work&$O(1)$&Single&$O(1/\epsilon^4)$&General \\
		Non-Convex	&SGDA~\citep{lin2020gradient}&$O(1/\epsilon^2)$&Single&$O(1/\epsilon^4)$&General\\
		Strongly-Concave	&SGDMax~\citep{lin2020gradient}&$O(1)$&Double&$O(1/\epsilon^8)$&General\\
		MinMax	&Epoch-SGDA~\citep{yan2020sharp}&$O(1)$&Double&$O(1/\epsilon^4)$&General\\
			&AccMDA~\citep{DBLP:journals/corr/abs-2008-08170}&$O(1)$&Single&$O(1/\epsilon^3)$&Lipschitz\\
		\bottomrule 
	\end{tabular}}
	\vspace*{-0.15in} 
\end{table*}

\section{Related Work}
We notice that the literature on stochastic non-convex optimization is huge and we cannot discuss all of them in this section.  We will focus on methods requiring only a general unbiased stochastic oracle model, before which we summarize some related lines of work and interested readers can refer to them and references therein. 
The optimization of non-convex minimization problems have extensively been studied in the literature, such as in \citep{nesterov2012make,ghadimi2013stochastic,ghadimi2016accelerated,allen2018make,davis2019stochastic}.
Convex-concave min-max optimization also has drawn tremendous attention  \citep{juditsky2011solving,nemirovski2009robust,yan2020sharp}.
Deterministic min-max optimization has been discussed in \citep{lin2020near,xu2020unified,zhang2021complexity,chen2022accelerated,bot2022fast}.
Readers who are interest in the setting of finite-sum min-max optimization can see \citep{zhang2021complexity,chen2022accelerated}.
In particular, there is a line of research on variational inequalities based analysis that can be applied to min-max problems \citep{juditsky2011solving,bot2022fast,chavdarova2019reducing,hsieh2020explore,zhang2022sapd+}.

\textbf{Stochastic Adaptive Methods.}
\label{sec:relatedwork}
Stochastic adaptive methods originating from AdaGrad for convex minimization~\citep{duchi2011adaptive,mcmahan2010adaptive} have attracted tremendous attention for stochastic non-convex optimization~\citep{ada18bottou,ada18orabona,adagradmom18,tieleman2012lecture,chen2018closing,luo2018adaptive,huang2021super}. 

Several recent works have tried to prove the (non)-convergence of Adam. In particular, \citet{DBLP:conf/cvpr/ZouSJZL19} establish some sufficient condition for ensuring Adam family to converge. In particular, they choose to increase the momentum parameter for the second-order moment and establish a convergence rate in the order of $O(\log(T)/\sqrt{T})$, which was similarly established in \cite{2020arXiv200302395D} with some improvement on the constant factor. \citet{DBLP:conf/nips/ZaheerRSKK18} show that Adam with a sufficiently large mini-batch size can converge to an accuracy level proportional to the inverse of the mini-batch size. \citet{DBLP:conf/iclr/ChenLSH19} analyze the convergence properties for a family of Adam family algorithms. However, their analysis requires a strong assumption of the updates to ensure the convergence, which does not necessarily hold as the authors give non-convergence examples.  Different from these works, we give an alternative way to ensure Adam \textbf{converges} by  using an increasing or large momentum parameter for the first-order moment without requiring large mini-batch.
Moreover, we do not enforce any requirement on the second order momentum parameter $\beta'$ other than $\beta'\in [0,1]$, in contrast with some existing works enforcing $\beta < \sqrt{\beta'} <1$ \citep{DBLP:conf/cvpr/ZouSJZL19,chen2021cada}, $\beta < \beta' <1$ \citep{2020arXiv200302395D}, or $\beta'$ is extremely close to 1 \citep{zhang2022adam}, where $\beta$ and $\beta'$ are first-order momentum and second-order momentum parameters, respectively.

Indeed, our analysis is applicable to a family of adaptive algorithms, and is agnostic to the method for updating the normalization factor in the adaptive step size as long as it can be upper bounded.  The large momentum parameter for the first-order moment is also the key part that differentiates our convergence analysis with existing non-convergence analysis of Adam~\citep{DBLP:conf/iclr/ChenLSH19,reddi2019convergence}, which require  the  momentum parameter for the first-order moment to be decreasing  to zero or sufficiently small. After our manuscript appeared on arXiv, there is another work that gives a convergence analysis using large first order momentum \citep{zhang2022adam}. However, their analysis cannot guarantee the convergence to an $\epsilon$-stationary point unless the second-order momentum is extremely close to 1. 

\textbf{Stochastic Non-Convex Min-Max Problems.} Stochastic non-convex concave min-max optimization has been studied in several recent works.  \citet{rafique2018non} establishes the first results for these problems. In particular, their algorithms suffer from an oracle complexity of $O(1/\epsilon^6)$ for finding a nearly $\epsilon$ stationary point of the primal objective function, and an oracle complexity of $O(1/\epsilon^4)$ when the objective function is strongly concave in terms of the dual variable and has a certain special structure. The same order oracle  complexity  of $O(1/\epsilon^4)$ is achieved in~\citep{yan2020sharp} for weakly-convex strongly-concave problems without a special structure of the objective function. However, these algorithms use a double-loop based on the proximal point method.  \citet{lin2020gradient} analyzes a single-loop stochastic gradient descent ascent  (SGDA) method for smooth non-convex concave min-max problems. They have established  the same order of oracle complexity  $O(1/\epsilon^4)$ for smooth non-convex strongly-concave problems but with a large mini-batch size. Recently, \citet{2020arXiv200713605I} extends the analysis to stochastic alternating (proximal) gradient descent ascent method but suffering from the same issue of requiring a large mini-batch size. 
In contrast, our methods enjoy the same order of oracle complexity without using a large mini-batch size. 
\citet{qiu2020single} have also obtained a complexity of $O(1/\epsilon^4)$ by utilizing the moving average estimator but as we will see later they suffer from a much worse dependency on the condition number. 
We note that an improved complexity of $O(1/\epsilon^3)$ was achieved in several recent works under the Lipschitz continuous oracle model~\citep{luo2020stochastic,DBLP:journals/corr/abs-2008-08170,DBLP:conf/nips/Tran-DinhLN20}, which is non-comparable to our work that only requires a general unbiased stochastic oracle.  Recently, several studies~\citep{nouiehed2019solving,liu2018fast,yang2020global,guo2020fast} propose stochastic algorithms for non-convex min-max problems by leveraging stronger conditions of the problem (e.g., PL condition). 
After the first appearance of our work, \cite{yang2022faster,zhang2022sapd+} have achieved better dependence on condition number than ours for the non-convex-strongly-concave problem, but they either use a two-staged algorithm or a double-loop algorithm, while our analysis allows a single-loop algorithm. More details will be discussed later.
\citet{yan2022adaptive} has studied adaptive min-max problem under convex setting. \citep{dou2021one} has proposed and analyzed an extra gradient AMSGrad method for non-convex-non-concave optimization.  
However, the analysis of primal-dual stochastic momentum and primal-dual adaptive methods for solving non-convex-strongly-concave min-max optimization problems remain rare, which is presented in this work.

Finally, we also note that the $O(1/\epsilon^4)$ oracle complexity is optimal for stochastic non-convex optimization under a general unbaised stochastic oracle model, according to Theorem 1 of ~\citep{arjevani2019lower}, which implies that our results are optimal up to a logarithmic factor.

\section{Notations and Preliminaries}
{\bf Notations and Definitions.} Let $\|\cdot\|$ denote the Euclidean norm of a vector or the spectral norm of a matrix.   Let $\|\cdot\|_F$ denote the Frobenius norm of a  matrix. A mapping $h$ is $L$-Lipschitz continuous iff $\|h(\x) - h(\x')\|\leq L\|\x - \x'\|$ for any $\x, \x' \in \R^d$.  A function $F$ is called $L$-smooth if its gradient $\nabla F(\cdot)$ is $L$-Lipschitz continuous. A function $g$ is $\lambda$-strongly convex iff $g(\x)\geq g(\x') + \nabla g(\x')^{\top}(\x  - \x') + \frac{\lambda}{2}\| \x - \x'\|^2$ for any $\x, \x'$.   A function $g(\y)$ is called $\lambda$-strongly concave if $-g(\y)$ is $\lambda$-strongly convex.  For a differentiable function $f(\x, \y)$, we let $\nabla_x f(\x, \y)$ and $\nabla_y f(\x, \y)$ denote the partial gradients with respect to $\x$ and $\y$, respectively. Denote by $\nabla f(\x, \y) = (\nabla_x f(\x, \y)^{\top}, \nabla_y f(\x, \y)^{\top})^{\top}$. 

In the following sections, we will focus on two families of non-convex optimization problems, namely  non-convex minimization~(\ref{eqn:snc}),  non-convex  min-max optimization problem~(\ref{eqn:mm}). 
These optimization problems have broad applications in machine learning. {\bf This paper focuses on theoretical analysis and our goal} for these problems is to find an $\epsilon$-stationary solution of the primal objective function $F(\x)$ by using stochastic oracles.
\begin{definition}
For a differentiable function $F(\x)$,  a randomized solution $\x$ is called an $\epsilon$-stationary point if it satisfies $\E[\|\nabla F(\x)\|^2] \leq \epsilon^2$. 
\end{definition} 
For deriving faster rates, we also consider the PL condition.
\begin{definition}
$F(\x)$ is said to satisfy $\mu$-PL condition for some constant $\mu>0$ if it holds that $\|\nabla F(\x)\|^2 \geq 2\mu (F(\x) - \min_{\x'} F(\x'))$.  
\end{definition}  

For non-convex min-max problems, our analysis covers two cases: non-convex strongly concave min-max optimization~\citep{rafique2018non}, and non-convex non-concave optimization with dual-side PL condition~\citep{yang2020global}. The dual-side PL condition is given below. 
\begin{definition}
$F(\x, \y)$ satisfies the dual-side $\mu$-PL condition, i.e., $\forall\x$, $\|\nabla_y F(\x, \y)\|^2 \geq 2\mu (\max_{\y}F(\x, \y)-F(\x, \y))$.  
\end{definition}


Depending on the problem's structure, we require different stochastic oracles that will be described for each problem later. 

Before ending this section, we present a closely related stochastic momentum method for solving non-convex minimization problem $\min_{\x\in\R^d} F(\x)$ through an unbiased stochastic oracle that returns a random variable $\O_F(\x)$ for any $\x$ such that $\E[\O_F(\x)] = \nabla F(\x)$. For solving this problem, the stochastic momentum method (in particular stochastic heavy-ball (SHB) method) that employs the SEMA update is given by 
\begin{equation}\label{eqn:sma}
 \begin{split}
&\v_{t+1} = (1-\gamma)\v_t + \gamma\O_F(\x_t)\\
&\x_{t+1} = \x_t - \eta \v_{t+1}, \quad t= 0, \ldots, T.
\end{split} 
\end{equation}
where $\v_0 = \O_F(\x_0)$. In the literature,  $\beta=1 - \gamma$ is known as the momentum parameter and $\eta$ is known as the step size or learning rate. Note that the stochastic momentum method can be also written as $\z_{t+1} = \gamma\z_t - \eta\O_F(\x_t)$, and $\x_{t+1} = \x_t + \z_{t+1}$~\citep{yangnonconvexmo}, which is equivalent to the above update with some parameter change shown in Appendix~\ref{appendix:SMM}. The above method has been analyzed in various studies~\citep{ghadimi2020single, NEURIPS2020_d3f5d4de,yu_linear,yangnonconvexmo}. Nevertheless, we will give a unified analysis for the adaptive methods with moving average estimator by a much more concise proof, which covers SHB as a special case. A core to the analysis is the use of a known variance recursion property of the SEMA estimator stated below. 
\begin{lemma}({\bf Variance Recursion of SEMA})[Lemma 2, \citep{wang2017stochastic}]
Consider a moving average sequence $\z_{t+1} = (1-\gamma_t)\z_t + \gamma_t \O_h(\x_t)$ for tracking $h(\x_t)$, where $\E_t[\O_h(\x_t)] = h(\x_t)$ and $h$ is a $L$-Lipschitz continuous mapping. Then we  have
\begin{equation}
\begin{split}
&\E_t[\|\z_{t+1} - h(\x_t)\|^2]\leq (1-\gamma_t)\|\z_t - h(\x_{t-1})\|^2 \\ 
&+2\gamma_t^2\E_t[\|\O_h(\x_t) - h(\x_t)\|^2]  + \frac{L^2\|\x_{t} - \x_{t-1}\|^2}{\gamma_t}, 
\end{split}
\end{equation}
where $\E_t$ denotes the expectation conditioned on all randomness before $\O_h(\x_t)$.  
\label{lem:sema_mengdi}
\end{lemma}
Let's explore in more detail why large or increasing momentum is required. By rearranging the terms of equation (\ref{lem:sema_mengdi}) of Lemma 1, we can derive the following expression:
\begin{equation}
\begin{split} 
\E_{t-1}[\|\z_t - h(\x_{t-1})\|^2 \leq& \frac{\E_{t-1}[\|\z_t - h(\x_{t-1})\|^2 - \E_t[\|\z_{t+1} - h(\x_t)\|^2]}{\gamma_t}
\\  
&+2\gamma_t\underbrace{\E_t[\|\O_h(\x_t) - h(\x_t)\|^2]}\limits_{(V)}  + \frac{L^2\|\x_{t} - \x_{t-1}\|^2}{\gamma_t^2}.  
\end{split} 
\label{eq:lemma1_re}
\end{equation} 
The first term on the right-hand side (RHS) of (\ref{eq:lemma1_re}) can be managed using the telescoping sum technique when $\gamma_t$ is constant across all $t$, and requires a slightly different telescoping sum when $\gamma_t$ varies with $t$. It is also important to note that the third term on the RHS, although positive, reduces as $\x_t$ approaches a stationary point. Specifically, suppose that $\x_t$ is updated by $\x_t = \x_{t-1} - \eta_t \v_t$, with $\v_t$ being a moving average estimator itself. In this case, with $\eta_t=O(\gamma_t)$ we have 
\begin{equation}
\begin{split}
&\frac{L^2\|\x_{t} - \x_{t-1}\|^2}{\gamma_t^2} = \frac{L^2\|\x_{t} - \x_{t-1}\|^2}{\gamma_t^2} = \frac{L^2\eta_t^2\|\v_{t}\|^2}{\gamma_t^2} \\
&= O(\|\v_t - \nabla F(\x_{t-1})\|^2 + \|\nabla F(\x_{t-1})\|^2), 
\end{split} 
\end{equation}
where both $\|\v_t - \nabla F(\x_{t-1})\|^2$ and $\|\v_t - \nabla F(\x_{t-1})\|^2$ are diminishing as the algorithm progresses. 
Next we turn to the second term on the RHS of (\ref{eq:lemma1_re}). Given the assumption of bounded variance (e.g., Assumption 1: $\E[\|\O_F(\x)- \nabla F(\x)\|^2]\leq \sigma^2(1 + c\|\nabla F(\x)\|^2)$), the error in the current stochastic gradient remains at least a constant level (i.e., $\sigma^2$). Then, we can see that to obtain an accurate estimator $\z$, i.e., $\E_t[\|\z_t - h(\x_{t-1})\|^2 \leq \epsilon^2$, $\gamma_t$ has to be at a small level of $\epsilon^2$ or decrease to $\epsilon^2$ over time to control the error introduced by the unbiaed stochastic estimator, i.e., the (V) term. This requirement underlines the rationale for using a large or increasing momentum $\beta_t$, where $\beta_t=1-\gamma_t$. 
We refer to the property in Lemma \ref{lem:sema_mengdi} as variance recursion (VR) of the SEMA.


\section{Novel Analysis of Adaptive Methods for Non-Convex Minimization}
\label{sec:adamin}
In this section, we consider the standard stochastic non-convex minimization, i.e., 
\begin{align}\label{eqn:snc}
\min_{\x\in\R^d} F(\x),
\end{align} 
where $F$ is smooth and is accessible only through an unbiased stochastic oracle. These conditions are summarized below for our presentation. 

\begin{assumption}\label{ass:1}Regarding problem~(\ref{eqn:snc}), the following conditions hold: 
\begin{itemize}
\item $\nabla F$ is $L_F$ Lipschitz continuous.
\item $F$  is accessible only through an unbiased stochastic oracle that returns a random variable $\O_F(\x)$ for any $\x$ such that $\E[\O_F(\x)] = \nabla F(\x)$, and $\O_F$ has a variance bounded by $\E[\|\O_F(\x)- \nabla F(\x)\|^2]\leq \sigma^2(1 + c\|\nabla F(\x)\|^2)$ for some $c>0$.
\item There exists $\x_0$ such that  $ F(\x_0) - F_*\leq \Delta_F$ where $F_* = \min\limits_{\x} F(\x)$ and $\Delta_F>0$.
\end{itemize} 
\end{assumption}
\vspace*{-0.1in}
{\bf Remark:} Note that the variance bounded condition is slightly weaker than the standard condition $\E[\|\O_F(\x)- \nabla F(\x)\|^2]\leq \sigma^2$. An example of a random oracle that satisfies our condition but not the standard condition is $\O_F(\x) =d\cdot \nabla F(\x)\circ\e_i$, where $i\in\{1,\ldots, d\}$ is randomly sampled and  $\e_i$ denotes the $i$-th canonical vector with only $i$-th element equal to one and others zero.  For this oracle, we can see that $\E[\O_F(\x)] = \nabla F(\x)$ and $\E[\|\O_F(\x) - \nabla F(\x)\|^2]\leq (d-1)\|\nabla F(\x)\|^2$. 

\setlength{\textfloatsep}{3pt}

\begin{table*}[t] 
	\caption{Different adaptive methods and their satisfactions of Assumption~\ref{ass:2}}
	\label{tab:2} 
	\centering
	\scalebox{0.7}{\begin{tabular}{lccc}
			\toprule
		method&update for $h_t$&Additional assumption&$c_l$ and $c_u$ \\
		\midrule
				SHB &$\u_{t+1} = 1,G_0=0$ & - & $c_l =1, c_u =1$ \\
\midrule
		Adam& $\u_{t+1} = \beta_t'\u_t + (1-\beta_t') \O_F^2(\x_t)$&$\|\O_F\|\infty\leq G$&$c_l = \frac{1}{G+G_0}, c_u = \frac{1}{G_0}$\\
		\midrule
		\multirow{2}{*}{AMSGrad}& $\u'_{t+1} =\beta_t' \u'_t + (1-\beta_t') \O_F^2(\x_t)$ &\multirow{2}{*}{$\|\O_F\|\infty\leq G$} & \multirow{2}{*}{$c_l = \frac{1}{G+G_0}, c_u = \frac{1}{G_0}$}\\
  & $\u_{t+1} = \max(\u_t, \u'_{t+1})$ \\
		\midrule
		AdaFom& \multirow{2}{*}{$\u_{t+1} = \frac{1}{t+1}\sum_{i=0}^t\O^2_F(\x_i)$} &\multirow{2}{*}{$\|\O_F\|\infty\leq G$}&\multirow{2}{*}{$c_l = \frac{1}{G+G_0}, c_u = \frac{1}{G_0}$}\\ 
        (AdaGrad) \\ 
		\midrule
        Adam$^+$& $\u_{t+1} = \|\v_{t+1}\|$& $\|\O_F\|\leq G$ & $c_l = \frac{1}{\sqrt{G}+G_0}, c_u = \frac{1}{G_0}$\\
        \midrule
        \multirow{2}{*}{AdaBound}& $\u'_{t+1} =\beta_t' \u'_t + (1-\beta_t') \O_F^2(\x_t)$& \multirow{2}{*}{-} & \multirow{2}{*}{ $c_l = c_l, c_u =c_u$}\\
        & $\u_{t} = \Pi_{[1/c_u^2, 1/c_l^2]}[\u'_{t+1}],\quad G_0 =0$ \\
		\bottomrule
	\end{tabular}}
\end{table*}

We will analyze a family of adaptive algorithms, whose updates are shown in Algorithm~\ref{alg:adam}. 
A key to our convergence analysis of adaptive methods is the boundness of the  step size scaling factor $\s_t = 1/(\sqrt{\u_{t+1}} + G_0)$, where $G_0$ is a constant to increase stability. We present the boundness of $\s_t$ as an assumption below for more generality. We denote by $\teta_t = \eta_t \s_t$. 
 
 \begin{algorithm}[t]
	\caption{
	Adaptive Stochastic Algorithms (ASA)}\label{alg:adam}
	\begin{algorithmic}[1]
\State {Input: $\x_0, \v_0, \eta_0, \beta_0, T$} 
\FOR {$t=0, 1, ..., T$}
\State ~~~~$\v_{t+1} = \beta_t\v_t +  (1-\beta_t)\O_F(\x_t)$
\State ~~~~{$\u_{t+1} = h_t (\O_F(\x_0), \ldots, \O_F(\x_t))$} \hfill $\diamond h_t\textcolor{teal}{\geq0}$ can be implemented as in Table~\ref{tab:2} 
\State ~~~~$\x_{t+1} = \x_{t} - \eta_t\frac{\v_{t+1}}{\sqrt{\u_{t+1}} + G_0}$
\ENDFOR
\State {\textbf{return} $\x_\tau, \v_{\tau}$ where $\tau$ is uniformly sampled from $0, ..., T$.} 
\end{algorithmic}
\end{algorithm}
 \begin{assumption}\label{ass:2}
 For the adaptive algorithms as shown in Algorithm~\ref{alg:adam}, we assume that $\s_t =  1/(\sqrt{\u_{t+1}} + G_0)$ is upper bounded and lower bounded, i.e., there exists $0<c_l<c_u$ such that $\forall i, t$, $c_l\leq \|\s_{t,i}\|\leq c_u$, where $\s_{t,i}$ denotes the $i$-th element of $\s_t$. 
 \end{assumption}
 {\bf Remark:} Under the standard assumption $\|\O_F(\x)\|_\infty\leq G$~\citep{kingma2014adam,reddi2019convergence}, we can see many adaptive algorithms can satisfy the above condition.  Examples include Adam~\citep{kingma2014adam}, AMSGrad~\citep{reddi2019convergence}, AdaFom~\citep{DBLP:conf/iclr/ChenLSH19}, Adam$^+$~\citep{DBLP:journals/corr/abs-2011-11985},  whose $\u_t$ shown in Table~\ref{tab:2} all satisfy the above condition under the bounded stochastic oracle assumption. 
Even if the condition $\|\O_F(\x)\|_\infty\leq G$ is not satisfied, we can use clipping to make $\u_t$ bounded. This is used in AdaBound~\citep{luo2018adaptive}, whose $\u_t$ is given by 
\begin{equation} 
\begin{split}
\text{AdaBound: }& \u'_{t+1} = \beta_t'\u'_t +  (1-\beta_t') \O_F^2(\x_t),\\
&\u_{t} = \Pi_{[1/c_u^2, 1/c_l^2]}[\u_{t+1}],\quad G_0 =0
\end{split}
\label{eq:adabound_update}
\end{equation}
where $c_l\leq c_u$ and $\Pi_{[a, b]}$ is a projection operator that projects each element of the input into the range $[a, b]$. 
It's important to note that (\ref{eq:adabound_update}) aligns with Assumption \ref{ass:2}, and therefore, fits within the scope of our analysis.  
We summarize various updates and their satisfactions of Assumption~\ref{ass:2} in Table~\ref{tab:2}. 
It is notable that the convergence analysis of AdaBound in \citep{luo2018adaptive} has some issues. As pointed out by \citep{savarese2019convergence}, \citep{luo2018adaptive} actually needs the step size for each coordinate to be non-increasing, which usually does not hold in adaptive algorithms.
\citep{luo2018adaptive} gives a convergence analysis for AdaBound but still based on a restrictive condition, i.e., $\frac{t}{\eta_l(t)} - \frac{t-1}{\eta_u(t)} \leq M$ where $M>0$ is a constant and $\eta_l(t), \eta_u(t)$ are the lower and upper bounds for step sizes at iteration $t$, respetively.
Note that SHB also satisfies Assumption \ref{ass:2} automatically.

To prove the convergence of the update~(\ref{eqn:adam}). We first present a key lemma. 
\begin{lemma}\label{lem:3}
Suppose Assumption \ref{ass:2} holds. For $\x_{t+1} = \x_t- \teta_t \circ \v_{t+1}$ with $\teta_t = \eta_t \circ \s_t$ and $\eta_t L_F\leq c_l/(2c_u^2)$, we have
\begin{equation} 
\begin{split}
F(\x_{t+1})  \leq& F(\x_t) +   \frac{ \eta_t c_u}{2}\|\nabla F(\x_t) - \v_{t+1}\|^2 - \frac{\eta_t c_l}{2}\|\nabla F(\x_t)\|^2  
- \frac{\eta_t c_l}{4}\|\v_{t+1}\|^2. 
\end{split}
\end{equation}
\end{lemma}

Based on the Lemma~\ref{lem:sema_mengdi} and Lemma \ref{lem:3}, we can easily establish the following convergence of adaptive methods.
\begin{theorem}\label{thm:2}
Let $\Delta_t = \|\v_{t+1} - \nabla F(\x_t)\|^2$. Suppose Assumptions~\ref{ass:1} and~\ref{ass:2} hold. With $1-\beta_t=\gamma=O(\frac{1}{\sqrt{T}})$,  $\eta_t=\eta = O( \frac{1}{\sqrt{T}})$, 
we have 
\begin{align*}
&\E\left[\frac{1}{T+1}\sum_{t=0}^T \|\nabla F(\x_t)\|^2\right]\leq O(\frac{1}{\sqrt{T}}), &\E\left[\frac{1}{T+1}\sum_{t=0}^T\Delta_t\right]\leq O(\frac{1}{\sqrt{T}}). 
\end{align*}
\end{theorem} 
{\bf Remark:} The second inequality above means the average variance of the SEMA sequence $\v_{t+1}$ is diminishing as  $T\rightarrow \infty$. We can see that the adaptive methods enjoy an oracle complexity of $T = O(1/\epsilon^4)$ for finding an $\epsilon$-stationary solution. 

One can also use a decreasing step size $\eta_t\propto 1/\sqrt{t}$ and increasing $\beta_t$ such that $1-\beta_t =1/\sqrt{t}$ (i.e, increasing momentum parameter) and establish a rate of $\widetilde O(1/\sqrt{T})$ as stated below.

\begin{theorem}\label{thm:SHB_decrease}
Suppose Assumption~\ref{ass:1} and \ref{ass:2} hold.   With 
$1-\beta_t = \gamma_t=\frac{c_l}{8\sigma^2 c c_u \sqrt{t+1}}$, 
and $\eta_t = \min\{\frac{\gamma_t \sqrt{c_l}}{2 L_F \sqrt{c_u^3}}, \frac{1}{2L_F c_u}\}$ 
we have  
\begin{align*}  
&\E\left[\frac{1}{T+1}\sum_{t=0}^T \|\nabla F(\x_t)\|^2\right]\leq O\left(\frac{1}{\sqrt{T}} +  \frac{\ln T}{\sqrt{T}} \right), \\
&\E\left[\frac{1}{T+1}\sum_{t=0}^T\Delta_t\right]\leq O\left(\frac{1}{\sqrt{T}} +  \frac{\ln T}{\sqrt{T}} \right).  
\end{align*}
\end{theorem}
\textbf{Remark. } It takes $\widetilde{O}(1/\epsilon^4)$ to converge to an $\epsilon$-stationary point.


\paragraph{An Improved Rate under PL condition.} 
\begin{algorithm}[t]
	\caption{Double Loop  Adaptive Algorithms}\label{alg:adam_PL}
	\begin{algorithmic}[1]
\State{Input: $\x_0, \v_0 = \O_F(\x_0), \eta_0, \beta_0, T_0$}   
\FOR{$k=0, 1, ..., K$} 
\State{~~~~$\x_{k+1}, \v_{k+1}$ = ASA($\x_k, \v_k, \eta_k, \beta_k, T_k$)}
\ENDFOR 
\State{\textbf{return} $\x_{K+1}$} 
\end{algorithmic}
\end{algorithm}

For problems satisfying PL condition, we develop a double loop algorithm (Algorithm \ref{alg:adam_PL}) where the step size is decayed exponentially after each stage, and establish an improved rate below. 
\begin{theorem}
\label{thm:min_PL} 
Suppose Assumption~\ref{ass:1} holds and $F(\x)$ satisfies  $\mu$-PL condition. Let $\Delta_t = \|\v_{t+1} - \nabla F(\x_t)\|^2$,  $\epsilon_0 = \max\{\Delta_F, \Delta_0\}$ and $\epsilon_k = \epsilon_0/2^k$.  
With $1-\beta_k = \gamma_k \leq \frac{\mu c_l \epsilon_k}{24 c_u \sigma^2}$, $\eta_k = \min\{\frac{\gamma_k \sqrt{c_l}}{2L_F \sqrt{c_u^3}}, \frac{1}{\sqrt{2}L_f c_u}\}$ and $T_k = \max\{\frac{48 c_u}{\mu\gamma_k c_l}, \frac{1}{6\mu\eta_k c_l}\}$, after $K=\log (\epsilon_0/\epsilon)$ stages, it holds that
\begin{equation}
\begin{split}
&\E[F(\x_{K+1}) - F_*] \leq \epsilon,\quad
\E[\Delta_K] \leq \epsilon.
\end{split}
\end{equation}
\end{theorem}
\textbf{Remark.} The total number of iterations is $\widetilde{O}(\frac{1}{\mu^2 \epsilon})$, which matches the state-of-the-art complexity for PL problems. \cite{chen2021cada} also considers leveraging PL condition in Adam type algorithm in distributed setting, but they require extra computation and storage to enhance the variance reduction.

\section{Adaptive Algorithms for Non-Convex Strongly-Concave Min-Max Optimization} 
In this section, we consider stochastic non-convex  min-max optimization: 
\begin{align}\label{eqn:mm}
\min_{\x\in\R^d} \max_{\y\in \Y}f(\x, \y),
\end{align}
and define
\begin{equation}
F(x) :=  \max_{\y\in \Y}f(\x, \y). 
\end{equation} 
We make the following assumption regarding this problem. 
\begin{assumption}\label{ass:3}Regarding the problem~(\ref{eqn:mm}), the following conditions hold: 
\begin{itemize}[leftmargin=*]
\item $F$ is $L_F$-smooth, $\nabla f(\x, \y)$ is $L_f$-Lipschitz continuous. 
\item   $f$  is accessible only through an unbiased stochastic oracle that returns a random tuple $(\O_{f,x}(\x, \y), \O_{f,y}(\x, \y))$ for any $(\x, \y)$ such that $\E[\O_{f,x}(\x, \y)] = \nabla_x f(\x, \y)$ and  $\E[\O_{f, y}(\x, \y)] = \nabla_y f(\x, \y)$, and $\O_{f,x/y}$ have  variance bounded by $\E[\|\O_{f,x}(\x, \y)- \nabla_x f(\x, \y)\|^2]\leq \sigma^2$ and $\E[\|\O_{f,y}(\x, \y)- \nabla_y f(\x, \y)\|^2]\leq \sigma^2$.
\item $\Y\subseteq\R^{d'}$ is a bounded or unbounded convex set and $f(\x, \cdot)$ is $\lambda$-strongly  concave for any $\x$.  Or $\Y=\R^{d
'}$ and $f(\x, \cdot)$ is concave and satisfies the dual-side $\lambda$-PL condition. 
\item There exists $\x_0$ such that  $ F(\x_0) - \min\limits_{\x} F(\x)\leq \Delta_F$. 
\end{itemize}
\end{assumption}
\vspace{-0.4cm}For solving the above problem, we propose primal-dual stochastic  momentum (PDSM) and adaptive (PDAda) methods and present them in a unified framework in Algorithm~\ref{alg:pdadam}, where $h_t$ can be implemented by the updates in Table~\ref{tab:2} and $h_t=0, G_0=1$ for PDSM. Note that  we only use the adaptive updates for updating $\x_{t+1}$ but not $\y_{t+1}$. This makes sense for many machine learning applications (e.g., AUC maximization~\citep{liu2018fast}, distributionally robust optimization~\citep{rafique2018non}), where the dual variable does not involve similar gradient issues as the primal variable (e.g. different gradient magnitude for different coordinates) to enjoy the benefit of adaptive step size.  

\begin{algorithm}[t]
	\caption{Primal-Dual Stochastic Momentum (PDSM) method and Adaptive (PDAda) method}\label{alg:pdadam}
	\begin{algorithmic}[1]
\State{Input: $\x_0\in \R^d, \y_0 \in\R^{d'}, \u_0, \v_0$}   
\FOR{$t=0, 1, ..., T$}
\State {~~~~$\v_{t+1} = \beta\v_t  +  (1-\beta)\O_x(\x_{t}, \y_{t}) $}
\State {~~~~$\u_{t+1} = h_t(\{\O_x(\x_{j}, \y_{j}), j=0, \ldots, t\})$} \hfill $\diamond h_t \textcolor{teal}{\geq0}$ as in Table~\ref{tab:2}
\State~~~~$\x_{t+1} = \x_{t} - \eta_x\frac{\v_{t+1}}{\sqrt{\u_{t+1}} + G_0}$
\State ~~~~{$\y_{t+1} = \Pi_{\Y}[\y_t + \eta_y \O_y(\x_t, \y_t)]$}
\ENDFOR 
\State{\textbf{return} $\x_\tau, \y_{\tau}$ where $\tau$ is uniformly sampled from $0, ..., T$}
\end{algorithmic}
\end{algorithm}

For understanding the algorithm, let us consider the primal-dual stochastic momentum (PDSM) method, i.e.,  Algorithm~\ref{alg:pdadam} with $h_t(\cdot)=1, G_0=0$ whose updates are given by 
\vspace*{-0.15in}\begin{equation}\label{eqn:sgdama}
\text{PDSM:}\quad\left\{
\begin{aligned}
& \v_{t+1}  =\beta_x  \v_t + (1-\beta_x)\O_{f,x}(\x_t, \y_t),\\
&\x_{t+1} = \x_t - \eta_x \v_{t+1},  \\
&\y_{t+1} = \Pi_{\Y}[\y_t + \eta_y \O_y(\x_t, \y_t)].
\end{aligned}\right. 
\end{equation}
Hence, the difference from the standard SGDA \citep{lin2020gradient} is that we use the SEMA to track the gradient in terms of $\x$, i.e., $\u_{t+1}$. The dual variable is updated in the same way by stochastic gradient ascent.   

We can see that PDSM/PDAda is a single-loop algorithm which only requires an $O(1)$ batch size at each iteration. In contrast, (i)  SGDMax~\citep{lin2020gradient} and the proximal-point based  methods proposed in~\citep{rafique2018non,yan2020sharp} are double-loop algorithms requiring solving a subproblem at each iteration to a certain accuracy level; (ii) SGDMax and SGDA~\citep{lin2020gradient} require a large mini-batch size in the order of $O(1/\epsilon^2)$. 

 Denote by $\Delta_{x,t} = \|\v_{t+1}- \nabla_x f(\x_t, \y_t)\|^2$, and $\delta_{y,t}=\|\y_t - \y^*(\x_t)\|^2$, where $\y^*(\x) = \arg\max\limits_{\y' \in \mathcal{Y}} f(\x, \y')$.  The convergence of PDSM/PDAda is presented below.
\begin{theorem}\label{thm:PDAda}
Suppose  Assumption \ref{ass:2} and Assumption~\ref{ass:3} hold. By setting $1-\beta=\gamma = O(\frac{\sqrt{\kappa}}{\sqrt{T}})$, $\eta_x = O(\frac{1}{\kappa^{3/2}\sqrt{T}})$, and $ \eta_y = O(\frac{\sqrt{\kappa}}{\sqrt{T}})$,  
\begin{align*}
&\E\left[\frac{1}{T+1}\sum_{t=0}^T \|\nabla F(\x_t)\|^2\right]\leq O\left(\frac{\kappa^{3/2}}{\sqrt{T}}\right), \\
&\E\left[\frac{1}{T+1}\sum_{t=0}^T(\Delta_{x,t} + L_f^2\delta_{y,t})\right]\leq O\left(\frac{\kappa^{3/2}}{\sqrt{T}}\right). 
\end{align*} 
\end{theorem}

{\bf Remark:} It is obvious to see that the sample complexity of PDSM and PDAda is $O(1/\epsilon^4)$, matching the state-of-the-art complexity for solving non-convex strongly-concave min-max problems. But it is also notable that our result above is applicable to non-convex concave min-max problem that satisfies the dual-side PL condition.

%

\textbf{Discussions.} Before ending this section, we provide more discussions on the dependence of complexity on the condition number, i.e.,  $\kappa = L_f/\lambda$. 
Since $\eta_y = O(\epsilon^2/\kappa)$ and $\eta_x= O(\epsilon^2/\kappa^{3})$, the sample complexity of PDSM/PDAda is $O(\kappa^{3}/\epsilon^4)$.  In contrast, SGDMax and SGDA using a large mini-batch size in the order of $\kappa^2/\epsilon^2$ have a sample complexity of $O(\kappa^3/\epsilon^4)$.  We notice that \citep{qiu2020single} which also utilize the moving average estimator has a worse dependency on condition number of $O(\kappa^6/\epsilon^4)$. The AccMDA algorithm with an $O(1)$ batch size presented in~\citep{DBLP:journals/corr/abs-2008-08170} requiring a Lipschitz continuous oracle has the dependence on the condition number of $O(\kappa^{4.5})$. The double-loop algorithms (e.g., Epoch-SGDA) with an $O(1)$ batch size presented in~\citep{rafique2018non,yan2020sharp} have an even worse dependence on $\kappa$ when applied to our considered problem. The convergence in \citep{yan2020sharp}, which originally considers problems that are weakly convex but not necessarily smooth in $\x$, guarantees a sample complexity of $O(\kappa^4/\epsilon^4)$ in order to find a solution $\x$ such that $\E[L_F^2\|\x- \x^*\|^2] \leq \epsilon^2$,  where $\x_* = \arg\min_{\z} F(\z) + L_F\|\z - \x\|^2$. Note that in the worst case $L_F = O(\kappa)$~\citep{lin2020gradient}. In order to transfer this convergence to that on $\nabla F(\x)$, we can use $\|\nabla F(\x)\|^2\leq 2L_F^2\|\x - \x^*\|^2 + 2\|\nabla F(\x^*)\|^2$. Hence the complexity of Epoch-SGDA for guaranteeing $\E\|\nabla F(\x)\|^2\leq \epsilon^2$ is $O(\kappa^4/\epsilon^4)$. 
\cite{yang2022faster} might demonstrate a more advantageous dependence on the condition number $\kappa$. However, their algorithm has a two-stage
structure. Initially, they address a smoothed sub-problem, then transition, incurring a translative cost, to the solutions for the primary problem. While such a two-stage structured method might not always be preferable in practical applications, their total cost of $O(\kappa^2/\epsilon^4 + \kappa^5/\epsilon^2)$ could be worse than ours when $\kappa>O(\frac{1}{\epsilon})$. More recently, \cite{zhang2022sapd+} has achieved a complexity of $O(\kappa/\epsilon^4)$, which utilizes a double loop algorithm. 
Neither of these algorithms, including ours, have matched an known lower bound of $\kappa^{1/3}/\epsilon^4$ in \citep{li2021complexity}, thus we highlight that it is still an open problem to find a tighter lower bound and/or develop algorithms that can match the lower bound.

Finally, we would like to point out we can also derive an improved rate for a min-max problem under an $\mu$-PL condition of $F(\x)$. But the analysis is a mostly straightforward extension, and hence we omit it. 

\subsection{Sketch of Analysis}
We first need to prove the following lemmas.
\begin{lemma}\label{lem:005_1} 
Suppose Assumption~\ref{ass:2} holds. Considering the PDAda update, with $\eta L_F\leq c_l/(2c_u^2)$ we have
\begin{align*}
&F(\x_{t+1})  \leq F(\x_t) +   \frac{\eta_x c_u}{2} \|\nabla_x f(\x_t, \y^*(\x_t)) - \v_{t+1}\|^2
- \frac{\eta_x c_l}{2}\|\nabla F(\x_t)\|^2  - \frac{\eta_x c_l}{4}\|\v_{t+1}\|^2. 
\end{align*} 
\end{lemma}
\vspace*{-0.1in}This resembles that of Lemma~\ref{lem:3}. Next, we establish a recursion for bounding  $\|\nabla_x f(\x_t, \y^*(\x_t)) - \u_{t+1}\|^2$.

\begin{lemma}
In Algorithm \ref{alg:pdadam}, it holds that
\label{lem:var_x_1}
\begin{equation*}
\begin{split}
& \E\|\v_{t+1} - \nabla_x f(\x_t, \y^*(\x_t))\|^2 \leq (1-\frac{\gamma}{2}) \E\|\v_t - \nabla_x f(\x_{t-1}, \y^*(\x_{t-1}))\|^2\\
&~~~~~~~~ +4\gamma^2 \sigma^2 + 10\gamma L_f^2\E\|\y_t - \y^*(\x_t)\|^2 + \frac{2L_F^2}{\gamma} \E\|\x_t - \x_{t-1}\|^2.  
\end{split}
\end{equation*}
\end{lemma}
Next, we establish a recursion for bounding  $\E\|\y_t - \y^*(\x_t)\|^2 $. 
\begin{lemma}\label{lem:03_1}
Suppose Assumption \ref{ass:3} holds. With $\y_{t+1} =\Pi_{\Y}[\y_t + \eta_y \O_{fy}(\x_t, \y_t)]$, $\eta_y \leq \lambda$ and $\kappa=L_f/\lambda$, we have
\begin{equation*}
\begin{split} 
\E\|\y_{t+1} - \y^*(\x_{t+1})\|^2 & \leq (1-\frac{\eta_y \lambda}{2}) \E\|\y_t - \y^*(\x_t)\|^2 + 2\eta_y^2 \sigma^2 + \frac{4\kappa^2}{\eta_y \lambda} \|\x_t - \x_{t+1}\|^2. 
\end{split}
\end{equation*}
\end{lemma}
By combining the above three lemmas, we can easily prove Theorem~\ref{thm:PDAda}.



\section{Experiments}
In this section, we show some experimental results to verify our theory. We consider both the minimization problem and the min-max problem. For the minimization problem, we use cross entropy (CE) as the loss function. For the min-max problem, we use the min-max formulated AUC loss \cite{liu2019stochastic} as the loss function. 
We conduct experiments on two types of data: 1) image data sets: Cifar10 and Cifar100 \citep{krizhevsky2009learning}, which has mutiple classes of images; 2) molecule data sets: BBBP and BACE \citep{wu2018moleculenet}, which involves a binary classification task to predict whether a molecule has a property or not, e.g., BBBP is short for blood-brain barrier penetration whose task is to predict whether a drug can penetrate the blood-brain barrier to arrive the targeted central nervous system or not. All experiments are conducted via Keras \citep{chollet2015keras} on Tensorflow framework \cite{tensorflow2015-whitepaper}.
\subsection{Minimization Problem}
In this subsection, we consider minimizing a standard CE loss. For the image data, i.e., Cifar10 and Cifar100, we use ResNet-50 as the network \cite{he2016deep}. For the molecule data, i.e., BBBP and BACE, we use a message-passing neural network (MPNN)  \citep{gilmer2017neural} implemented by the Keras team. 
We compare three optimization algorithms: the ``folklore'' Adam with first order momentum parameter $\beta$ fixed and $\eta$ fixed to a tuned value, the adaptive algorithm with an increasing first order momentum parameter $\beta$ and step size $\eta$ fixed to a tuned value, and the adaptive algorithm with both increasing first order momentum $\beta$ and decreasing step size $\eta$. The later two variants use the same second order moment as in Adam.  

For the folklore Adam, we fix the first order momentum to be $0.9$ as suggested in the \cite{kingma2014adam} and widely used in practice.
For the other two other variants, we tune initial $\gamma=1-\beta$ by tuning $\beta$  from $\{0.99, 0.9, 0.5\}$ and accordingly $\gamma$ from $\{0.01, 0.1, 0.5\}$. Then $\gamma$ (in the second variant)  is decayed by a factor of $0.2$ every 20 epochs with a total of 60 epochs. In the third variant, both $\gamma, \eta$ are decayed by a factor of $0.2$ every 20 epochs.  For all algorithms, the initial step size is tuned in $1e\text{-}5\sim 1e\text{-}2$ and $G_0$ is tuned in $\{1e\text{-}7, 1e\text{-}5, 1e\text{-}3\}$. In all experiments, we use a batch size of 32, and repeat the experiments 5 times and report the averaged results with standard deviation. 

\begin{figure}[htpb]
    \centering
    \subfigure[Results on Image Data]{
    \includegraphics[scale=0.23]{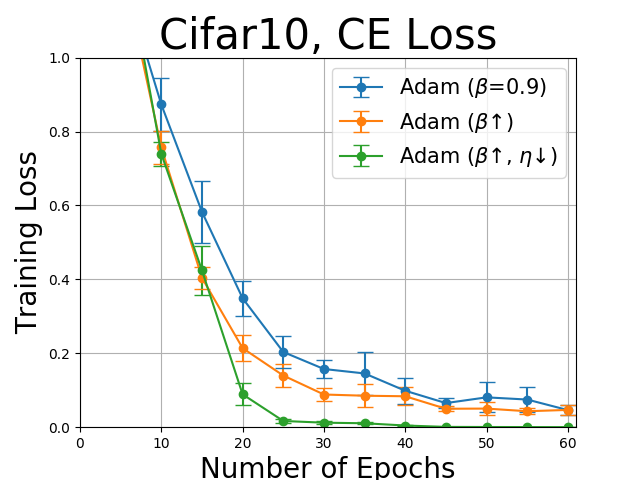}
    \includegraphics[scale=0.23]{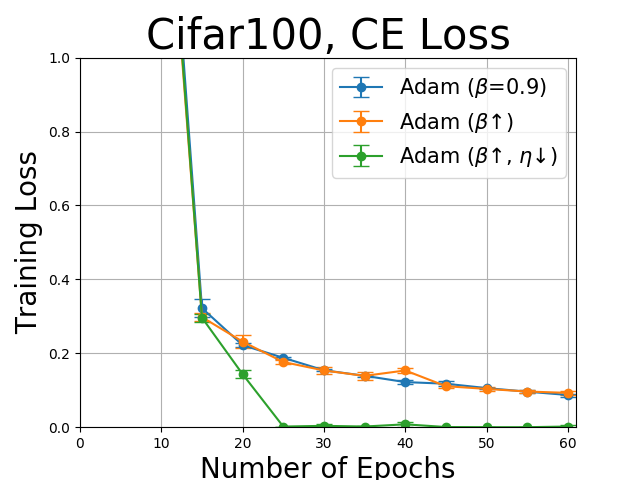}
    }
    \subfigure[Results on Molecule Data]{
    \includegraphics[scale=0.23]{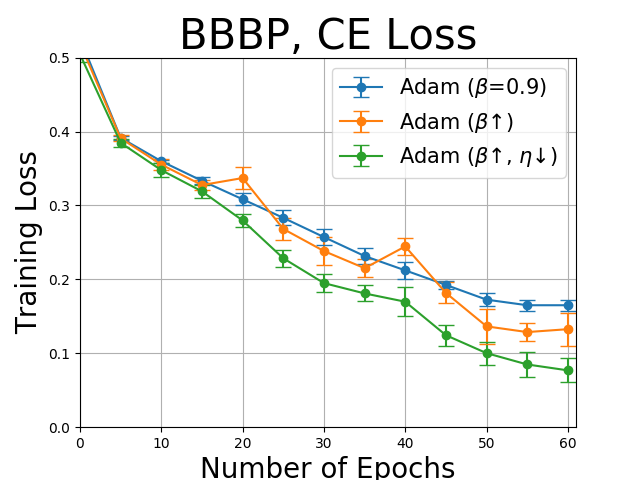}
    \includegraphics[scale=0.23]{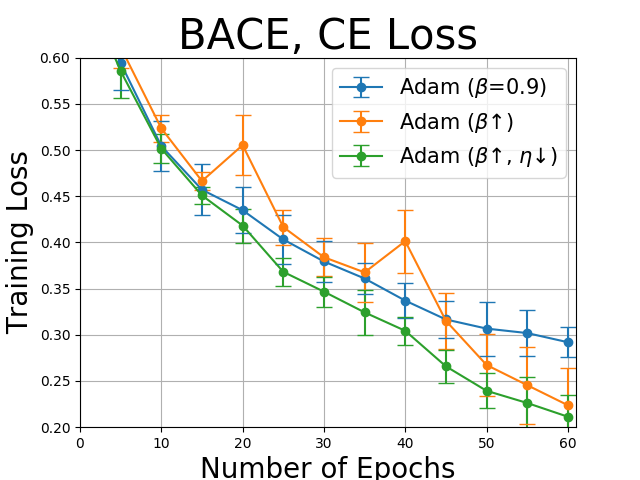}
    } 
    \caption{Minimizing the CE loss}
    \label{fig:min}
\end{figure}

From Figure \ref{fig:min}, we can see that in most cases decreasing $\beta$ especially together with decreasing $\eta$ can improve the convergence speed of the original Adam. This is reasonable because at the beginning the information in the current stochastic gradient is more valuable;  hence using a relatively small $\beta$ is helpful for improving the convergence speed. As the solution gets closer to the optimal solution, the variance of the current stochastic gradient will affect the convergence; hence increasing $\beta$ will help reduce the variance of the gradient estimator. 
What is more, decreasing step size can further accelerate the optimization, which is also consistent with observations in practice of non-adaptive optimization algorithms.

\subsection{Min-Max Problem}
In this subsection, we consider the AUC maximization task for binary classification tasks. Specifically, we optimize a min-max formulated AUC maximization problem \cite{ying2016stochastic,liu2019stochastic},
whose formulation is given in the Appendix \ref{sec:app_auc_formulation}.

The network we use for different data sets are the same as in the minimization problem. For Cifar10 and Cifar100, which has multiple classes, we merge half of their classes as the positive class and the others as the negative class. We compare our proposed algorithms PDSM and PDAda with baselines SGDA \citep{lin2020gradient}, Epoch-SGDA \cite{yan2020sharp}, PES-SGDA and PES-AdaGrad \citep{guo2020fast}. SGDA is a single loop algorithm that updates the primal and dual variable in turn using stochastic gradients. Epoch-SGDA, PES-SGDA and PES-AdaGrad are double algorithms that decay step sizes after a number of iterations, where Epoch-SGDA decays the step size polynomially while the other two decay step size exponentially. The Epoch-SGDA and PES-SGDA update variables using stochastic gradient while PES-AdaGrad uses an AdaGrad style update. The second order momentum of PDAda is implemented as Adam, shown in Table \ref{tab:2}. 

For all algorithms, we tune initial $\eta_x$ and $\eta_y$ in $1e\text{-}5\sim 1e\text{-}2$, and $G_0$ is tuned in $\{1e\text{-}7, 1e\text{-}5, 1e\text{-}3\}$. 
For PDSM and PDAda, we select the initial $\beta=1-\gamma$ from $\{0.99, 0.9, 0.5\}$.
For the algorithms other than SGDA, we decay the step size and $\gamma$ every $E$ epochs where $E$ is chosen from $\{5, 10, 15, 20, 25, 30\}$.
In Epoch-SGDA, the step sizes of the $k$-th stage are $\frac{\eta_x}{k+1}$ and $\frac{\eta_y}{k+1}$. PES-SGDA, PES-AdaGrad, PDSM and PDAda decays step size and $\gamma$ by a factor $e$ tuned in $\{0.1, 0.2, 0.5, 0.9\}$. 
Similar as before, we use a batch size of 32, and repeat the experiments 5 times and report the averaged results with standard deviation. 

The results in Figure \ref{fig:min_max} have demonstrated that in most cases our PDAda and PDSM can outperform the non-adaptive algorithms i.e., SGDA, Epoch-SGDA and PES-SGDA, which indicates that the moving average estimator is helpful for improving the convergence. Also, our PDAda can outperform PES-AdaGrad. Note that the main differences between PDAda and PES-AdaGrad are that PDAda uses SEMA to  estimate  the gradient while PES-AdaGrad simply uses stochastic gradient. 

\begin{figure} [htpb]
    \centering
    \subfigure[Results on Image Data]{
    \includegraphics[scale=0.23]{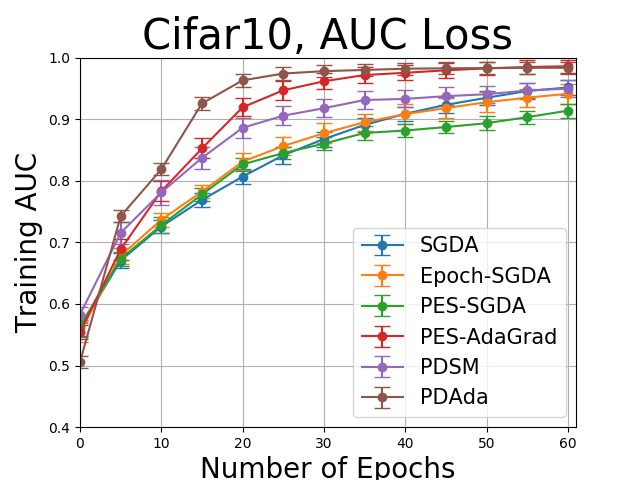}
    \includegraphics[scale=0.23]{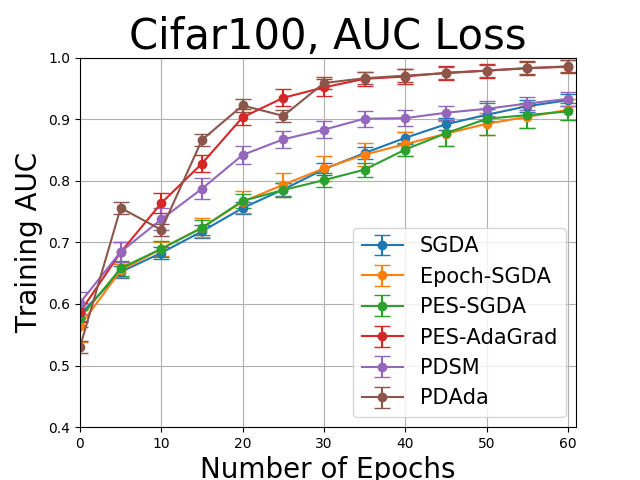}
    }
    \subfigure[Results on Molecule Data]{
    \includegraphics[scale=0.23]{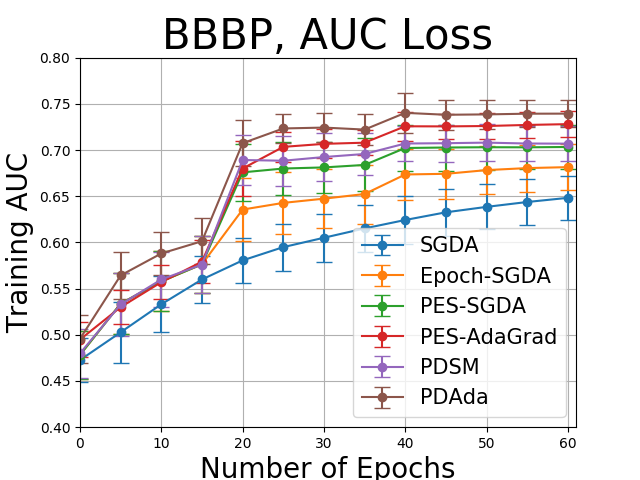}
    \includegraphics[scale=0.23]{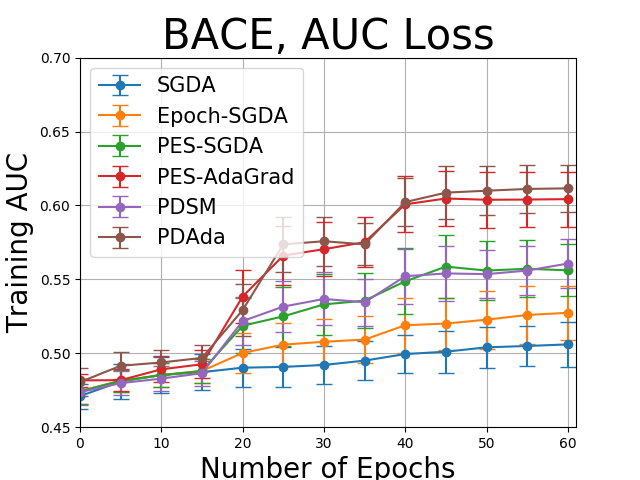}
    } 
    \caption{Optimizing the Min-Max AUC loss}
    \label{fig:min_max}
    
\end{figure}

\section{Conclusion}
In this paper, we have considered the application of stochastic moving average estimators in non-convex optimization and established some interesting and important results. Our results not only bring some new insights to make the Adaptive methods converge but also improve the state of the art results for stochastic non-convex strongly concave min-max optimization and stochastic bilevel optimization with a strongly convex lower-level problem.  The oracle complexities established in this paper are optimal up to a logarithmic factor under a general stochastic unbiased oracle model. 

\acks{We thank the reviewers for their feedback to improve the manuscripts. This work was partially supported National Science Foundation Career Award 1844403.}

\bibliography{reference} 
\onecolumn 
\appendix

\section{Stochastic Momentum Method}
\label{appendix:SMM}
In the literature~\cite{yangnonconvexmo}, the stochastic heavy-ball method is written as:
\begin{equation}\label{eqn:hb2}
\hspace*{-0.3in}\text{SHB:}\quad\left\{ \begin{aligned}
& \v'_{t+1} = \beta' \v'_{t} - \eta'\O_F(\x_t)\\
&\x_{t+1}  =  \x_t + \v'_{t+1}, \quad t= 0, \ldots, T. 
\end{aligned}\right.
\end{equation}
To show the resemblance between the above update and the one in~(\ref{eqn:sma}), we can transform them into one sequence update:
\begin{align*}
&(\ref{eqn:sma}):  \x_{t+1} = \x_t - \eta (1-\beta)\O_F(\x_t) +   \beta(\x_t - \x_{t-1})\\
& \text{SHB: } \x_{t+1} = \x_t - \eta'\O_F(\x_t) + \beta' (\x_t - \x_{t-1}).
\end{align*}
We can see that SHB is equivalent to~(\ref{eqn:sma}) with $\eta'= \eta(1-\beta)$ and $\beta' =   \beta$.

%

\section{Proof of Lemma~\ref{lem:3}}
\begin{proof}
Due to the smoothness of $F$, we can prove that under $\eta_t L_F\leq c_l/(2c_u^2)$
\begin{align*} 
&F(\x_{t+1}) \leq F(\x_t) + \nabla F(\x_t)^{\top} (\x_{t+1} - \x_t) + \frac{L_F}{2}\|\x_{t+1} - \x_t\|^2\\
&= F(\x_t) - \nabla F(\x_t)^{\top}(\teta_t\circ \v_{t+1}) + \frac{L_F}{2}\|\teta_t\circ\v_{t+1}\|^2\\ 
&= F(\x_t) +   \frac{1}{2}\|\sqrt{\teta_t}\circ(\nabla F(\x_t) - \v_{t+1})\|^2- \frac{1}{2}\|\sqrt{\teta_t}\circ\nabla F(\x_t)\|^2 \\
&~~~ + (\frac{L_F}{2} \|\teta_t\circ\v_{t+1}\|^2 - \frac{1}{2}\|\sqrt{\teta_t}\circ\v_{t+1}\|^2 ) \\
&  \leq F(\x_t) +   \frac{\eta_t c_u}{2}\|\nabla F(\x_t) - \v_{t+1}\|^2- \frac{\eta_t c_l}{2}\|\nabla F(\x_t)\|^2   +  \frac{\eta_t^2 c_u^2L_F - \eta_t c_l}{2}\|\v_{t+1}\|^2\\ 
&  \leq F(\x_t) +   \frac{\eta_t c_u}{2}\|\nabla F(\x_t) -  \v_{t+1}\|^2- \frac{\eta_t c_l}{2}\|\nabla F(\x_t)\|^2  -  \frac{\eta_t c_l}{4}\|\v_{t+1}\|^2. 
\end{align*}
\end{proof}

\section{Proof of Theorem~\ref{thm:2}}
\label{sec:app_proof_thm1}
\begin{proof}
 By applying Lemma~\ref{lem:sema_mengdi} to $\v_{t+1}$, we have
\begin{align*}
\E_t[\Delta_{t+1}]\leq (1-\gamma)\Delta_t+ 2\gamma^2\sigma^2(1+c\|\nabla F(\x_{t+1})\|^2)  + \frac{L_F^2\|\x_{t+1} - \x_{t}\|^2}{\gamma}.
\end{align*}
Hence we have, 
\begin{equation*}
\begin{split}
\E\left[\sum_{t=0}^T\Delta_t\right]\leq& \E\bigg[\sum_{t=0}^T\frac{\Delta_t - \Delta_{t+1}}{\gamma} + 2\gamma\sigma^2(T+1) + 2\gamma\sigma^2 c\sum_{t=0}^T\|\nabla F(\x_{t+1})\|^2 \\
& + \sum_{t=0}^T\frac{L_F^2\eta^2c_u^2\|\v_{t+1}\|^2}{\gamma^2}\bigg].
\end{split}
\end{equation*}

Note that $\gamma$ needs to be $O(1/\sqrt{T})$ such that $\frac{1}{T+1}\E\left[\sum_{t=0}^T\Delta_t\right]$ can eventually converge to $O(1/\sqrt{T})$, which is the fundamental reason why we need a large momentum $\beta_t$, where $\beta_t = 1-\gamma$.   

 Adding the above inequality with Lemma~\ref{lem:3}, we have
\begin{align*}
&\frac{\eta c_l}{2}\E\left[\sum_{t=0}^T\|\nabla F(\x_t)\|^2\right] \leq F(\x_0) - F_* - \frac{\eta c_l}{4}\sum_{t=0}^T\E\|\v_{t+1}\|^2\\
&  + \frac{\eta c_u}{2}\E\left[\sum_{t=0}^T\frac{\Delta_t - \Delta_{t+1}}{\gamma} + 2\gamma\sigma^2(T+1) + 2\gamma\sigma^2 c\sum_{t=0}^{T}\|\nabla F(\x_{t+1})\|^2+ \sum_{t=0}^T\frac{L_F^2\eta^2c_u^2\|\v_{t+1}\|^2}{\gamma^2}\right]\\ 
&\leq   F(\w_0) - F_* - \frac{\eta c_l}{4}\sum_{t=0}^T\E\|\v_{t+1}\|^2 
+  2\gamma\sigma^2 c\sum_{t=0}^{T}\E\|\nabla F(\x_{t}) - \nabla F(\x_{t}) + \nabla F(\x_{t+1})\|^2
 \\ 
& + \frac{\eta c_u}{2} \E\left[ \frac{\Delta_0 - \Delta_{T+1}}{\gamma} + 2\gamma\sigma^2(T+1) 
+ \sum_{t=0}^T\frac{L_F^2\eta^2c_u^2\|\v_{t+1}\|^2}{\gamma^2} \right] \\
&\leq   F(\w_0) - F_* - \frac{\eta c_l}{4}\sum_{t=0}^T\E\|\v_{t+1}\|^2 
+  2\eta c_u \gamma\sigma^2 c \sum_{t=0}^{T}\E\|\nabla F(\x_t)\|^2 \\
&+ 2\eta c_u \gamma\sigma^2 c L_F^2 \eta^2 c_u^2 \E\|\v_{t+1}\|^2
+ \frac{\eta c_u}{2} \E\left[ \frac{\Delta_0 - \Delta_{T+1}}{\gamma} + 2\gamma\sigma^2(T+1) 
+ \sum_{t=0}^T\frac{L_F^2\eta^2c_u^2\|\v_{t+1}\|^2}{\gamma^2} \right].
\end{align*}
Let $L_F^2\eta^2 c_u^3/(2\gamma^2) \leq  c_l/8$ (i.e., $\eta\leq \frac{\gamma\sqrt{c_l}}{2L_F\sqrt{c_u^3}})$ and $2c_u\gamma\sigma^2 c\leq c_l/4$, $2\gamma\sigma^2cL_F^2\eta^2c_u^3\leq c_l/8$ (i.e, $\eta \leq \frac{1}{\sqrt{2} L_F c_u}$), we have
\begin{equation*} 
\begin{split}
\frac{1}{T+1}\E\left[\sum_{t=0}^T\|\nabla F(\x_t)\|^2\right] &\leq \frac{\Delta_0c_u}{\gamma T c_l}  +  \frac{2(F(\x_0) - F_*)}{\eta c_lT} + 2\gamma\sigma^2\frac{c_u}{c_l}  +  \frac{1}{2}\frac{1}{T+1}\E\bigg[\sum_{t=0}^T\|\nabla F(\x_{t})\|^2\bigg].
\end{split}
\end{equation*}
As a result, 
\begin{align*}
\frac{1}{T+1}\E\left[\sum_{t=0}^T\|\nabla F(\x_t)\|^2\right] &\leq\frac{2\Delta_0c_u}{\gamma T c_l}  +  \frac{4(F(\x_0) - F_*)}{\eta c_lT} + 4\gamma\sigma^2\frac{c_u}{c_l}. 
\end{align*}

By setting $\gamma=O(\min(\frac{c_l}{8c_u\sigma^2 c}, \frac{1}{\sqrt{T}}))$, $\eta = O(\min(\frac{c_l}{2c_u^2 L_F}, \frac{1}{\sqrt{T}}))$, we have 
\begin{align*}
\frac{1}{T+1}\E\left[\sum_{t=0}^T\|\nabla F(\x_t)\|^2\right] &\leq O\left(\frac{1}{\sqrt{T}}\right).  
\end{align*}

Furthermore, 
we have
\begin{align*}
&\E\left[\sum_{t=0}^T\Delta_t\right]\leq \frac{\Delta_0}{\gamma} + \gamma\sigma^2(T+1) +\frac{c_l}{2c_u}\E\bigg[\sum_{t=0}^T\|\nabla F(\x_{t})\|^2\bigg] +  \E\bigg[\sum_{t=0}^T\frac{c_l}{2c_u}  \|\v_{t+1}\|^2\bigg]\\
& \leq \frac{\Delta_0}{\gamma} + 2\gamma\sigma^2(T+1) +\frac{1}{2}\E\bigg[\sum_{t=0}^T\|\nabla F(\x_{t})\|^2\bigg] +  \E\bigg[\sum_{t=0}^T\frac{1}{2}\Delta_t\bigg]
\end{align*}
As a result, 
\begin{align*}
&\E\left[\frac{1}{T+1}\sum_{t=0}^T\Delta_t\right]\leq O\left(\frac{1}{\sqrt{T}}\right).   
\end{align*}
\end{proof} 

\section{Poof of Theorem~\ref{thm:SHB_decrease}} 
\begin{proof}
By applying Lemma \ref{lem:sema_mengdi} to $\v_{t+1}$, we have \begin{equation}      
\begin{split}
\E_t[\Delta_{t+1}] \leq (1-\gamma_t) \Delta_t + 2\gamma_t^2 \sigma^2 (1+c\|\nabla F(\x_{t+1})\|^2) + \frac{L_F^2\|\x_{t+1}-\x_t\|^2}{\gamma_t}.
\end{split} 
\end{equation}
Hence we have
\begin{equation} 
\begin{split}
\E\left[ \sum\limits_{t=0}^T \gamma_t \Delta_t\right] \leq \E\left[\sum\limits_{t=0}^{T}[\Delta_t - \Delta_{t+1}] + \sum\limits_{t=0}^{T} 2\gamma_t^2\sigma^2(1+c\|\nabla F(\x_{t+1})\|^2) + \sum\limits_{t=0}^{T} \frac{L_F^2 \eta_t^2 c_u^2\|\v_{t+1}\|^2}{\gamma_t} \right]. 
\end{split}
\end{equation}
Note that by setting $\gamma_t=O(1/\sqrt{t+1})$, $\gamma_T$ goes to $O(1/\sqrt{T})$. Therefore $\frac{1}{T+1}\E\left[\sum_{t=0}^T\Delta_t\right]$ can eventually converge to $\widetilde{O}(1/\sqrt{T})$, which is the fundamental reason why we need a large (increasing) momentum $\beta_t$, where $\beta_t = 1-\gamma_t$.  
Combining this with Lemma~\ref{lem:3},  
\begin{equation}
\begin{split}
&\E\left[ \sum\limits_{t=0}^T \frac{\eta_t c_l}{2} \|\nabla F(\x_t)\|^2 \right] \leq \sum\limits_{t=0}^{T} \E[F(\x_t) - F(\x_{t+1})] - \sum\limits_{t=0}^{T} \frac{\eta_t c_l}{4} \E\|\v_{t+1}\|^2  \\ 
&  + \frac{\eta_1 c_u}{2 \gamma_1} \E\left[\sum\limits_{t=0}^T (\Delta_t - \Delta_{t+1}) + \sum\limits_{t=0}^{T} 2\gamma_t^2 \sigma^2 (1+c\|\nabla F(\x_{t+1})\|^2) + \sum\limits_{t=0}^{T} \frac{L_F^2 \eta_t^2 c_u^2 \|\v_{t+1}\|^2}{\gamma_t} \right] \Bigg] \\ 
&\leq F(\x_0) - F_* - \sum\limits_{t=0}^{T} \frac{\eta_t c_l}{4} \E\|\v_{t+1}\|^2 
+\frac{\eta_1 c_u \Delta_0}{2 \gamma_1} + \frac{\eta_1 c_u}{\gamma_1} \sum\limits_{t=0}^{T} \gamma_t^2 \sigma^2 + \frac{\eta_1 c_u}{2\gamma_1}\sum\limits_{t=0}^{T} \frac{L_F^2 \eta_t^2 c_u^2 \E\|\v_{t+1}\|^2}{\gamma_t} \\ 
&+ \frac{\eta_1 c_u}{2 \gamma_1} \sum\limits_{t=0}^{T} 2\gamma_t^2 \sigma^2 c\E\|\nabla F(\x_{t}) -  \nabla F(\x_{t}) + \nabla F(\x_{t+1}) \|^2) \\
&\leq F(\x_0) - F_* - \sum\limits_{t=0}^{T} \frac{\eta_t c_l}{4} \E\|\v_{t+1}\|^2 
+\frac{\eta_1 c_u \Delta_0}{2 \gamma_1} + \frac{\eta_1 c_u}{\gamma_1} \sum\limits_{t=0}^{T} \gamma_t^2 \sigma^2 + \sum\limits_{t=0}^{T} \frac{L_F^2 \eta_t^2 c_u^2 \E\|\v_{t+1}\|^2}{\gamma_t} \\ 
& + \frac{\eta_1 c_u}{2 \gamma_1} \left[\sum\limits_{t=0}^{T} 4\gamma_t^2 \sigma^2 c\E\|\nabla F(\x_{t})\|^2  + \sum\limits_{t=0}^{T} 4\gamma_t^2 \sigma^2 c L_F^2 \eta_t^2 c_u^2 \E\|\v_{t+1}\|^2 \right] \\
&\leq F(\x_0) - F_* + \frac{\eta_1 c_u \Delta_0}{2\gamma_1}  + \frac{c_u \eta_1}{\gamma_1 } \sum\limits_{t=0}^{T} \gamma_t^2\sigma^2 + \sum\limits_{t=0}^{T}  \frac{\eta_t c_l}{4}\E\|\nabla F(\x_{t})\|^2,
\end{split}  
\end{equation} 
where the last inequality holds because $\frac{2\eta_1 c_u}{\gamma_1} \gamma_t^2 \sigma^2 c \leq \frac{\eta_t c_l}{4}$, $\frac{\eta_1}{2\gamma_1} \frac{L_F^2 \eta_t^2 c_u^3}{\gamma_t} \leq \frac{\eta_t c_l}{8}$ and  $\frac{2\eta_1}{\gamma_1}\gamma_t^2 \sigma^2 c L_F^2 \eta_t^2 c_u^3 \leq \frac{\eta_t c_l}{8}$ by the setting of $\eta_t$ and $\gamma_t$ in the theorem.  
Hence,
\begin{equation}
\begin{split}
&\E\left[\sum\limits_{t=0}^{T} \eta_T c_l \|\nabla F(\x_t)\|^2\right] \leq \E\left[\sum\limits_{t=0}^{T} \eta_t c_l \|\nabla F(\x_t)\|^2\right] \\
&\leq 4(F(\x_0) - F_*) + \frac{2\eta_1 c_u \Delta_0}{\gamma_1} + \sum\limits_{t=0}^T \frac{4 c_u \eta_1}{\gamma_1} \gamma_t^2 \sigma^2 \\
&\leq  4(F(\x_0) - F_*) + \frac{\sqrt{c_l} \Delta_0}{L_F \sqrt{c_u}} + \sum\limits_{t=0}^T \frac{2 \sqrt{c_l}}{L_F \sqrt{c_u}} \gamma_t^2 \sigma^2.
\end{split} 
\end{equation}
Thus,
\begin{equation*}
\small 
\begin{split}
&\E \left[ \frac{1}{T+1} \sum\limits_{t=0}^{T} \|\nabla F(\x_t)\|^2 \right] \\
& \leq \frac{4(F(\x_0)-F(\x_*))}{\eta_T c_l (T+1)} + \frac{\Delta_0}{L_F \sqrt{c_l c_u} \eta_T (T+1)} +\sum\limits_{t=0}^{T} \frac{2\gamma_t^2}{\eta_T L_F \sqrt{c_l c_u} (T+1)} \sigma^2  \\
&\leq \frac{4\Delta_F}{\eta_T c_l (T+1)} + \frac{\Delta_0}{L_F \sqrt{c_l c_u} \eta_T (T+1)} + \frac{c_l^2}{32 \eta_T \sigma^2 L_F c^2 \sqrt{c_l c_u^5} (T+1)} \ln(T+2)\\ 
&=O\left(\frac{1}{\sqrt{T}} + \frac{\ln T}{\sqrt{T}}\right). 
\end{split} 
\end{equation*}
Furthermore, we have 
\begin{equation} 
\begin{split}
\E\left[ \sum\limits_{t=0}^T \gamma_t \Delta_t \right] \leq \Delta_0 + \frac{c_l^2}{16 \sigma^4 c^2 c_u^2} \ln(T+2) + \frac{1}{2} \E\left[  \sum\limits_{t=0}^{T} \gamma_t \|\nabla F(\x_t)\|^2\right] 
+ \E\left[ \sum\limits_{t=0}^{T} \frac{1}{2} \gamma_t \Delta_t \right] .
\end{split} 
\end{equation}
Then, 
\begin{equation}
\begin{split}
& \E\left[ \frac{1}{T+1} \sum\limits_{t=0}^{T} \Delta_t \right] \leq \frac{2\Delta_0}{\gamma_T T} + \frac{c_l^2}{8 \gamma_T \sigma^4 c^2 c_u^2 (T+1)} \ln(T+2) 
 + \frac{1}{2\gamma_T T} \E\left[\sum\limits_{t=0}^{T} \gamma_t \|\nabla F(\x_t)\|^2\right]\\
&=O\left(\frac{1}{\sqrt{T}} +  \frac{\ln T}{\sqrt{T}} \right). 
\end{split}
\end{equation}
which concludes the proof of the second part of the theorem. 
\end{proof}

\section{Proof of Theorem \ref{thm:min_PL}}
\begin{proof}[Proof of Theorem \ref{thm:min_PL}]
In this proof the subscript denote the epoch index $(1, ..., K)$.
Denote $\epsilon_0 = \max\{F(\x_0) - F_*, \Delta_0\}$. We prove by induction. Assume that at the initialization of $k$-th stage, we have $\E[F(\x_k) - F_*] \leq \epsilon_k$ and $\E\|\v_k - \nabla F(\x_k) \|^2 \leq \mu \epsilon_k$. By the analysis in Appendix \ref{sec:app_proof_thm1}, we know that after the $k$-th stage,
\begin{align*}
\E[F(\x_{k+1}) - F_*] \leq \frac{1}{2\mu} \E\left[\|\nabla F(\x_k)\|^2\right] &\leq\frac{\epsilon_k c_u}{\mu \gamma_k T_k c_l}  +  \frac{2\epsilon_k}{\mu \eta_k c_l T_k} + 2\gamma_k\sigma^2\frac{c_u}{\mu c_l}, 
\end{align*}
and 
\begin{align*}
&\E\left[\Delta_{k+1}\right]= \E\|\v_k - \nabla F(\x_k) \|^2 \leq \frac{2\epsilon_k}{\gamma_k T_k} + 4\gamma_k\sigma^2 +\E\left[\|\nabla F(\x_{k})\|^2\right].
\end{align*}
By setting $\gamma_k \leq \frac{\mu c_l \epsilon_k}{24 c_u \sigma^2}$, $\eta_k = \min\{\frac{\gamma_k \sqrt{c_l}}{2L_F \sqrt{c_u^3}}, \frac{1}{\sqrt{2}L_f c_u}\}$ and $T_k = \max\{\frac{48 c_u}{\mu\gamma_k c_l}, \frac{1}{6\mu\eta_k c_l}\}$,
then we have $\E[F(\x_{k+1}) - F_*] \leq \epsilon_{k+1} = \epsilon_k/2$ and $\E\|\v_k - \nabla F(\x_k) \|^2 \leq \epsilon_{k+1} = \epsilon_k/2$, where we assume $\mu\leq 1$ without loss of generality. 

Hence, after $K=\log (\epsilon_0/\epsilon)$ stages, it holds that  $\E[F(\x_{K+1}) - F_*] \leq \epsilon$ and $\E\|\v_K - \nabla F(\x_K) \|^2 \leq \epsilon$. The total number of iterations is $\widetilde{O}(\frac{1}{\mu^2 \epsilon})$.

\end{proof}

\section{Analysis of PDSM/PDAda with Strong Concavity}
\label{section:PDAda}
In this section, we analyze PDAda under Assumption \ref{ass:3} with the option $\Y\subseteq\R^{d'}$ is a bounded or unbounded convex set and $f(\x, \cdot)$ is $\lambda$-strongly  concave for any $\x$, while analysis with dual side PL condition is discussed in Appendix \ref{sec:app_pdada_withoutSC}. 
We need the following lemmas.

\begin{lemma}\label{lem:005} 
Suppose Assumption~\ref{ass:2} holds. Considering the PDAda update, with $\eta L_F\leq c_l/(2c_u^2)$ we have
\begin{align*}
&F(\x_{t+1})  \leq F(\x_t) +   \frac{\eta_x c_u}{2} \|\nabla_x f(\x_t, \y^*(\x_t)) - \u_{t+1}\|^2
- \frac{\eta_x c_l}{2}\|\nabla F(\x_t)\|^2  - \frac{\eta_x c_l}{4}\|\u_{t+1}\|^2. 
\end{align*} 
\end{lemma}

\begin{proof}[Proof of Lemma \ref{lem:005}] 
Denote $\teta_x = \eta_x s_t$.
Due to the smoothness of $F$, we have that under $\eta_x \leq 1/(2L_F)$
\begin{align*} 
&F(\x_{t+1}) \leq F(\x_t) + \nabla F(\x_t)^{\top} (\x_{t+1} - \x_t) + \frac{L_F}{2}\|\x_{t+1} - \x_t\|^2\\
&= F(\x_t) -  \nabla F(\x_t)^{\top} (\eta_x\circ \v_{t+1}) + \frac{L_F}{2} \|\teta_x \circ \v_{t+1}\|^2\\
&  = F(\x_t) +   \frac{1}{2}\|\sqrt{\teta_x}\circ (\nabla_x f(\x_t, \y^*(\x_t)) - \v_{t+1})\|^2- \frac{1}{2}  \|\sqrt{\teta_x}\circ \nabla F(\x_t)\|^2 
\\
&~~~ + (\frac{L_F}{2}  \|\teta_x\circ \v_{t+1}\|^2 - \frac{1}{2}\|\sqrt{\teta_x}\circ\v_{t+1}\|^2 ) \\
&  = F(\x_t) +   \frac{1}{2}\|\sqrt{\teta_x}\circ (\nabla_x f(\x_t, \y^*(\x_t)) - \v_{t+1})\|^2- \frac{1}{2}  \|\sqrt{\teta_x}\circ \nabla F(\x_t)\|^2 
\\
&~~~ + (\frac{\eta_x^2 c_u^2 L_F - \eta_x c_l}{2}  \|\v_{t+1}\|^2) \\
&  \leq F(\x_t)
+ \frac{\eta_x c_u}{2}\|\nabla_x f(\x_t, \y^*(\x_t)) - \v_{t+1}\|^2
- \frac{\eta_x c_l}{2}\|\nabla F(\x_t)\|^2-\frac{\eta_x c_l}{4}\|\v_{t+1}\|^2. 
\end{align*}
\end{proof}

\begin{lemma}
In Algorithm \ref{alg:pdadam}, it holds that
\label{lem:var_x}
\begin{equation*}
\begin{split}
& \E\|\v_{t+1} - \nabla_x f(\x_t, \y^*(\x_t))\|^2  \leq (1-\frac{\gamma}{2}) \E\|\v_t - \nabla_x f(\x_{t-1}, \y^*(\x_{t-1}))\|^2 \\
&~~~ + 4\gamma^2 \sigma^2 + 10\gamma L_f^2\E\|\y_t - \y^*(\x_t)\|^2 + \frac{2L_F^2}{\gamma} \E\|\x_t - \x_{t-1}\|^2.  
\end{split}
\end{equation*}
\end{lemma}
\begin{proof}[Proof of Lemma \ref{lem:var_x}]
Let $e_t =(1-\gamma)(\nabla F(\x_t) - \nabla F(\x_{t-1})) = (1-\gamma)(\nabla_x  f(\x_t, \y^*(\x_t)) - \nabla_x f(\x_{t-1}, \y^*(\x_{t-1})))$. We have
\begin{equation}
\|e_t\| \leq (1-\gamma) \|\nabla F(\x_t) - \nabla F(\x_{t-1})\| \leq (1-\gamma) L_F\|\x_t - \x_{t-1}\|. 
\end{equation}
and then it follows that
\begin{equation}
\begin{split}
&\E_t\|\v_{t+1} - \nabla_x f(\x_t, \y^*(\x_t)) + e_t\|^2 \\
& = \E \|(1-\gamma)(\v_{t} - \nabla_x f(\x_{t-1}, \y^*(\x_{t-1}))) + \gamma (\O_x(\x_t, \y_t) - \nabla_x f(\x_t, \y^*(\x_t)))\|^2 \\
&= (1-\gamma)^2 \E\|\v_t - \nabla_x f(\x_{t-1}, \y^*(\x_{t-1}))\|^2
+ \gamma^2 \E\|\O_x(\x_t, \y_t) - \nabla_x f(\x_t, \y^*(\x_t))\|^2 \\
&~~~+2(1-\gamma)\gamma\E \langle \v_t - \nabla_x f(\x_{t-1}, \y^*(\x_{t-1})), 
\O_x(\x_t, \y_t) - \nabla_x f(\x_t, \y_t)
\rangle. \\
&~~~+2(1-\gamma)\gamma\E \langle \v_t - \nabla_x f(\x_{t-1}, \y^*(\x_{t-1})), 
 \nabla_x f(\x_t, \y_t) - \nabla_x f(\x_t, \y^*(\x_t))
\rangle\\
&= (1-\gamma)^2 \E\|\v_t - \nabla_x f(\x_{t-1}, \y^*(\x_{t-1}))\|^2
+ \gamma^2 \E\|\O_x(\x_t, \y_t) - \nabla_x f(\x_t, \y^*(\x_t))\|^2 \\
&~~~+2(1-\gamma)\gamma\E \langle \v_t - \nabla_x f(\x_{t-1}, \y^*(\x_{t-1})), 
 \nabla_x f(\x_t, \y_t) - \nabla_x f(\x_t, \y^*(\x_t))
\rangle\\
&\leq (1-\gamma)^2 \E\|\v_t - \nabla_x f(\x_{t-1}, \y^*(\x_{t-1}))\|^2
+ 2\gamma^2 \E\|\O_x(\x_t, \y_t) - \nabla_x f(\x_t, \y_t)\|^2 \\
&~~~ + 2\gamma^2 \E\|\nabla_x f(\x_t, \y_t) - \nabla_x f(\x_t, \y^*(\x_t))\|^2
 + \frac{\gamma}{4}\E\|\v_t - \nabla_x f(\x_{t-1}, \y^*(\x_{t-1}))\|^2 \\
&~~~ + 4\gamma\E\|\nabla_x f(\x_t, \y_t) - \nabla_x f(\x_t, \y^*(\x_t))\|^2\\
&\leq \left((1-\gamma)^2+\frac{\gamma}{4}\right)\E\|\v_t - \nabla_x f(\x_{t-1}, \y^*(\x_{t-1}))\|^2
+ 2\gamma^2 \sigma^2 + 5 \gamma L_f^2 \|\y_t - \y^*(\x_t)\|^2.
\end{split}
\end{equation}

We have with $\gamma<\frac{1}{4}$, 
\begin{equation*}
\begin{split}
&\E\|\v_{t+1} - \nabla_x f(\x_t, \y^*(\x_t))\|^2 \leq (1+\gamma)\E\|\v_{t+1} - \nabla_x f(\x_t, \y^*(\x_t)) + e_t\|^2 + (1+1/\gamma)\|e_t\|^2\\
&\leq (1-\frac{\gamma}{2}) \E\|\v_t - \nabla_x f(\x_{t-1}, \y^*(\x_{t-1}))\|^2 + 2(1+\gamma)\gamma^2\sigma^2 + 5(1+\gamma)\gamma L_f^2 \E\|\y_t - \y^*(\x_t)\|^2 \\
&~~~ + (1+1/\gamma) (1-\gamma)^2 L_F^2 \|\x_t - \x_{t-1}\|^2 \\
&\leq (1-\frac{\gamma}{2}) \E\|\v_t - \nabla_x f(\x_{t-1}, \y^*(\x_{t-1}))\|^2
+ 4\gamma^2 \sigma^2 + 10\gamma L_f^2\E\|\y_t - \y^*(\x_t)\|^2 \\
&~~~ + \frac{2L_F^2}{\gamma} \E\|\x_t - \x_{t-1}\|^2.  
\end{split}
\end{equation*} 
\end{proof}


\begin{lemma}\label{lem:03}
Suppose Assumption \ref{ass:3} holds with the option $\Y\subseteq\R^{d'}$ is a bounded or unbounded convex set and $f(\x, \cdot)$ is $\lambda$-strongly  concave for any $\x$. With $\y_{t+1} =\Pi_{\Y}[\y_t + \eta_y \O_{fy}(\x_t, \y_t)]$, $\eta_y \leq \lambda/L_f^2$ and $\kappa=L_f/\lambda$, we have
\begin{equation*}
\begin{split} 
\E\|\y_{t+1} - \y^*(\x_{t+1})\|^2 & \leq (1-\frac{\eta_y \lambda}{2}) \E\|\y_t - \y^*(\x_t)\|^2
+ 2\eta_y^2 \sigma^2 + \frac{4\kappa^2}{\eta_y \lambda} \|\x_t - \x_{t+1}\|^2. 
\end{split}
\end{equation*}
\end{lemma}

\begin{proof}[Proof of Lemma \ref{lem:03}]
Since $\y^*(\x_t) = \Pi_{\mathcal{Y}} [\y^*(\x_t) + \eta_y \nabla_y f(\x_t, \y^*(\x_t))]$ and $\y_{t+1} = \Pi_{\mathcal{Y}}[\y_t + \eta_y \O_y(\x_t, \y_t)]$, we have
\begin{equation*}
\begin{split}
& \E\|\y_{t+1} - \y^*(\x_t)\|^2 
= \| \Pi_{y}[\y_t + \eta_y \O_y(\x_t, \y_t)] - \Pi_{y}[\y^*(\x_t) + \eta_y \nabla_y f(\x_t, \y^*(\x_t))] \|^2 \\
& \leq \E\| [\y_t + \eta_y \O_y(\x_t, \y_t)] - [\y^*(\x_t) + \eta_y \nabla_y f(\x_t, \y^*(\x_t))] \|^2 \\
& = \E\| [\y_t + \eta_y \O_y(\x_t, \y_t) - \eta_y \nabla f(\x_t, \y_t) + \eta_y \nabla f(\x_t, \y_t)] - [\y^*(\x_t) + \eta_y \nabla_y f(\x_t, \y^*(\x_t))] \|^2 \\
& = \E\| [\y_t + \eta_y \nabla f(\x_t, \y_t)] - [\y^*(\x_t) + \eta_y \nabla_y f(\x_t, \y^*(\x_t))] \|^2 + \eta_y^2\E\|\O_y(\x_t, \y_t)-\nabla f(\x_t, \y_t)\|^2 \\ 
&+ 2\eta_y\E \langle [\y_t + \eta_y \nabla f(\x_t, \y_t)] - [\y^*(\x_t) + \eta_y \nabla_y f(\x_t, \y^*(\x_t))], \nabla \O_y(\x_t, \y_t)-f(\x_t, \y_t)\rangle\\
&\leq \E\| [\y_t + \eta_y \nabla f(\x_t, \y_t)] - [\y^*(\x_t) + \eta_y \nabla_y f(\x_t, \y^*(\x_t))] \|^2 + \eta_y^2 \sigma^2,
\end{split}
\end{equation*}
where 
\begin{equation}
\begin{split}
&\E\| [\y_t + \eta_y \nabla f(\x_t, \y_t)] - [\y^*(\x_t) + \eta_y \nabla_y f(\x_t, \y^*(\x_t))] \|^2 \\
&= \E\|\y_t - \y^*(\x_t)\|^2 + \eta_y^2 \E\|\nabla f(\x_t, \y_t) - \nabla_y f(\x_t, \y^*(\x_t))\|^2 \\
&~~~ + 2\eta_y \E\langle \y_t - \y^*(\x_t), \nabla_y f(\x_t, \y_t) - \nabla_y f(\x_t, \y^*(\x_t))\rangle\\
&\leq \E\|\y_t - \y^*(\x_t)\|^2 + \eta_y^2 \E\|\nabla f(\x_t, \y_t) - \nabla_y f(\x_t, \y^*(\x_t)) \|^2 -2\eta_y \lambda \E\| \y_t - \y^*(\x_t) \|^2\\
&\leq (1- \eta_y\lambda)\E\|\y_t - \y^*(\x_t)\|^2,
\end{split} 
\end{equation}
where the first inequality uses strong monotone inequality as $-f(\x_t, \cdot)$ is $\lambda$-strongly convex and the second inequality uses the setting $\eta_y \leq \lambda/L_f^2$.
Then,
\begin{equation}
\begin{split} 
&\E\|\y_{t+1} - \y^*(\x_{t+1})\|^2 \leq 
(1+\frac{\eta_y \lambda}{2})\E\|\y_{t+1} - \y^*(\x_t)\|^2 + (1+\frac{2}{\eta_y\lambda})\E\|\y^*(\x_t) - \y^*(\x_{t+1})\|^2 \\
& \leq (1+ \frac{\eta_y \lambda}{2}) (1-\eta_y \lambda) \E\|\y_t - \y^*(\x_t)\|^2 + (1+\frac{\eta_y \lambda}{2}) \eta_y^2 \sigma^2 + \frac{4}{\eta_y\lambda} \kappa^2 
\E\|\x_t - \x_{t+1}\|^2 \\ 
& \leq (1-\frac{\eta_y \lambda}{2}) \E\|\y_t - \y^*(\x_t)\|^2
+ 2\eta_y^2 \sigma^2 + \frac{4\kappa^2}{\eta_y \lambda} \E\|\x_t - \x_{t+1}\|^2,
\end{split}
\end{equation} 
where the first inequality uses the property that $\y^*(\cdot)$ is $\kappa$-Lipschitz (Lemma 4.3 of \citep{lin2020gradient}).
\end{proof}

Then we are ready to prove Theorem \ref{thm:PDAda} in the case $\Y\subseteq\R^{d'}$ is a bounded or unbounded convex set and $f(\x, \cdot)$ is $\lambda$-strongly  concave for any $\x$.



\begin{proof}[Proof of Theorem \ref{thm:PDAda}]
Below, we denote by $\Delta_{x, t}= \|\nabla_x f(\x_t, \y^*(\x_t))- \v_{t+1}\|^2$. 
By applying Lemma~\ref{lem:var_x}, we have 
\begin{align*}
&\E\bigg[\sum_{t=0}^T\Delta_{x,t}\bigg] \leq \frac{2\E[\Delta_{x,0}]}{\gamma} + 8\gamma\sigma^2(T+1) +20L_f^2\E\left[\sum_{t=1}^{T+1}\delta_{y,t}\right] + \E\bigg[\sum_{t=0}^T\frac{4L_F^2\eta_x^2 c_u^2\|\v_{t+1}\|^2)}{\gamma^2}\bigg]. 
\end{align*} 

Applying Lemma~\ref{lem:03}, we have
\begin{align*}
&\E\bigg[\sum_{t=0}^T \delta_{y,t}\bigg] \leq \frac{2}{\eta_y \lambda}\E[\delta_{y,0}] + \frac{4\eta_y\sigma^2(T+1)}{\lambda} + \frac{8\kappa^2\eta_x^2 c_u^2}{\eta_y^2  \lambda^2}\sum\limits_{t=0}^{T} \|\v_{t+1}\|^2.  
\end{align*} 
and 
\begin{align*}
&\E\bigg[\sum_{t=1}^{T+1} \delta_{y,t}\bigg] \leq \frac{2}{\eta_y \lambda}\E[\delta_{y,0}] + \frac{4\eta_y\sigma^2(T+1)}{\lambda} + \frac{8\kappa^2\eta_x^2 c_u^2}{\eta_y^2  \lambda^2}\sum\limits_{t=0}^{T} \|\v_{t+1}\|^2.  
\end{align*} 

Combining the above two bounds with Lemma~\ref{lem:005}, we have 
\begin{align*}
&\E\bigg[\sum_{t=0}^T\|\nabla F(\x_t)\|^2\bigg]\leq \frac{2(F(\x_0) - F(\x_{T+1}))}{\eta_x c_l} + \frac{c_u}{c_l} \E\left[\sum\limits_{t=0}^{T} \Delta_{x, t} \right] - \frac{1}{2} \sum\limits_{t=0}^{T} \|\v_{t+1}\|^2 \\
&\leq \frac{2(F(\x_0) - F_*)}{\eta_x c_l} 
+ \frac{2c_u\E[\Delta_{x,0}]}{\gamma c_l} + \frac{8c_u\gamma\sigma^2(T+1)}{c_l}  + \frac{20c_u L_f^2}{c_l}\E\left[\sum\limits_{t=1}^{T+1}\delta_{y,t}\right] \\
&~~~ + \E\bigg[\sum_{t=0}^T\frac{4c_u^3 L_F^2\eta_x^2 \|\v_{t+1}\|^2)}{c_l \gamma^2}\bigg]  - \frac{1}{2} \sum\limits_{t=0}^{T} \|\v_{t+1}\|^2  \\
&\leq \frac{2(F(\x_0) - F_*)}{\eta_x c_l} 
+ \frac{2c_u\E[\Delta_{x,0}]}{\gamma c_l} + \frac{8c_u\gamma\sigma^2(T+1)}{c_l} 
+\frac{40c_uL_f^2 }{c_l\eta_y  \lambda}\E[\delta_{y,0}] + \frac{80c_u L_f^2 \eta_y\sigma^2(T+1)}{c_l \lambda} \\ 
&~~~ + \bigg(\frac{160c_u^3L_f^2\kappa^2\eta_x^2}{c_l\eta_y^2 \lambda^2} +  \frac{4c_u^3L_F^2\eta_x^2}{c_l \gamma^2}  - \frac{1}{2} \bigg) \E\left[\sum\limits_{t=0}^{T} \|\v_{t+1}\|^2\right]\\
&\leq \frac{2(F(\x_0) - F_*)}{\eta_x c_l} 
+ \frac{2c_u\E[\Delta_{x,0}]}{\gamma c_l} + \frac{8c_u\gamma\sigma^2(T+1)}{c_l} 
+\frac{40c_uL_f^2 }{c_l\eta_y  \lambda}\E[\delta_{y,0}] + \frac{80c_u L_f^2 \eta_y\sigma^2(T+1)}{c_l \lambda}, 
\end{align*}  
where the last inequality uses the fact $\bigg(\frac{160c_u^3(L_f^2\kappa^2\eta_x^2}{c_l\eta_y^2 \lambda^2} + \frac{4c_u^3L_F^2\eta_x^2}{c_l \gamma^2}  - \frac{1}{2} \bigg) \leq -\frac{1}{4} \leq 0$ due to 
\begin{align*}
\eta_x \leq \sqrt{\frac{c_l}{c_u^3}} \min\{\frac{\eta_y \lambda}{48 L_f\kappa}, \frac{\gamma}{8L_F}\}. 
\end{align*}
Hence, we have
\begin{align*}
    \frac{1}{T+1}\E\bigg[\sum_{t=0}^T\|\nabla F(\x_t)\|^2\bigg]\leq& \frac{2\Delta_F}{\eta_x c_l T} 
+ \frac{2c_u\E[\Delta_{x,0}]}{\gamma c_l T} + \frac{8c_u\gamma\sigma^2}{c_l}
+ \frac{40c_uL_f^2}{c_l\eta_y  \lambda T}\E[\delta_{y,0}] + \frac{80c_u L_f^2 \eta_y\sigma^2}{c_l \lambda}. 
\end{align*} 
With $\gamma = O(\frac{\sqrt{\kappa}}{\sqrt{T}})$, $\eta_x = O(\frac{1}{\kappa^{3/2}\sqrt{T}})$ and $\eta_y = O(\frac{\sqrt{\kappa}}{\sqrt{T}})$, we have 
\begin{align*} 
    \frac{1}{T+1}\E\bigg[\sum_{t=0}^T\|\nabla F(\x_t)\|^2\bigg]\leq O\left(\frac{\kappa^{3/2}}{\sqrt{T}}\right),
\end{align*}  
which concludes the first part of the theorem. For the second part, we have 
\begin{equation}
\small 
\begin{split} 
&\E\left[\sum\limits_{t=0}^{T} (\Delta_{x, t} + L_f^2 \delta_{y,t}) \right]\\
& \leq \frac{2\E[\Delta_{x,0}]}{\gamma} + 8\gamma\sigma^2(T+1) +\E\left[\sum\limits_{t=0}^{T}(20L_f^2\delta_{y,t+1} + L_f^2\delta_{y,t})\right] +  \E\bigg[\sum_{t=0}^T\frac{4L_F^2c_u^2 \eta_x^2 \|\v_{t+1}\|^2)}{\gamma^2}\bigg]\\
& \leq \frac{2\E[\Delta_{x,0}]}{\gamma} + 8\gamma\sigma^2(T+1) +
\frac{42 L_f^2}{\eta_y \lambda}\E[\delta_{y,0}] + \frac{84 L_f^2 \eta_y\sigma^2(T+1)}{\lambda} \\
& + \frac{168 L_f^2\kappa^2\eta_x^2c_u^2}{\eta_y^2  \lambda^2}\sum\limits_{t=0}^{T} \|\v_{t+1}\|^2 
+ \E\bigg[\sum_{t=0}^T\frac{4L_F^2\eta_x^2 c_u^2 \|\v_{t+1}\|^2}{\gamma^2}\bigg] \\
&\leq \frac{2\E[\Delta_{x,0}]}{\gamma} + 8\gamma\sigma^2(T+1) +
\frac{42L_f^2}{\eta_y \lambda}\E[\delta_{y,0}] + \frac{84L_f^2\eta_y\sigma^2(T+1)}{\lambda} \\ 
&+ \frac{1}{3} \E\left[\sum\limits_{t=0}^{T} \|\v_{t+1} - \nabla_x f(\x_t, \y^*(\x_t)) - \nabla F(\x_t)\|^2   \right]\\
&\leq \frac{2\E[\Delta_{x,0}]}{\gamma} + 8\gamma\sigma^2(T+1) +
\frac{42L_f^2}{\eta_y \lambda}\E[\delta_{y,0}] + \frac{84L_f^2 \eta_y\sigma^2(T+1)}{\lambda} \\ 
&+ \frac{2}{3} \E\left[\sum\limits_{t=0}^{T} \Delta_{x,t} +\|\nabla F(\x_t)\|^2   \right] 
\end{split}
\end{equation}
Thus, 
\begin{equation}
\begin{split} 
\frac{1}{T+1}\E\left[\sum\limits_{t=0}^{T} (\Delta_{x, t} + L_f^2 \delta_{y,t}) \right] \leq O\left(\frac{\kappa^{3/2}}{\sqrt{T}}\right).  
\end{split}
\end{equation}
\end{proof}

\section{Analysis of PDSM/PDAda without Strong Concavity}
\label{sec:app_pdada_withoutSC}
Note that the analysis of PDSM/PDAda in the previous section uses strong concavity mainly in proving Lemma \ref{lem:03}. Therefore, we only need to provide a similar bound as in \ref{lem:03} then we can fit in the framework of the previous section.
\begin{lemma}[Lemma A.3 of \citep{nouiehed2019solving}]
\label{lem:nouiheld}
Under Assumption \ref{ass:3}, with \\$\kappa = L_f/\lambda$, for any $\x_1, \x_2$, and $\y^*(\x_1) \in \arg\max_{y'} f(\x_1, \y')$, there exists some \\$\y^*(\x_2) \in \arg\max_{y'} f(\x_2, \y')$ such that
\begin{equation}
  \|\y^*(\x_1) - \y^*(\x_2) \| \leq \kappa \|\x_1-\x_2\|.
\end{equation}
\end{lemma} 
Note that in the further analysis we need to properly choose all the $\y^*(\x_t)$ as required by Lemma \ref{lem:nouiheld}, i.e., given $\y^*(\x_t)$, $\y^*(\x_{t+1})$ is chosen from the set $\arg\max_{y'} f(\x_{t+1}, \y')$  such that $\|\y^*(\x_t) - \y^*(\x_{t+1}) \| \leq \kappa \|\x_t-\x_{t+1}\|$. Then we can prove the following lemma.
\begin{lemma}\label{lem:03_withoutSC}
Suppose Assumption \ref{ass:3} holds. With $\y_{t+1} =\y_t + \eta_y \O_{fy}(\x_t, \y_t)$, $\eta_y \leq \lambda/(2L_f^2)$ and $\kappa=L_f/\lambda$, we have that for any $\y^*(\x_t)\in \arg\max_{y'} f(\x_{t}, \y')$ there is a $\y^*(\x_{t+1})\in \arg\max_{y'} f(\x_{t+1}, \y')$ such that 
\begin{equation*}
\begin{split} 
\E\|\y_{t+1} - \y^*(\x_{t+1})\|^2 & \leq (1-\frac{\eta_y \lambda}{4}) \E\|\y_t - \y^*(\x_t)\|^2
+ 2\eta_y^2 \sigma^2 + \frac{8\kappa^2}{\eta_y \lambda} \|\x_t - \x_{t+1}\|^2. 
\end{split}
\end{equation*}
\end{lemma}
\begin{proof}[Proof of Lemma \ref{lem:03_withoutSC}]
It holds that
\begin{equation*}
\begin{split}
& \E\|\y_{t+1} - \y^*(\x_t)\|^2 
= \| \y_t + \eta_y \O_y(\x_t, \y_t) - \y^*(\x_t)] \|^2 \\
& = \E\| [\y_t + \eta_y \O_y(\x_t, \y_t)] - [\y^*(\x_t) + \eta_y \nabla_y f(\x_t, \y^*(\x_t))] \|^2 \\
& = \E\| [\y_t + \eta_y \O_y(\x_t, \y_t) - \eta_y \nabla f(\x_t, \y_t) + \eta_y \nabla f(\x_t, \y_t)] - [\y^*(\x_t) + \eta_y \nabla_y f(\x_t, \y^*(\x_t))] \|^2 \\
& = \E\| [\y_t + \eta_y \nabla f(\x_t, \y_t)] - [\y^*(\x_t) + \eta_y \nabla_y f(\x_t, \y^*(\x_t))] \|^2 + \eta_y^2\E\|\O_y(\x_t, \y_t)-\nabla f(\x_t, \y_t)\|^2 \\ 
&+ 2\eta_y\E \langle [\y_t + \eta_y \nabla f(\x_t, \y_t)] - [\y^*(\x_t) + \eta_y \nabla_y f(\x_t, \y^*(\x_t))], \nabla \O_y(\x_t, \y_t)-f(\x_t, \y_t)\rangle\\
&\leq \E\| [\y_t + \eta_y \nabla f(\x_t, \y_t)] - [\y^*(\x_t) + \eta_y \nabla_y f(\x_t, \y^*(\x_t))] \|^2 + \eta_y^2 \sigma^2,
\end{split}
\end{equation*}
where 
\begin{equation}
\begin{split}
&\E\| [\y_t + \eta_y \nabla f(\x_t, \y_t)] - [\y^*(\x_t) + \eta_y \nabla_y f(\x_t, \y^*(\x_t))] \|^2 \\
&= \E\|\y_t - \y^*(\x_t)\|^2 + \eta_y^2 \E\|\nabla f(\x_t, \y_t) - \nabla_y f(\x_t, \y^*(\x_t))\|^2 
+ 2\eta_y \E\langle \y_t - \y^*(\x_t), \nabla f(\x_t, \y_t) \rangle\\
&\leq \E\|\y_t - \y^*(\x_t)\|^2 + \eta_y^2 \E\|\nabla f(\x_t, \y_t) - \nabla_y f(\x_t, \y^*(\x_t))\|^2 + 2\eta_y(f(\x_t, \y_t) - f(\x_t, \y^*(\x_t))\\
&\leq \E\|\y_t - \y^*(\x_t)\|^2 + \eta_y^2 \E\|\nabla f(\x_t, \y_t) - \nabla_y f(\x_t, \y^*(\x_t)) \|^2 -\eta_y \lambda \E\| \y_t - \y^*(\x_t) \|^2\\
&\leq (1- \frac{\eta_y\lambda}{2})\E\|\y_t - \y^*(\x_t)\|^2.  
\end{split} 
\end{equation}
where the first inequality uses the concavity of $f(\x, \cdot)$,  the second inequality is due to that the dual side $\lambda$-PL condition of $f(\x, \cdot)$ (Appendix A of \citep{karimi2016linear})
and the third inequality uses the setting $\eta_y \leq \lambda/(2L_f^2)$.

Then,
\begin{equation}
\begin{split} 
&\E\|\y_{t+1} - \y^*(\x_{t+1})\|^2 \leq 
(1+\frac{\eta_y \lambda}{4})\|\y_{t+1} - \y^*(\x_t)\|^2 + (1+\frac{4}{\eta_y\lambda})\|\y^*(\x_t) - \y^*(\x_{t+1})\|^2 \\
& \leq (1+ \frac{\eta_y \lambda}{4}) (1-\frac{\eta_y \lambda}{2}) \E\|\y_t - \y^*(\x_t)\|^2 + (1+\frac{\eta_y \lambda}{4}) \eta_y^2 \sigma^2 + \frac{8}{\eta_y\lambda} \kappa^2 
\|\x_t - \x_{t+1}\|^2 \\ 
& \leq (1-\frac{\eta_y \lambda}{4}) \E\|\y_t - \y^*(\x_t)\|^2
+ 2\eta_y^2 \sigma^2 + \frac{8\kappa^2}{\eta_y \lambda} \|\x_t - \x_{t+1}\|^2,
\end{split}
\end{equation} 
where the second inequality uses the Lemma \ref{lem:nouiheld}. 

\end{proof}


\section{Stochastic Non-Convex Bilevel Optimization}
Let $\Pi_{\Omega}$ denote a projection onto a convex set $\Omega$. With $\Omega_R=\{\x\in\R^d: \|\x\|\leq R\}$, we also use $\Pi_{R} = \Pi_{\Omega_R}$ for simplicity.  Let $\S^l_{\lambda}[X]$ denote a projection onto the set $\{X\in \R^{d\times d}: X\succeq \lambda I\}$, and let $\S_C^u[X]$ denote a projection onto the set $\{X\in \R^{d\times d}: \|X\|\leq C\}$. Both $\S^l_{\lambda}[X]$ and $\S^u_{C}[X]$ can be implemented by using singular value decomposition (SVD) and thresholding the singular values. Let $\circ$ denote an element-wise product. We denote by $\x^2$, $\sqrt{\x}$  an element-wise square and element-wise square-root, respectively.

In this section, we consider stochastic non-convex bilevel optimization in the following form: 
\begin{equation}\label{eqn:sbo_formulation}
\begin{split}
&\min_{\x\in\R^d}F(\x) = f(\x, \y^*(\x)), \quad s.t.\quad\y^*(\x)=\arg\min_{\y\in \Y}g(\x, \y),  
\end{split} 
\end{equation}
which satisfies the following assumption: 
\begin{assumption}\label{ass:5}
For $f, g$ we assume the following conditions hold 
\begin{itemize} [leftmargin=*] 
\item $g(\x,\cdot)$ is $\lambda$-strongly convex with respect to $\y$ for any fixed $\x$. 
\item $\nabla_x f(\x, \y)$ is $L_{fx}$-Lipschitz continuous, $\nabla_y f(\x, \y)$ is $L_{fy}$-Lipschitz continuous, $\nabla_y g(\x, \y)$ is $L_{gy}$-Lipschitz continuous, $\nabla_{xy}^2 g(\x, \y)$ is $L_{gxy}$-Lipschitz continuous, $\nabla_{yy}^2 g(\x, \y)$ is $L_{gyy}$-Lipschitz continuous, all respect to $(\x, \y)$.
\vspace{-0.05in}
\item  $\O_{fx},  \O_{fy},   \O_{gy}, \O_{gxy}, \O_{gyy}$ are unbiased stochastic oracles of  $\nabla_x f(\x, \y), \nabla_y f(\x, \y)$, $\nabla_y g(\x, \y)$, $\nabla_{xy}^2 g(\x, \y)$ and $\nabla_{yy}^2 g(\x, \y)$, and their variances are  by $\sigma^2$, and $0\preceq \O_{gyy}(\x,\y)\preceq C_{gyy} I$. 
\vspace{-0.05in}
\item  $\|\nabla_y f(\x, \y)\|^2 \leq C_{fy}^2$, $\|\nabla_{xy}^2 g(\x, \y)\|^2 \leq C_{gxy}^2$. 
\end{itemize}
\end{assumption}
{\bf Remark: }The above assumptions are similar to that assumed  in \citep{99401,DBLP:journals/corr/abs-2007-05170} except for an additional assumption that $\nabla_x f(\x, \y) $ is Lipschitz continuous with respect to $\x$ for fixed $\y$. Note that \citep{99401,DBLP:journals/corr/abs-2007-05170} have made implicitly a stronger assumption $\lambda I\preceq \O_{gyy}(\x,\y)\preceq C_{gyy} I$ (cf. the proof of Lemma 3.2 in~\citep{99401}). Our assumption regarding $\O_{gyy}(\x,\y)$, i.e.,  $0\preceq \O_{gyy}(\x,\y)\preceq C_{gyy} I$ is weaker. Indeed, we can also remove this assumption by sacrificing the per-iteration complexity. We present this result in Appendix  \ref{section:sbo_alter} for interesting readers. 

The proposed algorithm SMB is presented in Algorithm~\ref{alg:smb}. To understand the algorithm, we write the exact expression for the gradient of $F(\x)$, i.e.,  $\nabla F(\x) =  \nabla_x f(\x, \y^*(\x))  - \nabla_{xy}^2 g(\x, \y^*(\x))[\nabla_{yy}^2g(\x, \y^*(\x))]^{-1}\nabla_y f(\x, \y^*(\x))$ \citep{99401}. At each iteration with $(\x_t, \y_t)$, we can first approximate $\y^*(\x_t)$  by $\y_t$ and define $\nabla F(\x_t, \y_t) = \nabla_x f(\x_t, \y_t) - \nabla^2_{xy}g(\x_t, \y_t)[\nabla_{yy}^2g(\x_t, \y_t)]^{-1}\nabla_y f(\x_t, \y_t)$ as an approximate of $\nabla F(\x_t)$. Except for $[\nabla_{yy}^2g(\x_t, \y_t)]^{-1}$ other components of $\nabla F(\x_t, \y_t)$ have an unbiased estimator based on stochastic oracles. For estimating $[\nabla_{yy}^2g(\x_t, \y_t)]^{-1}$, we use the biased estimator proposed in \citep{99401}, which is given by $h_{t+1}$ in step 3 of SMB, where $k_t$ is a logarithmic number meaning that a logarithmic calls of $\O_{gyy}$ is required at each iteration. Hence, we have a biased estimator of $\nabla F(\x_t, \y_t)$ by $\O_{fx}(\x_t, \y_t) - \O_{gxy}(\x_t,\y_t) h_{t+1} \O_{fy}(\x_t, \y_t))$. To further reduce the variance of this estimator, we apply  the SEMA estimator on top of it as in step 4 of SMB.

\begin{algorithm}[t]
	\caption{
	Stochastic Momentum method for Bilevel Optimization (SMB)}\label{alg:smb}
	\begin{algorithmic}[1]
\State{Input: $\x_0\in \R^d, \y_0 \in\R^{d'}$, $\h_0 = \frac{k_0}{C_{gyy}}\prod_{i=1}^p(I - \frac{1}{C_{gyy}}\O_{gyy, i}(\x_0, \y_0))$, where $p$ is uniformly sampled from $1, \ldots, k_0$, and
$\z_0=\O_{fx}(\x_0, \y_0) - \O_{gxy}(\x_0,\y_0) h_{0} \O_{fy}(\x_0, \y_0)$.}   
\FOR{$t=0, 1, ..., T$}
\State ~~~~Uniformly sample $p$ from $1, \ldots, k_t$
\State ~~~~$h_{t+1}  = \frac{k_t}{C_{gyy}}\prod_{i=1}^p(I - \frac{1}{C_{gyy}}\O_{gyy, i}(\x_t, \y_t))$
\State ~~~~$\z_{t+1}= \beta\z_t + (1-\beta) (\O_{fx}(\x_t, \y_t) - \O_{gxy}(\x_t,\y_t) h_{t+1} \O_{fy}(\x_t, \y_t))$,
\State ~~~~$\x_{t+1} = \x_t - \eta_x\z_{t+1}$
\State {~~~~$\y_{t+1} = \Pi_{\Y}[\y_t + \eta_y \O_{gy}(\x_{t}, \y_t)]$} 
\ENDFOR 
\end{algorithmic}
\end{algorithm}

We have the following convergence regarding SMB. We denote $\Delta_t =  \|\z_{t+1} - \nabla F(\x_t)\|^2$.
\begin{theorem}\label{thm:5}
Let  $ F(\x_0) - F_*\leq \Delta_F$. Suppose Assumption~\ref{ass:5} holds.  By setting $1-\beta=\gamma =O\left(\frac{1}{\sqrt{T}}\right)$, $\eta_x=O\left(\frac{1}{\sqrt{T}}\right)$,  
$\eta_y =O\left(\frac{1}{\sqrt{T}}\right)$, and $k_t = O(\ln T)$ 
we have 
\begin{align*}
\E\left[\frac{1}{T+1}\sum_{t=0}^T \|\nabla F(\x_t)\|^2\right] \leq O\left(\frac{1}{\sqrt{T}}\right), \quad \E\left[\frac{1}{T+1}\sum_{t=0}^T\Delta_{t}\right]\leq O\left(\frac{1}{\sqrt{T}}\right).  
\end{align*} 
\end{theorem}
\vspace*{-0.1in}
{\bf Remark:} The total oracle complexity is $O(\sum_t k_t) =  O(T \ln T)$, which is $\widetilde O(1/\epsilon^4)$ to acheive an $\epsilon$-stationary point. It is not difficult to extend SMB to its adaptive variant by using the similar step size for updating $\x_{t+1}$ as in Algorithm~\ref{alg:pdadam}.

We develop the analysis of Theorem \ref{thm:5} in the following. 
\begin{lemma}[Lemma 2.2, \citep{99401}]
Under Assumption \ref{ass:5}, we have
\begin{equation}
\begin{split}
\|\y^*(\x) - \y^*(\x')\| \leq L_y \|\x - \x'\|,
\end{split}
\end{equation}
where $L_y$ is an appropriate constant.
\label{lem:sbo_lip_y}
\end{lemma}

We generalize Lemma \ref{lem:3} to the bilevel problem as:
\begin{lemma}
\label{lem:sbo_1}
Considering the update in Algorithm \ref{alg:smb} where $\x_{t+1} = \x_t - \eta_x \z_{t+1} $, with $\eta_x \leq 1/(2L_F)$, we have
\begin{equation*}
\begin{split}
F(\x_{t+1}) \leq F(\x_t) + \frac{\eta_x}{2} \|\nabla F(\x_t) - \z_{t+1}\|^2 - \frac{\eta_x}{2} \|\nabla F(\x_t)\|^2 - \frac{\eta_x}{4} \|\z_{t+1}\|^2. 
\end{split}
\end{equation*}
\end{lemma}
Lemma \ref{lem:03} can also be generalized to the bilevel problem as: 
\begin{lemma}  \label{lem:sbo_y_recursion} 
Suppose Assumption \ref{ass:5} holds. With $\y_{t+1} = \Pi_{\mathcal{Y}} [\y_t +\eta_y \O_{gy}(\x_{t}, \y_t)]$, $\eta_y \leq \lambda$ and $L_y$ as specified in Lemma \ref{lem:sbo_lip_y},
\begin{equation*}
\begin{split}
\E\|\y_{t+1} - \y^*(\x_{t+1})\|^2 \leq  (1-\frac{\eta_y  \lambda}{2}) \E\|\y_t - \y^*(\x_t)\|^2
+ 2\eta_y^2 \sigma^2 + \frac{4 L_y^2}{\eta_y \lambda}  \|\x_{t+1} - \x_t\|^2.
\end{split} 
\end{equation*}
\end{lemma}

We need the following lemma to bound the residual between $h_{t+1}$ and $[\nabla_{yy}^2 g(\x_t, \y_t)]^{-1}$. 

\begin{lemma}Under Assumption \ref{ass:5} and considering update of Algorithm \ref{alg:smb}, we have
\begin{align}
& \|\E[h_{t+1}] - [\nabla_{yy}^2 g(\x_t, \y_t)]^{-1}\| \leq \frac{1}{\lambda} \left(1-\frac{\lambda}{C_{gyy}}\right)^{k_t} \label{equ:ghadimi_h_to_g1}\\
&\|h_{t+1}\| \leq \frac{k_t}{C_{gyy}},  \|[\nabla_{yy}^2 g(\x_t, \y_t)]^{-1} - h_{t+1} \| \leq \frac{1}{\lambda} + \frac{k_t}{C_{gyy}}. \label{equ:ghadimi_h_to_g2}
\end{align} 
\label{lem:ghadimi_h_to_g}
\end{lemma} 

\begin{proof}[Proof of Lemma \ref{lem:ghadimi_h_to_g}]
The (\ref{equ:ghadimi_h_to_g1}) is proven in Lemma 3.2 of \citep{99401}. 
Under Assumption \ref{ass:5}, we have
\begin{equation*}
\begin{split}
\|h_{t+1}\| \leq \frac{k_t}{C_{gyy}} \prod\limits_{i=1}^{p} \|I-\frac{1}{C_{gyy}} \O_{gyy,i}(\x_t, \y_t) \| \leq \frac{k_t}{C_{gyy}}.
\end{split} 
\end{equation*}
Thus,
\begin{equation*}
\begin{split}
\|[\nabla_{yy}^2 g(\x_t, \y_t)]^{-1} - h_{t+1} \| \leq
\|\nabla_{yy}^2 g(\x_t, \y_t)]^{-1}\| + \|h_{t+1}\|
\leq \frac{1}{\lambda} + \frac{k_t}{C_{gyy}}.
\end{split}
\end{equation*}
\end{proof}

We bound $\| \nabla F(\x_t) - \z_{t+1}\|^2$ in the following lemma.
\begin{lemma}\label{lem:sbo_var_1}  
For all $t\geq 0$, we have  
\begin{equation*} 
\begin{split} 
&\E\|\z_{t+1} - \nabla F(\x_t, \y^*(\x_t)) \|^2
\leq (1-\frac{\gamma}{2})\E\|\z_t - \nabla F(\x_{t-1}, \y^*(\x_{t-1}))\|^2 \\
&+ 8\gamma C_{gxy}^2 \frac{1}{\lambda^2}(1-\frac{\lambda}{C_{gyy}})^{2k_t} + 8\gamma C_0 \E\|\y_t - \y^*(\x_t)\|^2
+ \gamma^2 C_1 + \frac{2}{\gamma}L_F^2 \E\|\x_{t} - \x_{t-1}\|^2, 
\end{split}
\end{equation*}
where $L_F:=\left(2L_{fx}^2 + \frac{6 C_{gxy}^2 L_{fy}^2}{\lambda^2} + \frac{6 C_{gxy}^2 L_{gyy}^2 C_{fy}^2 }{\lambda^4} + \frac{6L_{gxy}^2 C_{fy}^2 }{\lambda^2}\right)(1+L_y^2)$, $C_0:=(2L_{fx}^2 +  \frac{6C_{fy}^2L_{gxy}^2}{\lambda^2} + \frac{6C_{fy}^2C^2_{gxy}L^2_{gyy}}{\lambda^4}+ \frac{6L_{fy}^2C^2_{gxy}}{\lambda^2})$, $C_1:=2\sigma^2 +  6\sigma^2\frac{k_t^2}{C_{gyy}^2}(C_{fy}^2+\sigma^2) + 24C_{gxy}^2\frac{k_t^2}{C_{gyy}^2}(C_{fy}^2+\sigma^2) + 6 C_{gxy}^2 \frac{k^2_t}{C^2_{gyy}}\sigma^2$ and $L_y$ is as in Lemma \ref{lem:sbo_lip_y}.   
\end{lemma}

\begin{proof}[Proof of Lemma \ref{lem:sbo_var_1}]
First, note that it has shown in the bilevel optimization literature \citep{99401} that  
\begin{equation}
\begin{split}
\nabla F(\x) &= \nabla_x f(\x, \y^*(\x)) + \nabla \y^*(\x)^\top \nabla_y f(\x, \y^*(\x)) \\
&= \nabla_x f(\x, \y^*(\x)) - \nabla_{xy}^2 g(\x, \y^*(\x))[\nabla_{yy}^2 g(\x, \y^*(\x)) ]^{-1} \nabla_y f(\x, \y^*(\x)), 
\end{split}
\label{eq:sbo_nabla_F}
\end{equation}
and denote 
\begin{equation}
\begin{split}
\nabla F(\x_t, \y_t) := \nabla_x f(\x_t, \y_t) - \nabla_{xy}^2 g(\x_t, \y_t)[\nabla_{yy}^2 g(\x_t, \y_t)]^{-1} \nabla_y f(\x_t, \y_t). 
\end{split}
\label{eq:sbo_nabla_F_x_y}
\end{equation}

Let $e_t :=(1-\gamma)(\nabla F(\x_t) - \nabla F(\x_{t-1})) = (1-\gamma)(\nabla_x  F(\x_t, \y^*(\x_t)) - \nabla_x F(\x_{t-1}, \y^*(\x_{t-1})))$. 
We have 
\begin{small}
\begin{equation}
\small
\begin{split}
&\|e_t\| \leq (1-\gamma) \|\nabla F(\x_t) - \nabla F(\x_{t-1})\| \\ 
&\leq \|[\nabla_x f(\x_t, \y^*(\x_t)) - \nabla_{xy}^2 g(\x_t, \y^*(\x_t))[\nabla_{yy}^2 g(\x_t, \y^*(\x_t)) ]^{-1} \nabla_y f(\x_t, \y^*(\x_t))] \\
&~~~ - [\nabla_x f(\x_{t-1}, \y^*(\x_{t-1})) \\
&~~~ - \nabla_{xy}^2 g(\x_{t-1}, \y^*(\x_{t-1}))[\nabla_{yy}^2 g(\x_{t-1}, \y^*(\x_{t-1})) ]^{-1} \nabla_y f(\x_{t-1}, \y^*(\x_{t-1}))]\|^2 \\
&\leq 2 \|\nabla_x f(\x_t, \y^*(\x_t)) - \nabla_x f(\x_{t-1}, \y^*(\x_{t-1}))\|^2\\
& + 6\|\nabla_{xy}^2 g(\x_t, \y^*(\x_t))[\nabla_{yy}^2 g(\x_t, \y^*(\x_t)) ]^{-1} \nabla_y f(\x_t, \y^*(\x_t)) \\
&~~~~~~ - \nabla_{xy}^2 g(\x_t, \y^*(\x_t))[\nabla_{yy}^2 g(\x_t, \y^*(\x_t)) ]^{-1} \nabla_y f(\x_{t-1}, \y^*(\x_{t-1}))\|^2 \\
& + 6\|\nabla_{xy}^2 g(\x_t, \y^*(\x_t))[\nabla_{yy}^2 g(\x_t, \y^*(\x_t)) ]^{-1} \nabla_y f(\x_{t-1}, \y^*(\x_{t-1})) \\
&~~~~~~ - \nabla_{xy}^2 g(\x_t, \y^*(\x_t))[\nabla_{yy}^2 g(\x_{t-1}, \y^*(\x_{t-1})) ]^{-1} \nabla_y f(\x_{t-1}, \y^*(\x_{t-1}))\|^2 \\
& + 6\|\nabla_{xy}^2 g(\x_t, \y^*(\x_t))[\nabla_{yy}^2 g(\x_{t-1}, \y^*(\x_{t-1})) ]^{-1} \nabla_y f(\x_{t-1}, \y^*(\x_{t-1})) \\
&~~ - \nabla_{xy}^2 g(\x_{t-1}, \y^*(\x_{t-1}))[\nabla_{yy}^2 g(\x_{t-1}, \y^*(\x_{t-1})) ]^{-1} \nabla_y f(\x_{t-1}, \y^*(\x_{t-1}))\|^2 \\
&\leq (2L_{fx}^2 \!+\! \frac{6 C_{gxy}^2 L_{fy}^2}{\lambda^2} \!+\! \frac{6 C_{gxy}^2 L_{gyy}^2 C_{fy}^2 }{\lambda^4} \!+ \!\frac{6L_{gxy}^2 C_{fy}^2 }{\lambda^2}) (\|\x_t-\x_{t-1}\|^2 + \|\y^*(\x_t) - \y^*(\x_{t-1})\|^2) \\
&\leq \underbrace{(2L_{fx}^2 + \frac{6 C_{gxy}^2 L_{fy}^2}{\lambda^2} + \frac{6 C_{gxy}^2 L_{gyy}^2 C_{fy}^2 }{\lambda^4} + \frac{6L_{gxy}^2 C_{fy}^2 }{\lambda^2})(1+L_y^2)}\limits_{L_F} \|\x_t-\x_{t-1}\|^2. 
\end{split}
\end{equation}
\end{small}


Define
\begin{equation}
\begin{split}
\widehat{\nabla} F(\x_t, \y_t) = \nabla_{x} f(\x_t, \y_t) - \nabla_{xy} g(\x_t, \y_t) \E[h_{t+1}] 
\nabla_y f(\x_t, \y_t). 
\end{split}
\end{equation}

Then
\begin{equation}
\begin{split}  
&\E\|\z_{t+1} - \nabla F(\x_t) + e_t\|^2 \\
&= \E\| (1-\gamma)(\z_t - \nabla F(\x_{t-1}, \y^*(\x_{t-1}))) \\
&~~~~~~~~~ + \gamma ( \O_{fx}(\x_t,\y_t) - \O_{gy}(\x_t,\y_t)h_{t+1} \O_{fy}(\x_t, \y_t) - \nabla F(\x_t, \y^*(\x_t))) \|^2 \\
&= \E\| (1-\gamma)(\z_t - \nabla F(\x_{t-1}, \y^*(\x_{t-1}))) \\
&~~~~~~ + \gamma ( \O_{fx}(\x_t,\y_t) - \O_{gy}(\x_t,\y_t)h_{t+1} \O_{fy}(\x_t, \y_t) - \widehat{\nabla} F(\x_t, \y_t)) \\  
&~~~~~~ +\gamma (\widehat{\nabla} F(\x_t, \y_t) - \nabla F(\x_t, \y^*(\x_t))) \|^2 \\
&= \E\|(1-\gamma)(\z_t - \nabla F(\x_{t-1}, \y^*(\x_{t-1}))) + \gamma(\widehat{\nabla} F(\x_t, \y_t) - \nabla F(\x_t, \y^*(\x_t))) \|^2 \\
&~~~ + \gamma^2\|\O_{fx}(\x_t,\y_t) - \O_{gy}(\x_t,\y_t)h_{t+1} \O_{fy}(\x_t, \y_t) - \widehat{\nabla} F(\x_t, \y_t)\|^2 \\
&\leq (1+\frac{\gamma}{2})(1-\gamma)^2 \E\|\z_t - \nabla F(\x_{t-1}, \y^*(\x_{t-1}))\|^2
+ (1+\frac{2}{\gamma}) \gamma^2\|\widehat{\nabla} F(\x_t, \y_t) - \nabla F(\x_t, \y^*(\x_t))\|^2 \\
&~~~ + \gamma^2\|\O_{fx}(\x_t,\y_t) - \O_{gy}(\x_t,\y_t)h_{t+1} \O_{fy}(\x_t, \y_t) - \widehat{\nabla} F(\x_t, \y_t)\|^2\\
&\leq (1+\frac{\gamma}{2})(1-\gamma)^2 \E\|\z_t - \nabla F(\x_{t-1}, \y^*(\x_{t-1}))\|^2
+ (1+\frac{2}{\gamma}) 2\gamma^2\|\widehat{\nabla} F(\x_t, \y_t) - \nabla F(\x_t, \y_t)\|^2 \\
&~~~ + (1+\frac{2}{\gamma}) 2\gamma^2\|\nabla F(\x_t, \y_t) - \nabla F(\x_t, \y^*(\x_t))\|^2 \\
&~~~ + \gamma^2\|\O_{fx}(\x_t,\y_t) - \O_{gy}(\x_t,\y_t)h_{t+1} \O_{fy}(\x_t, \y_t) - \widehat{\nabla} F(\x_t, \y_t)\|^2,
\end{split}
\end{equation} 
where the last three terms can be bounded as below.
First,
\begin{small}
\begin{align}
\begin{split} 
\label{eq:F_y_F_y_star}
&\|\nabla F(\x_t, \y_t) - \nabla F(\x_t, \y^*(\x_t))\|^2 \\
&\leq \|\nabla_x f(\x_t, \y_t) - \nabla^2_{xy}g(\x_t, \y_t)[\nabla_{yy}^2g(\x_t, \y_t)]^{-1}\nabla_y f(\x_t, \y_t) \\
&- \nabla_x f(\x_t, \y^*(\x_t)) + \nabla^2_{xy}g(\x_t, \y^*(\x_t))[\nabla_{yy}^2g(\x_t, \y^*(\x_t))]^{-1}\nabla_y f(\x_t, \y^*(\x_t)) \|^2\\
& \leq 2\|\nabla_x f(\x_t, \y_t) -   \nabla_x f(\x_t, \y^*(\x_t))\|^2 \\
&+ 2\| \nabla^2_{xy}g(\x_t, \y_t)[\nabla_{yy}^2g(\x_t, \y_t)]^{-1}\nabla_y f(\x_t, \y_t) \\ 
&~~~ -  \nabla^2_{xy}g(\x_t, \y^*(\x_t)) [\nabla_{yy}^2g(\x_t, \y^*(\x_t))]^{-1}\nabla_y f(\x_t, \y^*(\x_t))\|^2\\
&\leq 2\|\nabla_x f(\x_t, \y_t) -   \nabla_x f(\x_t, \y^*(\x_t))\|^2\\
&  + 6\| \nabla^2_{xy}g(\x_t, \y_t)[\nabla_{yy}^2g(\x_t, \y_t)]^{-1}\nabla_y f(\x_t, \y_t) \\
&~~~ -  \nabla^2_{xy}g(\x_t, \y^*(\x_t))[\nabla_{yy}^2g(\x_t, \y_t)]^{-1}\nabla_y f(\x_t, \y_t)\|^2\\  
& + 6\| \nabla^2_{xy}g(\x_t, \y^*(\x_t))[\nabla_{yy}^2g(\x_t, \y_t)]^{-1}\nabla_y f(\x_t, \y_t) \\
&~~~ -  \nabla^2_{xy}g(\x_t, \y^*(\x_t))[\nabla_{yy}^2g(\x_t, \y^*(\x_t))]^{-1}\nabla_y f(\x_t, \y_t)\|^2\\
& + 6\| \nabla^2_{xy}g(\x_t, \y^*(\x_t))[\nabla_{yy}^2g(\x_t, \y^*(\x_t))]^{-1}\nabla_y f(\x_t, \y_t) \\
&~~~~~~ -  \nabla^2_{xy}g(\x_t, \y^*(\x_t))[\nabla_{yy}^2g(\x_t, \y^*(\x_t))]^{-1}\nabla_y f(\x_t, \y^*(\x_t))\|^2\\
&\overset{(a)}{\leq}2L^2_{fx} \|\y_t - \y^*(\x_t)\|^2 +  \frac{6L_{gxy}^2 C_{fy}^2}{\lambda^2}\|\y_t - \y^*(\x_t)\|^2 
+ \frac{6 C^2_{gxy}L^2_{gyy} C_{fy}^2}{\lambda^4}\|\y_t - \y^*(\x_t)\|^2\\
&+ \frac{6C^2_{gxy} L_{fy}^2}{\lambda^2}\|\y_t - \y^*(\x_t)\|^2\\ 
& = \underbrace{ (2L_{fx}^2 +  \frac{6C_{fy}^2L_{gxy}^2}{\lambda^2} + \frac{6C_{fy}^2C^2_{gxy}L^2_{gyy}}{\lambda^4}+ \frac{6L_{fy}^2C^2_{gxy}}{\lambda^2})}\limits_{C_0}\|\y_t - \y^*(\x_t)\|^2,
\end{split}
\end{align}
\end{small} 
where (a) uses the Lipschitz continuity of $\nabla_x f(\x, \cdot)$, $\nabla_y f(\x, \cdot)$, $\nabla_{xy}^2 g(\x, \cdot)$, $\nabla_{yy}^2 g(\x, \cdot)$, and upper bound of $\|\nabla_y f(\x, \y)\|$, $\|\nabla_{xy}^2 g(\x, \y)\|$, and $\|\nabla_{yy}^2 g(\x, \y)\|$. 
Second,
\begin{equation}
\begin{split} 
&\E\|\widehat{\nabla} F(\x_t, \y_t) - \nabla F(\x_t, \y_t)\|^2 \leq C_{gxy}^2 \|\E[h_{t+1}] - [\nabla_{yy}^2 g(\x_t, \y_t)]^{-1}\|^2 C_{fy}^2 \\
&\leq C^2_{gxy} \frac{1}{\lambda^2} (1-\frac{\lambda}{C_{gyy}})^{2k_t},
\end{split} 
\end{equation} 
where the last inequality uses Assumption \ref{ass:5} and Lemma \ref{lem:ghadimi_h_to_g}. Also we have
\begin{equation} 
\begin{split}
&\E\|\O_{fx}(\x_t,\y_y) - \O_{gy}(\x_t,\y_t)h_{t+1} \O_{fy}(\x_t, \y_t) - \widehat{F}(\x_t, \y_t) \|^2 \\
&\leq 2\E\|\O_{fx}(\x_t, \y_t) - \nabla_x f(\x_t, \y_t)\|^2 \\
&~~~ + 6\E\|\O_{gxy}(\x_t,\y_t)h_{t+1} \O_{fy}(\x_t, \y_t) - \nabla_{xy}^2 g(\x_t, \y_t)h_{t+1} \O_{fy}(\x_t, \y_t) \|^2 \\
&~~~ + 6\E\|\nabla_{xy}^2 g(\x_t, \y_t)h_{t+1} \O_{fy}(\x_t, \y_t) 
- \nabla_{xy}^2 g(\x_t, \y_t)\E[h_{t+1}] \O_{fy}(\x_t, \y_t) \|^2 \\
&~~~ + 6\E\|\nabla_{xy}^2 g(\x_t, \y_t)\E[h_{t+1}] \O_{fy}(\x_t, \y_t)
- \nabla_{xy}^2 g(\x_t, \y_t)\E[h_{t+1}] \nabla_y f(\x_t, \y_t)\|^2 \\ 
&\leq 2\sigma^2 +  6\sigma^2\frac{k_t^2}{C_{gyy}^2}(C_{fy}^2+\sigma^2) + 24C_{gxy}^2\frac{k_t^2}{C_{gyy}^2}(C_{fy}^2+\sigma^2) + 6 C_{gxy}^2 \frac{k^2_t}{C^2_{gyy}}\sigma^2 := C_1. 
\end{split}
\end{equation} 

Thus,
\begin{equation*}
\begin{split} 
&\E\|\z_{t+1} - \nabla F(\x_t, \y^*(\x_t)) \|^2 \leq 
(1+\gamma) \E\|\z_{t+1} - \nabla F(\x_t, \y^*(\x_t)) + e_t\|^2 + (1+1/\gamma) \E\|e_t\|^2 \\ 
&\leq (1-\frac{\gamma}{2})\E\|\z_t - \nabla F(\x_{t-1}, \y^*(\x_{t-1}))\|^2 + 8\gamma C_{gxy}^2 \frac{1}{\lambda^2}(1-\frac{\lambda}{C_{gyy}})^{2k_t}\\
&~~~ 
+ 8\gamma C_0 \E\|\y_t - \y^*(\x_t)\|^2
+ \gamma^2 C_1 + \frac{2}{\gamma}L_F^2 \E\|\x_{t} - \x_{t-1}\|^2. 
\end{split}
\end{equation*}
\end{proof}

Now we are ready to prove Theorem \ref{thm:5}
\begin{proof}[Proof of Theorem \ref{thm:5}] 
Denote by $\Delta_{z,t} = \|\z_{t+1}-\nabla F(\x_t)\|^2$ and $\delta_{y,t} = \|\y_t - \y^*(\x_t)\|^2$. 
Using Lemma \ref{lem:sbo_var_1}, we have
\begin{equation}
\begin{split}
\E\left[\sum\limits_{t=0}^{T} \Delta_{z, t} \right] \leq& \frac{2\E[\Delta_{z, 0}]}{\gamma} + \sum\limits_{t=0}^{T}\left[\frac{16C_{gxy}^2}{\lambda^2}(1-\frac{\lambda}{C_{gyy}})^{2k_t}\right] + 16 C_0 \E\left[\sum\limits_{t=0}^{T} \|\y_t - \y^*(\x_t)\|^2\right] \\
&+ 2\gamma C_1(T+1) + \frac{4L_F^2}{\gamma^2}\E\left[\sum\limits_{t=0}^{T}\|\x_{t+1} - \x_t\|^2\right].
\end{split}
\label{eq:sbo_var_rearr}
\end{equation} 

Setting $k_t = O(\frac{C_{gyy}}{2\lambda} \ln (64 \sqrt{T+1} C_{gxy}^2/(\lambda^2 )) $), 
we have $16 C_{gxy}^2 \frac{1}{\lambda} \left( 1-\frac{\lambda}{C_{gyy}} \right)^{2 k_t} \leq \frac{1}{4\sqrt{T+1}}$.  
Therefore, 
\begin{small}
\begin{equation*}
\small 
\begin{split}
&\sum\limits_{t=0}^{T} \E[\Delta_{z,t}] \leq \frac{2\E[\Delta_{z, 0}]}{\gamma}
+ 16 C_0 \E\left[ \sum\limits_{t=1}^{T+1} \delta_{y,t} \right]
+2\gamma C_1(T+1) \\
&+ \frac{4L_F^2 }{\gamma^2}\E\left[ \sum\limits_{t=0}^{T} \|\x_{t+1} - \x_{t}\|^2\right] + \frac{\sqrt{T+1}}{4}. 
\end{split}
\end{equation*} 
\end{small} 

Rearranging Lemma \ref{lem:sbo_y_recursion} yields
\begin{equation} 
\begin{split} 
\sum\limits_{t=0}^{T} \E[\delta_{y,t}] &\leq \frac{2}{\eta_y \lambda} \E[\delta_{y,0}]  
+ \frac{4\eta_y \sigma^2(T+1)}{\lambda} + \sum\limits_{t=0}^{T} \frac{ 8 L_y^2 \eta_x^2}{\eta_y^2  \lambda^2} \E\|\z_{t+1}\|^2, 
\end{split} 
\end{equation}
and 
\begin{equation} 
\begin{split} 
\sum\limits_{t=1}^{T+1} \E[\delta_{y,t}] &\leq \frac{2}{\eta_y \lambda} \E[\delta_{y,0}]  
+ \frac{4\eta_y \sigma^2(T+1)}{\lambda} + \sum\limits_{t=0}^{T} \frac{ 8 L_y^2 \eta_x^2}{\eta_y^2  \lambda^2} \E\|\z_{t+1}\|^2. 
\end{split} 
\label{eq:sbo_var_rearr2}
\end{equation}

Using (\ref{eq:sbo_var_rearr}), (\ref{eq:sbo_var_rearr2}) and  Lemma \ref{lem:sbo_1}, 
\begin{small}
\begin{equation*}
\begin{split}
&\sum\limits_{t=0}^T \E\|\nabla F(\x_t)\|^2 \leq \frac{2(F(\x_0) - F(\x_{T+1}))}{\eta_x} +  \sum\limits_{t=0}^{T} \E\|\nabla F(\x_t) - \z_{t+1}\|^2 - \frac{1}{2} \sum\limits_{t=0}^{T}  \E\|\z_{t+1}\|^2 \\ 
& \leq \frac{2(F(\x_0) - F_*)}{\eta_x} 
+ \frac{2\E[\Delta_{z, 0}]}{\gamma}
+ 16 C_0 \E\left[ \sum\limits_{t=1}^{T+1} \delta_{y,t} \right]
+2\gamma C_1(T+1) \\
&+ \frac{4L_F^2 }{\gamma^2}\E\left[ \sum\limits_{t=0}^{T} \|\x_t - \x_{t-1}\|^2\right] 
+ \frac{\sqrt{T+1}}{4}  - \frac{1}{2}\sum\limits_{t=0}^{T} \E\|\z_{t+1}\|^2 \\ 
&\leq \frac{2(F(\x_0) - F_*)}{\eta_x} 
+ \frac{2\E[\Delta_{z, 0}]}{\gamma} 
+  \frac{32C_0}{\eta_y \lambda} \E[\delta_{y,0}]  
+ \frac{64C_0\eta_y  \sigma^2(T+1)}{\lambda} \\
&~~~ + \sum\limits_{t=0}^{T} \left(\frac{ 128C_0 L_y^2 \eta_x^2}{\eta_y^2  \lambda^2} + \frac{4L_F^2 \eta_x^2}{\gamma^2} - \frac{1}{2} \right) \E\|\z_{t+1}\|^2  +2\gamma C_1(T+1) + \frac{\sqrt{T+1}}{4}  \\
&\leq  \frac{2(F(\x_0) - F_*)}{\eta_x} 
+ \frac{2\E[\Delta_{z, 0}]}{\gamma}
+  \frac{32C_0}{\eta_y\lambda} \E[\delta_{y,0}]  
+ \frac{64C_0\eta_y  \sigma^2(T+1)}{\lambda} \\
&~~~ + 2\gamma C_1 (T+1) + \frac{\sqrt{T+1}}{4},
\end{split}
\end{equation*}
\end{small} 
where the last inequality is due to the fact
$\left(\frac{ 128C_0 L_y^2 \eta_x^2}{\eta_y^2 \lambda^2} + \frac{4L_F^2 \eta_x^2}{\gamma^2} - \frac{1}{2} \right) \leq -\frac{1}{4} \leq 0$ by the setting 
\begin{equation}  
\begin{split}     
\eta_x \leq \min\left(\frac{\eta_y  \lambda}{32\sqrt{C_0} L_y}, 
\frac{\gamma}{8L_F} \right). 
\end{split}
\end{equation} 

Thus,
\begin{small}
\begin{equation*}
\begin{split}
\frac{1}{T+1}\sum\limits_{t=0}^{T} \E\|\nabla F(\x_t)\|^2 \leq  \frac{2\Delta_F}{\eta_x T} 
+ \frac{2\E[\Delta_{z, 0}]}{\gamma T}
+  \frac{32C_0}{\eta_y \lambda T} \E[\delta_{y,0}]
+ \frac{64C_0\eta_y  \sigma^2}{\lambda} + 2\gamma C_1 + \frac{1}{4\sqrt{T+1}}.
\end{split} 
\end{equation*} 
\end{small} 
By the setting
$\gamma = O\left(\frac{1}{\sqrt{T}}\right)$, 
$\eta_x = O\left(\frac{1}{\sqrt{T}}\right)$, and
$\eta_y = O\left(\frac{1}{\sqrt{T}}\right)$, 
we have 
\begin{equation}
\begin{split}
\frac{1}{T+1} \E\left[ \sum\limits_{t=0}^{T} \|\nabla F(\x_t)\|^2 \right] \leq O\left(\frac{1}{\sqrt{T}}\right),
\end{split} 
\end{equation}
which concludes the first part of the theorem. 
For the second part, we have
\begin{small}
\begin{align*} 
&\sum\limits_{t=0}^{T} \E(\Delta_{z, t} + C_0 \delta_{y, t}) \\
&\leq \frac{2\E[\Delta_{z, 0}]}{\gamma}
+ \E\left[ \sum\limits_{t=0}^{T} (16 C_0 \delta_{y,t+1} + C_0 \delta_{y,t} )\right] 
+2\gamma C_1(T+1) \\
&~~~ + \frac{4L_F^2 }{\gamma^2}\E\left[ \sum\limits_{t=0}^{T} \|\x_t - \x_{t-1}\|^2\right] + \frac{\sqrt{T+1}}{4}\\
&\leq \frac{2\E[\Delta_{z, 0}]}{\gamma} + 
\frac{34C_0}{\eta_y \lambda} \E[\delta_{y,0}]  
+ \frac{68C_0\eta_y  \sigma^2(T+1)}{\lambda} + 2\gamma C_1(T+1) \\
&~~~ + \sum\limits_{t=0}^{T} \left(\frac{ 136C_0 \kappa^2 \eta_x^2}{\eta_y^2 \lambda^2} + \frac{4L_F^2}{\gamma^2} \right) \E\|\z_{t+1}\|^2 + \frac{\sqrt{T+1}}{4} \\
&\leq \frac{2\E[\Delta_{z, 0}]}{\gamma} + 
\frac{34C_0}{\eta_y \lambda} \E[\delta_{y,0}]  
+ \frac{68C_0\eta_y  \sigma^2(T+1)}{\lambda} + 2\gamma C_1(T+1) \\
&~~~ +\frac{1}{3} \sum\limits_{t=0}^{T} \E\|\z_{t+1} - F(\x_t) + F(\x_t)\|^2 + \frac{\sqrt{T+1}}{4} \\
&\leq \frac{2\E[\Delta_{z, 0}]}{\gamma} + 
\frac{34C_0}{\eta_y \lambda} \E[\delta_{y,0}]  
+ \frac{68C_0\eta_y  \sigma^2(T+1)}{\lambda} + 2\gamma C_1(T+1) \\
&~~~ + \frac{2}{3} \sum\limits_{t=0}^{T} \E[\Delta_{z,t} + \|F(\x_t)\|^2] + \frac{\sqrt{T+1}}{4}.
\end{align*}  
\end{small}   
Thus, 
\begin{equation}
\begin{split}
\frac{1}{T+1}\sum\limits_{t=0}^{T} \E(\Delta_{z, t} + C_0 \delta_{y, t}) \leq& \frac{6\E[\Delta_{z, 0}]}{\gamma T} + 
\frac{102C_0}{\eta_y \lambda T} \E[\delta_{y,0}]  
+ \frac{204 C_0\eta_y  \sigma^2}{\lambda} \\
&+ 6\gamma C_1  + \frac{2}{T+1}\sum\limits_{t=0}^{T} \|F(\x_t)\|^2] + \frac{3}{4\sqrt{T+1}}\\
\leq & O\left(\frac{1}{\sqrt{T}}\right).   
\end{split}
\end{equation} 
\end{proof}

 \section{An Alternative for Stochastic Bilevel Optimization} 
\label{section:sbo_alter}
In this algorithm, we present an alternative for stochastic bilevel optimization. Compared with Assumption \ref{ass:5}, we require a weaker assumption, i.e. without requiring $0 \preceq \mathcal{O}_{gyy}(\x, \y) \preceq C_{gyy} I$. The algorithm does projections and has a cost of $O(d_l^3)$ at each iteration. 

\begin{assumption}\label{ass:5_proj} 
For $f, g$ we assume the following conditions hold 
\begin{itemize}[leftmargin=*]
\item $g(\x,\cdot)$ is $\lambda$-strongly convex with respect to $\y$ for any fixed $\x$. 
\item $\nabla_x f(\x, \y)$ is $L_{fx}$-Lipschitz continuous, $\nabla_y f(\x, \y)$ is $L_{fy}$-Lipschitz continuous, $\nabla_y g(\x, \y)$ is $L_{gy}$-Lipschitz continuous for any $\xi$, $\nabla_{xy}^2 g(\x, \y)$ is $L_{gxy}$-Lipschitz continuous, $\nabla_{yy}^2 g(\x, \y)$ is $L_{gyy}$-Lipschitz continuous, all respect to $(\x, \y)$.
\item  $\O_{fx},  \O_{fy},   \O_{gy}, \O_{gxy}, \O_{gyy}$ are unbiased stochastic oracle of  $\nabla f(\x, \y), \nabla_y f(\x, \y)$, $\nabla_y g(\x, \y)$, $\nabla_{xy}^2 g(\x, \y)$ and $\nabla_{yy}^2 g(\x, \y)$, and they  have a variance bounded by $\sigma^2$. 
\item  $\|\nabla_y f(\x, \y)\|^2 \leq C_{fy}^2$, $\|\nabla_{xy}^2 g(\x, \y)\|^2 \leq C_{gxy}^2$. 
\end{itemize}
\end{assumption}

We initialize
\begin{equation}
\begin{split}
&\u_0 = \O_{fx}(\x_t, \y_t),\\
&\v_0 = \Pi_{C_{fy}} \O_{fy}(\x_t, \y_t),\\
&V_0 = \Pi_{L_{gxy}} \O_{gxy}(\x_t, \y_t),\\
&h_0 = \frac{k_0}{L_{gy}}\prod_{i=1}^p(I - \frac{1}{L_{gy}}\O_{gyy, i}(\x_0, \y_0))\\
&H_0 = \Pi_{1/\lambda}(h_0)
\end{split}
\end{equation}
and consider the following update: 
\begin{equation}\label{eqn:sbo_proj}
\hspace*{-0.3in}\text{SBMA:}\quad\left\{
\begin{aligned}
\u_{t+1} & = \beta\u_t + (1-\beta) \O_{fx}(\x_t, \y_t)\\
\v_{t+1}&  = \Pi_{C_{fy}} [\beta\v_t + (1-\beta) \O_{fy}(\x_t, \y_t)],\\
V_{t+1}&  = \Pi_{L_{gxy}} [\beta V_t + (1-\beta) \O_{gxy}(\x_t, \y_t)],\\
h_{t+1}&  = \frac{k_t}{L_{gy}}\prod_{i=1}^p(I - \frac{1}{L_{gy}}\O_{gyy, i}(\x_t, \y_t)),  \\
H_{t+1} &=  \Pi_{1/\lambda} [\beta H_t + (1-\beta) h_{t+1}] \\ 
\x_{t+1}& = \x_t - \eta_x( \u_{t+1} - V_{t+1}H_{t+1}\v_{t+1}),  \\
\y_{t+1}& = \Pi_{\Y}[\y_t + \eta_y \O_{gy}(\x_t, \y_t)], \quad t=0, \ldots, T.
\end{aligned}\right.,
\end{equation}
where $p\in \{1, k_t\}$  is uniformly sampled and the projection of $V_{t+1}$ and $H_{t+1}$ project the largest eigen values to $1/L_{gxy}$ and $1/\lambda$, respectively. We have the following convergence regarding SBMA. We denote $\Delta_t =  \|\u_{t+1} - V_{t+1}H_{t+1}\v_{t+1}- \nabla F(\x_t)\|^2$.


\begin{theorem}\label{thm:6}
Let  $ F(\x_0) - F_*\leq \Delta_F$. Suppose Assumption~\ref{ass:5_proj} holds. By setting
$1-\beta=\gamma = O\left(\frac{1}{\sqrt{T}}\right)$, 
$\eta_x = O\left(\frac{1}{\sqrt{T}}\right)$, 
$\eta_y = O\left(\frac{1}{\sqrt{T}}\right)$,
$k_t = O(\ln T)$, 
we have 
\begin{align*}
\frac{1}{T+1}\sum_{t=0}^T \E[\|\nabla F(\x_t)\|^2]\leq O\left(\frac{1}{\sqrt{T}}\right), \quad \frac{1}{T+1}\sum\limits_{t=0}^{T} \E[\Delta_{z, t} + C_0 \delta_{y,t}] \leq O\left(\frac{1}{\sqrt{T}}\right).   
\end{align*} 
\end{theorem}

\subsection{Proof of Theorem \ref{thm:6}} 
Denote
\begin{equation}
\begin{split}
\z_{t} = \u_{t} - V_{t} H_{t} \v_{t}. 
\end{split}
\end{equation}

Note that Lemma \ref{lem:sbo_lip_y}, Lemma \ref{lem:sbo_1} and Lemma \ref{lem:sbo_y_recursion} still hold.  
We bound $\| \nabla F(\x_t) - \z_t\|^2$ in the following lemma.
\begin{lemma}\label{lem:sbo_var_1_proj}  
For all $t\geq 0$, we have 
\begin{align*} 
&\|\nabla F(\x_t) - \z_{t+1}\|^2 = \|\nabla F(\x_t) - (\u_{t+1} -V_{t+1} H_{t+1} \v_{t+1} ) \|^2 \\ 
&\leq 2C_0\|\y_t - \y^*(\x_t)\|^2 + 2C_1\|\u_{t+1} -  \nabla_x f(\x_t, \y^*(\x_t)) \|^2 +2C_2\|\v_{t+1} -\nabla_y f(\x_t, \y^*(\x_t)) \|^2 \\
& + 2C_3\|V_{t+1} - \nabla^2_{xy}g(\x_t, \y^*(\x_t))\|^2 +2C_4\|H_{t+1} - (\nabla_{yy}^2g(\x_t, \y^*(\x_t)))^{-1}\|^2, 
\end{align*}
where $C_0=(2L_{fx}^2 +  \frac{6C_{fy}^2L_{gxy}^2}{\lambda^2} + \frac{6C_{fy}^2C^2_{gxy}L^2_{gyy}}{\lambda^4}+ \frac{6L_{fy}^2C^2_{gxy}}{\lambda^2})$, $C_1 = 2$, $C_2 = \frac{6C_{fy}^2}{\lambda^2}$, $C_3 = 6C_{gxy}^2C_{fy}^2$ and $C_4 = \frac{6C_{gxy}^2}{\lambda^2}$.
\end{lemma}

\begin{proof}
First, (\ref{eq:sbo_nabla_F}), (\ref{eq:sbo_nabla_F_x_y}), and (\ref{eq:F_y_F_y_star}) of the last section still hold i.e.,
\begin{equation*}
\begin{split}
\nabla F(\x) &= \nabla_x f(\x, \y^*(\x)) + \nabla \y^*(\x)^\top \nabla_y f(\x, \y^*(\x)) \\
&= \nabla_x f(\x, \y^*(\x)) - \nabla_{xy}^2 g(\x, \y^*(\x))[\nabla_{yy}^2 g(\x, \y^*(\x)) ]^{-1} \nabla_y f(\x, \y^*(\x)), 
\end{split}
\end{equation*}
\begin{equation*}
\begin{split}
\nabla F(\x_t, \y_t) := \nabla_x f(\x_t, \y_t) - \nabla_{xy}^2 g(\x_t, \y_t)[\nabla_{yy}^2 g(\x_t, \y_t)]^{-1} \nabla_y f(\x_t, \y_t). 
\end{split}
\end{equation*} 
and
\begin{align*}
\small
&\|\nabla F(\x_t, \y_t) - \nabla F(\x_t)\|^2 \\
&\leq \underbrace{ (2L_{fx}^2 +  \frac{6C_{fy}^2L_{gxy}^2}{\lambda^2} + \frac{6C_{fy}^2C^2_{gxy}L^2_{gyy}}{\lambda^4}+ \frac{6L_{fy}^2C^2_{gxy}}{\lambda^2})}\limits_{C_0}\|\y_t - \y^*(\x_t)\|^2 
\end{align*}
Next, 
\begin{small}
\begin{align*}
&\E\|\u_{t+1} -V_{t+1}H_{t+1} \v_{t+1}  - \nabla F(\x_t, \y^*(\x_t))\|^2\\
& = \E\| [\u_{t+1} -V_{t+1} H_{t+1} \v_{t+1}] \\
& -  [\nabla_x f(\x_t, \y^*(\x_t)) - \nabla^2_{xy}g(\x_t, \y^*(\x_t)) [\nabla_{yy}^2g(\x_t, \y^*(\x_t))]^{-1} \nabla_y f(\x_t, \y^*(\x_t))]\|^2\\
&\leq 2\E\|\u_{t+1} -  \nabla_x f(\x_t, \y^*(\x_t)) \|^2\\
&+ 2\E\|V_{t+1} H_{t+1} \v_{t+1} - \nabla^2_{xy}g(\x_t, \y^*(\x_t))[\nabla_{yy}^2g(\x_t, \y^*(\x_t))]^{-1}\nabla_y f(\x_t, \y^*(\x_t))\|^2\\ 
&\leq 2\E\|\u_{t+1} -  \nabla_x f(\x_t, \y^*(\x_t)) \|^2 + 6\E\|V_{t+1} H_{t+1} \v_{t+1} - \nabla^2_{xy}g(\x_t, \y^*(\x_t)) H_{t+1} \v_{t+1}\|^2 \\  
&+ 6\E\| \nabla^2_{xy}g(\x_t, \y^*(\x_t)) H_{t+1} \v_{t+1} - \nabla^2_{xy}g(\x_t, \y^*(\x_t))[\nabla_{yy}^2g(\x_t, \y^*(\x_t))]^{-1}\v_{t+1}\|^2\\
&+ 6\E\| \nabla^2_{xy}g(\x_t, \y^*(\x_t))[\nabla_{yy}^2g(\x_t, \y^*(\x_t)]^{-1}\v_{t+1} \\
&~~~ -  \nabla^2_{xy}g(\x_t, \y^*(\x_t))[\nabla_{yy}^2g(\x_t, \y^*(\x_t))]^{-1}\nabla_y f(\x_t, \y^*(\x_t))\|^2\\
&\overset{(a)}{\leq} 2\E\|\u_{t+1} -  \nabla_x f(\x_t, \y^*(\x_t)) \|^2 + \frac{6C_{fy}^2}{\lambda^2}\E\|V_{t+1} - \nabla^2_{xy}g(\x_t, \y^*(\x_t))\|^2 \\
&+ 6 C_{gxy}^2C_{fy}^2\E\|H_{t+1} -[\nabla_{yy}^2g(\x_t, \y^*(\x_t))]^{-1}\|^2
+ \frac{6C_{gxy}^2}{\lambda^2}\E\|\v_{t+1} -\nabla_y f(\x_t, \y^*(\x_t)) \|^2\\ 
&\overset{(b)}{\leq} C_1\E\|\u_{t+1} -  \nabla_x f(\x_t, \y^*(\x_t)) \|^2 + C_2\E\|V_{t+1} - \nabla^2_{xy}g(\x_t, \y^*(\x_t))\|^2 \\
&+C_3\E\|H_{t+1} - [\nabla_{yy}^2g(\x_t, \y^*(\x_t))]^{-1}\|^2 +C_4\E\|\v_{t+1} -\nabla_y f(\x_t, \y^*(\x_t)) \|^2, 
\end{align*} 
\end{small} 
where $(a)$ holds due to  $\|\v_{t+1}\|^2\leq C^2_{fy}$, $\|H_{t+1}\|^2 \leq \frac{1}{\lambda^2}$ and  $\|V_{t+1}\|^2\leq C^2_{gxy}$; and $(b)$ uses $C_1 = 2$, $C_2 = \frac{6C_{fy}^2}{\lambda^2}$, $C_3 = 6C_{gxy}^2C_{fy}^2$ and $C_4 = \frac{6C_{gxy}^2}{\lambda^2}$. 
\end{proof}

Now we are ready to prove Theorem \ref{thm:6}
\begin{proof}[Proof of Theorem \ref{thm:6}] 
Because $h_{t+1}$ is an biased estimator, we cannot use Lemma \ref{lem:sema_mengdi} for the variance recursion of $H_{t+1}$.
Let $e_t = (1-\gamma)((\nabla_{yy}^2 g(\x_t, \y^*(\x_t)))^{-1} - (\nabla_{yy}^2 g(\x_{t-1}, \y^*(\x_{t-1})))^{-1})$. 
We have
\begin{equation}
\begin{split}
\|e_t\|^2 &\leq (1-\gamma)^2 \frac{L^2_{gyy}}{\lambda^4} \left( \|\x_t - \x_{t-1}\|^2 + \|\y^*(\x_t) - \y^*(\x_{t-1})\|^2 \right)\\
&\leq  (1-\gamma)^2 \frac{L^2_{gyy}}{\lambda^4} (1+L_y^2)\|\x_t - \x_{t-1}\|^2. 
\end{split}
\end{equation}
and 
\begin{small}
\begin{align*}
\small
&\E[\|H_{t+1} - (\nabla_{yy}^2 g(\x_t, \y^*(\x_t)))^{-1} + e_t\|^2] \\
&= \E\left\|(1-\gamma)\left( H_t - (\nabla_{yy}^2 g(\x_{t-1}, \y^*(\x_{t-1})))^{-1})\right) + \gamma \left( h_{t+1} - (\nabla_{yy}^2 g(\x_t, \y^*(\x_t)))^{-1} \right) \right\|^2 \\
&= \E\|(1-\gamma)\left( H_t - (\nabla_{yy}^2 g(\x_{t-1}, \y^*(\x_{t-1})))^{-1})\right) 
\\
&~~~ + \gamma \left( (\nabla_{yy}^2 g(\x_t, \y_t))^{-1} - (\nabla_{yy}^2 g(\x_t, \y^*(\x_t)))^{-1} \right)  
+ \gamma \left( h_{t+1} - (\nabla_{yy}^2 g(\x_t, \y_t))^{-1} \right) \|^2 \\ 
&= \E\|(1-\gamma)\left( H_t - (\nabla_{yy}^2 g(\x_{t-1}, \y^*(\x_{t-1})))^{-1})\right) \\
&~~~ + \gamma \left( (\nabla_{yy}^2 g(\x_t, \y_t))^{-1} - (\nabla_{yy}^2 g(\x_t, \y^*(\x_t)))^{-1} \right)\|^2  + \gamma^2\E\| \left( h_{t+1} - (\nabla_{yy}^2 g(\x_t, \y_t))^{-1} \right) \|^2 \\
& ~~~
+ 2\|(1-\gamma)\left( H_t - (\nabla_{yy}^2 g(\x_{t-1}, \y^*(\x_{t-1})))^{-1})\right) \| \|\E[h_{t+1}] - (\nabla_{yy}^2 g(\x_t, \y_t))^{-1}\| \\ 
&~~~ 
+ 2\|\gamma \left( (\nabla_{yy}^2 g(\x_t, \y_t))^{-1} - (\nabla_{yy}^2 g(\x_t, \y^*(\x_t)))^{-1} \right)\| \|\E[h_{t+1}] - (\nabla_{yy}^2 g(\x_t, \y_t))^{-1}\| \\
&\leq (1+\frac{\gamma}{2})(1-\gamma)^2\E\|H_t - (\nabla_{yy}^2 g(\x_{t-1}, \y^*(\x_{t-1})))^{-1}\|^2 
\\
&~~~ + (1+\frac{2}{\gamma}) \gamma^2\E\| (\nabla_{yy}^2 g(\x_t, \y_t))^{-1} - (\nabla_{yy}^2 g(\x_t, \y^*(\x_t)))^{-1} \|^2 \\
&~~~ + 2\gamma^2(\frac{k_t^2}{C^2_{gyy}}+\frac{1}{\lambda^2})
+ \left(\frac{4}{\lambda} + \frac{4\gamma^2}{\lambda}  \right)\|\E[h_{t+1}] - (\nabla_{yy}^2 g(\x_t, \y_t))^{-1}\| \\
&\leq (1+\frac{\gamma}{2})(1-\gamma)^2\E\|H_t - (\nabla_{yy}^2 g(\x_{t-1}, \y^*(\x_{t-1})))^{-1}\|^2 
+ \frac{4\gamma}{\lambda^4} L_{gyy}^2\E\|\y_t - \y^*(\x_t)\|^2 \\
&~~~ + \gamma^2 C_5 + \left(\frac{4}{\lambda} + \frac{4\gamma^2}{\lambda}  \right)\frac{1}{\lambda}\left(1-\frac{\lambda}{C_{gyy}}\right)^{k_t},
\end{align*}
\end{small} 
where the last inequality uses Lemma \ref{lem:ghadimi_h_to_g}, Assumption \ref{ass:5_proj} and $C_5:=2\gamma^2(\frac{k_t^2}{C^2_{gyy}}+\frac{1}{\lambda^2})$.    

Denote $\Delta_{fx,t} = \E\|\u_{t+1}-\nabla_x f(\x_t, \y^*(\x_t))\|^2$, $\Delta_{fy,t}=\E\|\v_{t+1}-\nabla_y f (\x_t, \y^*(\x_t))\|^2$, 
$\Delta_{gxy,t}=\E\|V_{t+1}-\nabla_{xy}^2 g(\x_t, \y^*(\x_t))\|^2$,
$\Delta_{gyy,t} = \E\|H_{t+1}-(\nabla_{yy}^2 g(\x_t, \y^*(\x_t)))^{-1}\|^2$ and $\delta_{y,t} = \E\|\y_t - \y^*(\x_t)\|^2$. 

Using Young's inequality, 
\begin{small}
\begin{equation*}
\begin{split}
&\E\|H_{t+1} - (\nabla_{yy}^2 g(\x_t, \y^*(\x_t)))^{-1}\|^2 \\
& \leq (1+\gamma) \|H_{t+1} -  (\nabla_{yy}^2 g(\x_{t}, \y^*(\x_{t})))^{-1}+ \e_t\|^2 + (1+1/\gamma) \|\e_t\|^2 
\\
&\leq (1-\frac{\gamma}{2}) \E\|H_t - (\nabla^2_{yy} g(\x_{t-1}, \y^*(\x_{t-1})))^{-1}\|^2
+ \frac{8\gamma L^2_{gyy}}{\lambda^4}\E\|\y_t-\y^*(\x_t)\|^2 \\
&~~~ + 2\gamma^2 C_5 + 2 \left(\frac{4}{\lambda} + \frac{4\gamma^2}{\lambda}  \right)\frac{1}{\lambda}\left(1-\frac{\lambda}{C_{gyy}}\right)^{k_t} + \frac{2}{\gamma} \frac{L^2_{gyy}}{\lambda^4}(1+L_y^2)\|\x_t - \x_{t-1}\|^2,
\end{split} 
\end{equation*} 
which is 
\begin{equation}
\begin{split}
\E[\Delta_{gyy, t}] \leq& (1-\frac{\gamma}{2}) \E[\Delta_{gyy, t-1}]
+ \frac{8\gamma L^2_{gyy}}{\lambda^4} \E[\delta_{y,t}]  
+ 2\gamma^2 C_5 \\
&+ 2 \left(\frac{4}{\lambda} + \frac{4\gamma^2}{\lambda}  \right)\frac{1}{\lambda}\left(1-\frac{\lambda}{C_{gyy}}\right)^{k_t} + \frac{2}{\gamma} \frac{L^2_{gyy}}{\lambda^4}(1+L_y^2)\|\x_t - \x_{t-1}\|^2. 
\end{split} 
\end{equation} 
\end{small} 
Setting $k_t = O\left(\frac{C_{gyy}}{\lambda} \ln \left(\frac{4}{\gamma}\left(\frac{4}{\lambda} + \frac{4\gamma^2}{\lambda}  \right)\frac{\sqrt{T+1}}{\lambda}\right)\right)$,  
we have 
\begin{equation*}
    \frac{4}{\gamma} \left(\frac{4}{\lambda} + \frac{4\gamma^2}{\lambda}  \right) \frac{1}{\lambda} \left(1-\frac{\lambda}{C_{gyy}}\right)^{k_t} \leq \frac{\sqrt{T+1}}{4}.  
\end{equation*}

Let $r_t = (1-\gamma)(\nabla_x  f(\x_t, \y^*(\x_t)) - \nabla_x f(\x_{t-1}, \y^*(\x_{t-1})))$. We have
\begin{equation}
\|r_t\| \leq (1-\gamma) L_{fx}(\|\x_t - \x_{t-1}\| + \|\y^*(\x_t) - \y^*(\x_{t-1})\|) \leq (1-\gamma)L_{fx}(1+L_y) \|\x_t - \x_{t-1}\|. 
\end{equation}
and 
\begin{equation*}
\begin{split}
&\E\|\u_{t+1} - \nabla_x f(\x_t, \y^*(\x_t)) + r_t\|^2 \\
& = \E \|(1-\gamma)(\u_{t} - \nabla_x f(\x_{t-1}, \y^*(\x_{t-1}))) 
+ \gamma(\nabla_x f(\x_t, \y_t) - \nabla_x f(\x_t, \y^*(\x_t))) \\
&~~~~~~ + \gamma (\O_x(\x_t, \y_t) - \nabla_x f(\x_t, \y_t))\|^2 \\
& = \E \|(1-\gamma)(\u_{t} - \nabla_x f(\x_{t-1}, \y^*(\x_{t-1})))  
+ \gamma \nabla_x f(\x_t, \y_t) - \nabla_x f(\x_t, \y^*(\x_t))\|^2 \\
&~~~~~~ + \gamma^2 \E\|(\O_x(\x_t, \y_t) - \nabla_x f(\x_t, \y_t))\|^2 \\
&\leq (1+\frac{\gamma}{2})(1-\gamma)^2\E\|\u_{t} - \nabla_x f(\x_{t-1}, \y^*(\x_{t-1}))\|^2
+ (1+\frac{2}{\gamma}) \gamma^2 L_{fx}^2 \E\|\y_t - \y^*(\x_t)\|^2 + \gamma^2 \sigma^2
\end{split}
\end{equation*}
then, with $\gamma<\frac{1}{4}$, we get
\begin{equation*}
\begin{split} 
&\E\|\u_{t+1} - \nabla_x f(\x_t, \y^*(\x_t))\|^2 \leq (1+\gamma)\E\|\u_{t+1} - \nabla_x f(\x_t, \y^*(\x_t)) + r_t\|^2 + (1+1/\gamma)\|r_t\|^2\\
&\leq (1-\frac{\gamma}{2}) \E\|\u_t - \nabla_x f(\x_{t-1}, \y^*(\x_{t-1}))\|^2 + (1+\gamma)\gamma^2\sigma^2 + 4(1+\gamma)\gamma L_f^2 \E\|\y_t - \y^*(\x_t)\|^2 \\
&~~~ + (1+1/\gamma) (1-\gamma)^2 L_{fx}^2(1+L_y)^2 \|\x_t - \x_{t-1}\|^2 \\
&\leq (1-\frac{\gamma}{2}) \E\|\u_t - \nabla_x f(\x_{t-1}, \y^*(\x_{t-1}))\|^2
+ 2\gamma^2 \sigma^2 + 10\gamma L_{fx}^2\E\|\y_t - \y^*(\x_t)\|^2 \\
&~~~ + \frac{2L_{fx}^2(1+L_y)^2}{\gamma} \E\|\x_t - \x_{t-1}\|^2.  
\end{split}
\end{equation*}
Therefore,
\begin{equation} 
\begin{split} 
&\E [\Delta_{fx, t}] \leq (1-\frac{\gamma}{2})\E[\delta_{fx, t-1}]+ 2\gamma^2 \sigma^2 + 10\gamma L_{fx}^2\E[\delta_{y,t}] + \frac{2L_{fx}^2(1+L_y)^2}{\gamma} \E\|\x_t - \x_{t-1}\|^2.  
\end{split}
\end{equation}

With similar techniques though projections are applied to $\v$ and $V$, we can get 
\begin{equation} 
\begin{split}  
&\E[\Delta_{fy,t}] \leq (1-\frac{\gamma}{2})\E[\delta_{fy,t-1}] + 2\gamma^2 \sigma^2 + 10\gamma L_{fy}^2\E[\delta_{y,t}] + \frac{2L_{fy}^2(1+L_y)^2}{\gamma} \E\|\x_t - \x_{t-1}\|^2, \\
&\E[\Delta_{gxy,t}] \leq (1-\frac{\gamma}{2})\E[\delta_{gxy,t-1}] + 2\gamma^2 \sigma^2 + 10\gamma L_{gxy}^2\E[\delta_{y,t}] + \frac{2L_{gxy}^2(1+L_y)^2}{\gamma} \E\|\x_t - \x_{t-1}\|^2. 
\end{split}
\end{equation}

Rearranging terms and taking summation for $t=0, ..., T$, we get
\begin{small}
\begin{equation}
\begin{split}
&\sum\limits_{t=0}^{T} \E[\Delta_{fx,t}] \leq 
\frac{2\E[\Delta_{fx, 0}]}{\gamma} + 4\gamma \sigma^2 (T+1)
+ 20L_{fx}^2\sum\limits_{t=1}^{T+1} \E[\delta_{y, t}]\\
&~~~~~~~~~~~~~~~~~~~~~~~~ 
+ \frac{4L_{fx}^2(1+L_y)^2}{\gamma^2} \sum\limits_{t=0}^{T} \E\|\x_t - \x_{t-1}\|^2,
\\
& \sum\limits_{t=0}^{T} \E[\Delta_{fy,t}] \leq 
\frac{2\E[\Delta_{fy, 0}]}{\gamma} + 4\gamma \sigma^2 (T+1) 
+ 20L_{fy}^2\sum\limits_{t=1}^{T+1} \E[\delta_{y,t}] \\
&~~~~~~~~~~~~~~~~~~~~~~~~ 
+ \frac{4L_{fy}^2(1+L_y)^2}{\gamma^2}  \sum\limits_{t=0}^{T} \E\|\x_t - \x_{t-1}\|^2, \\
&\sum\limits_{t=0}^{T} \E[\Delta_{gxy,t}] \leq 
\frac{2\E[\Delta_{gxy, 0} ]}{\gamma} + 4\gamma \sigma^2 (T+1)
+ 20L_{gxy}^2\sum\limits_{t=1}^{T+1} \E[\delta_{y,t}] \\
&~~~~~~~~~~~~~~~~~~~~~~~~ 
+ \frac{4L_{gxy}^2(1+L_y)^2}{\gamma^2} \sum\limits_{t=0}^{T} \E\|\x_t - \x_{t-1}\|^2, \\
&\sum\limits_{t=0}^{T} \E[\Delta_{gyy,t}] \leq 
\frac{2\E[\Delta_{gyy, 0}]}{\gamma} + \frac{20 L_{gyy}^2}{\lambda^4}\sum\limits_{t=1}^{T+1} \E[\delta_{y,t}] + 4\gamma C_5 (T+1) \\
&~~~~~~~~~~~~~~~~~~~~~~~~ +  \frac{4}{\gamma^2} \frac{L_{gyy}^2}{\lambda^4} (1+L_y^2) \sum\limits_{t=0}^{T}\E\|\x_t - \x_{t-1}\|^2 +  \frac{\sqrt{T+1}}{4}. 
\end{split} 
\end{equation}
\end{small}

By Lemma \ref{lem:sbo_y_recursion}, we have
\begin{equation} 
\begin{split} 
\sum\limits_{t=0}^{T} \E[\delta_{y,t}] &\leq \frac{2}{\eta_y \lambda} \E[\delta_{y,0}]  
+ \frac{4\eta_y \sigma^2(T+1)}{\lambda} + \sum\limits_{t=0}^{T} \frac{ 8 L_y^2 \eta_x^2}{\eta_y^2 \lambda^2} \E\|\z_{t+1}\|^2. 
\end{split} 
\end{equation} 
and 
\begin{equation} 
\begin{split} 
\sum\limits_{t=1}^{T+1} \E[\delta_{y,t}] &\leq \frac{2}{\eta_y \lambda} \E[\delta_{y,0}]  
+ \frac{4\eta_y \sigma^2(T+1)}{\lambda} + \sum\limits_{t=0}^{T} \frac{ 8 L_y^2 \eta_x^2}{\eta_y^2 \lambda^2} \E\|\z_{t+1}\|^2. 
\end{split} 
\end{equation}

Using Lemma \ref{lem:sbo_1} and Lemma \ref{lem:sbo_var_1_proj}, 
\begin{small}
\begin{equation}
\begin{split}
&\sum\limits_{t=0}^T \E\|\nabla F(\x_t)\|^2 \leq \frac{2(F(\x_0) -  F(\x_{T+1}))}{\eta_x} + \sum\limits_{t=0}^{T} \E\|\nabla F(\x_t) - \z_{t+1}\|^2 - \frac{1}{2} \sum\limits_{t=0}^{T}  \E\|\z_{t+1}\|^2 \\ 
& \leq \frac{2(F(\x_0) - F(\x_{T+1}))}{\eta_x} 
+ 2C_0 \sum\limits_{t=0}^{T} \E\|\y_t - \y^*(\x_{t})\|^2 + 2C_1 \sum\limits_{t=0}^{T} \E\|\u_{t+1} - \nabla_x f(\x_t, \y^*(\x_t))\|^2 \\
&~~~ + 2C_2 \sum\limits_{t=0}^{T} \E\|\v_{t+1} - \nabla_y f(\x_t, \y^*(\x_t))\|^2 
+ 2C_3 \sum\limits_{t=0}^{T} \E\| V_{t+1}-\nabla_{xy}^2 g(\x_t, \y^*(\x_t)) \|_F^2 \\
&~~~ + 2C_4 \sum\limits_{t=0}^{T} \E\|H_{t+1} - (\nabla_{yy}^2 g(\x_t, \y^*(\x_t)))^{-1}\|^2  
- \frac{1}{2}\sum\limits_{t=0}^{T} \E\|\z_{t+1}\|^2 \\ 
& \leq \frac{2(F(\x_0) - F(\x_{T+1}))}{\eta_x}
+ \frac{4C_1}{\gamma}\E[\Delta_{fx,0}]
+ \frac{4C_2}{\gamma}\E[\Delta_{fy,0}]
+\frac{4C_3}{\gamma}\E[\Delta_{gxy,0}]
+\frac{4C_4}{\gamma}\E[\Delta_{gyy, 0}] \\
&~~~ 
+ \left(8C_1+8C_2+8C_3+\frac{8C_4C_5}{\sigma^2}\right)\gamma\sigma^2(T+1)  
 \\
&~~~+  
\left( \frac{8 (1+L_y)^2\eta_x^2}{\gamma^2} (C_1 L_{fx}^2+C_2L_{fy}^2 +C_3L_{gxy}^2+\frac{C_4L_{gyy}^2}{ \lambda^4})  - \frac{1}{2}\right)\sum\limits_{t=0}^{T} \E\|\z_{t+1}\|^2 +\frac{\epsilon^2(T+1)}{4} \\
&~~~ 
+  \sum\limits_{t=0}^{T}\E\left[2C_0\delta_{y,t} + 40\left(C_1 L_{fx}^2+C_2L_{fy}^2+C_3L_{gxy}^2+\frac{C_4 L_{gyy}^2}{\lambda^4}\right)\delta_{y,t+1}\right]\\  
&\leq \frac{2(F(\x_0) - F(\x_{T+1}))}{\eta_x}
+\frac{4C_0+80(C_1 L_{fx}^2+C_2L_{fy}^2+C_3L_{gxy}^2+C_4 L_{gyy}^2/\lambda^4)}{\eta_y\lambda}\E[\delta_{y,0}] \\  
&~~~ + \frac{4C_1}{\gamma}\E[\Delta_{fx,0}]
+ \frac{4C_2}{\gamma}\E[\Delta_{fy,0}]
+\frac{4C_3}{\gamma}\E[\Delta_{gxy,0}] 
+\frac{4C_4}{\gamma}\E[\Delta_{gyy, 0}] \\
&~~~
+ \bigg(\frac{(8C_0+160(C_1L_{fx}^2+C_2L_{fy}^2
+C_3L^2_{gxy} + C_4L_{gyy}^2/\lambda^4) )\eta_y}{\lambda}\\
&~~~ +(8C_1+8C_2+8C_3+\frac{8C_4C_5}{\sigma^2})\gamma \bigg)\sigma^2(T+1) 
 \\
&  ~~~ 
+\frac{\epsilon^2(T+1)}{4}  
+  
\bigg[\frac{(16C_0+320(C_1+C_2+C_3+C_4 L_{gyy}^2/\lambda^4))L_y^2\eta_x^2}{\eta_y^2\lambda^2}\\ 
&~~~~~~~~~~+ \frac{8 (1+L_y)^2\eta_x^2}{\gamma^2} (C_1 L_{fx}^2+C_2 L_{fy}^2+C_3 L_{gxy}^2+\frac{C_4L_{gyy}^2}{ \lambda^4})  - \frac{1}{2}\bigg]\sum\limits_{t=0}^{T}  \E\|\z_{t+1}\|^2.\\
\end{split} 
\end{equation}
\end{small} 

By setting 
\begin{small}
\begin{equation*}
\begin{split}
\eta_x \leq&  \min\bigg(\frac{\eta_y \lambda}{56\sqrt{(C_0+C_1+C_2+C_3+C_4 L_{gyy}^2/\lambda^4)}L_y} 
, \\
&~~~\frac{\gamma}{8(1+L_y)}(C_1 L_{fx}^2+C_2 L_{fy}^2+C_3 L_{gxy}^2+\frac{2C_4L_{gyy}^2}{ \lambda^4})^{-1} \bigg), 
\end{split} 
\end{equation*}
\end{small}
we have

\begin{equation*} 
\begin{split}
&\frac{1}{T+1}\sum\limits_{t=0}^{T} \E\|\nabla F(\x_t)\|^2 \\
&\leq \frac{2(F(\x_0) - F(\x_{T+1}))}{\eta_x}
+\frac{4C_0}{\eta_y\lambda}\E[\delta_{y,0}] + \frac{4C_1}{\gamma}\E[\Delta_{fx,0}]
+ \frac{4C_2}{\gamma}\E[\Delta_{fy,0}]
+\frac{4C_3}{\gamma}\E[\Delta_{gxy,0}] 
+\frac{4C_4}{\gamma}\E[\Delta_{gyy, 0}] \\
&+ \left(\frac{(8C_0+160(C_1L_{fx}^2+C_2L_{fy}^2
+C_3L^2_{gxy} + C_4L_{gyy}^2/\lambda^4) )\eta_y}{\lambda}+(8C_1+8C_2+8C_3+\frac{8C_4C_5}{\sigma^2})\gamma\right)\sigma^2(T+1) 
 \\ 
&  +\frac{\sqrt{T+1}}{4}. 
\end{split} 
\end{equation*}
With $\gamma=O(\frac{1}{\sqrt{T}})$, $\eta_x=O(\frac{1}{\sqrt{T}})$, and $\eta_y=O(\frac{1}{\sqrt{T}})$,
we have 
\begin{equation}
\begin{split}
\frac{1}{T+1} \E\left[ \sum\limits_{t=0}^{T}  \|\nabla F(\x_t)\|^2 \right] \leq O\left(\frac{1}{\sqrt{T}}\right),
\end{split}
\end{equation}
which concludes the first part of the theorem. 
For the second part, by Lemma \ref{lem:sbo_var_1_proj}, we have
\begin{small}
\begin{equation*} 
\begin{split} 
&\sum\limits_{t=0}^{T} \E[\Delta_{z, t} + C_0 \delta_{y,t}]  \\ 
&\leq  \sum\limits_{t=0}^{T}\E[2C_0 \delta_{y,t+1} + C_0\delta_{y, t}]  +  \sum\limits_{t=0}^{T} \left(2C_1\E[\Delta_{fx, t}] 
+2C_2\E[\Delta_{fy,t}] + 2C_3  \E[\Delta_{gxy, t}]  +2C_4  \E[\Delta_{gyy,t}] \right) \\
&\leq \frac{4C_1\E[\Delta_{fx, 0}]}{\gamma} 
+ \frac{4C_2 \E[\Delta_{fy, 0}]}{\gamma}
+ \frac{4C_3\E[\Delta_{gxy, 0} ]}{\gamma}
+ \frac{4C_4\E[\Delta_{gyy, 0}]}{\gamma}\\
& + 8(C_1 + C_2+C_3 +C_4C_5/\sigma^2)\gamma \sigma^2 (T+1) \\
& + \left( \frac{8(C_1L_{fx}^2 + C_2L_{fy}^2+C_3L_{gxy}^2+C_4L^2_{gyy}/\lambda^4)(1+L_y)^2}{\gamma^2} \right) \sum\limits_{t=0}^{T} \E[\|\z_{t+1}\|^2] + \frac{\sqrt{T+1}}{4}\\
& + \sum\limits_{t=0}^T \E[(2C_0 + 40(C_1 L_{fx}^2 + C_2 L_{fy}^2 + C_3 L_{gxy}^2 + C_4 L_{gyy}^2/\lambda^4) )\delta_{y,t+1} + C_0 \delta_{y,t}] \\  
&\leq  \frac{2(F(\x_0) - F(\x_{T+1}))}{\eta_x}
+\frac{4C_0}{\eta_y\lambda}\E[\delta_{y,0}] + \frac{4C_1}{\gamma}\E[\Delta_{fx,0}]
+ \frac{4C_2}{\gamma}\E[\Delta_{fy,0}]
+\frac{4C_3}{\gamma}\E[\Delta_{gxy,0}] 
+\frac{4C_4}{\gamma}\E[\Delta_{gyy, 0}] \\
&
+ \left(\frac{(8C_0+160(C_1L_{fx}^2+C_2L_{fy}^2
+C_3L^2_{gxy} + C_4L_{gyy}^2/\lambda^4) )\eta_y}{\lambda}+(8C_1+8C_2+8C_3+\frac{8C_4C_5}{\sigma^2})\gamma\right)\sigma^2(T+1) 
 \\
&   
+\frac{\sqrt{T+1}}{4}  
+  
\bigg[\frac{(16C_0+320(C_1+C_2+C_3+C_4 L_{gyy}^2/\lambda^4))L_y^2\eta_x^2}{\eta_y^2\lambda^2}\\ 
&~~~~~~~~~~~~~~~~~~~~~~~~~~~~~~~ + \frac{8 (1+L_y)^2\eta_x^2}{\gamma^2} (C_1 L_{fx}^2+C_2 L_{fy}^2+C_3 L_{gxy}^2+\frac{C_4L_{gyy}^2}{ \lambda^4})  - \frac{1}{2}\bigg]\sum\limits_{t=0}^{T}  \E\|\z_{t+1}\|^2,
\end{split}
\end{equation*}
which is followed by
\begin{equation*} 
\begin{split} 
&\sum\limits_{t=0}^{T} \E[\Delta_{z, t} + C_0 \delta_{y,t}] 
\leq \frac{2(F(\x_0) - F(\x_{T+1}))}{\eta_x}
+\frac{6C_0+80(C_1 L_{fx}^2+C_2L_{fy}^2+C_3L_{gxy}^2+C_4 L_{gyy}^2/\lambda^4)}{\eta_y\lambda}\E[\delta_{y,0}] \\  
& + \frac{4C_1}{\gamma}\E[\Delta_{fx,0}]
+ \frac{4C_2}{\gamma}\E[\Delta_{fy,0}]
+\frac{4C_3}{\gamma}\E[\Delta_{gxy,0}] 
+\frac{4C_4}{\gamma}\E[\Delta_{gyy, 0}] \\
&
+ \left(\frac{(12C_0+160(C_1L_{fx}^2+C_2L_{fy}^2
+C_3L^2_{gxy} + C_4L_{gyy}^2/\lambda^4) )\eta_y}{\lambda}+(8C_1+8C_2+8C_3+\frac{8C_4C_5}{\sigma^2})\gamma\right)\sigma^2(T+1) 
 \\
&
+\frac{\sqrt{T+1}}{4}  
+  
\frac{1}{3}\sum\limits_{t=0}^{T}  \E\|\z_{t+1} - F(\x_t) + F(\x_t)\|^2,
\end{split}
\end{equation*}  
\end{small}   
which implies that with parameters set as above we have
\begin{equation}
\begin{split}
\frac{1}{T+1}\sum\limits_{t=0}^{T} \E[\Delta_{z, t} + C_0 \delta_{y,t}] \leq O\left(\frac{1}{\sqrt{T}} \right). 
\end{split}
\end{equation} 
\end{proof}

\section{Min-Max Formulation of AUC Maximization Problem}
\label{sec:app_auc_formulation}
The area under the ROC curve (AUC) on a population level for a scoring function $h: \mathcal{X}\rightarrow\mathbb{R}$ is defined as 
\vspace{-0.1in} 
\begin{equation} 
AUC(h) = \text{Pr}(h(\mathbf{a}) \geq h(\mathbf{a}') \vert b=1, b'=-1),
\end{equation}  
where $\mathbf{a}, \mathbf{a}' \in \mathbf{R}^{d_0}$ are data features, $b, b' \in \{-1, 1\}$ are the labels, $\z = (\mathbf{a}, b)$ and $\z'=(\mathbf{a}', b')$ are drawn independently from $\mathbb{P}$.
By employing the squared loss as the surrogate for the indicator function which is commonly used by previous studies~\citep{ying2016stochastic,liu2018fast,liu2019stochastic}, the deep AUC maximization problem can be formulated as 
\begin{equation}
\min\limits_{\w\in\mathbb{R}^d} \E_{\z, \z'}\left[(1-h(\w;\mathbf{a}) +h(\w;\mathbf{a}'))^2 \vert b=1, b'=-1 \right], 
\label{prob:auc_square} 
\end{equation}  
where $h(\w; \mathbf{a})$ denotes the prediction score for a data sample $\mathbf{a}$ made by a deep neural network parameterized by $\w$. 
It was shown in \citep{ying2016stochastic} that the above problem is equivalent to the following min-max problem: 
\begin{equation}
\label{auc_min-max-1}
\min\limits_{(\w, s, r)} \max\limits_{y\in \mathbb{R}} f(\w, s, r, y)=\E_\z[F(\mathbf{w}, s, r, y, \z)], 
\end{equation}
where 
\begin{equation}
\begin{split}
&F(\w, s, r, y; \z) = (1-p) (h(\w; \mathbf{a})- s)^2 \mathbb{I}_{[b=1]} 
+ p(h(\w; \mathbf{a}) - r)^2\mathbb{I}_{[b=-1]}  \\
&
+ 2(1+y)(p h(\w; \mathbf{a})\mathbb{I}_{[b=-1]} 
- (1-p) h(\w;\mathbf{a}) \mathbb{I}_{[b=1]}) - p(1-p)y^2,
\end{split} 
\end{equation}
where $p=\Pr(b=1)$ denotes the prior probability that an example belongs to the positive class, and $\mathbb I$ denotes an indicator function whose output is $1$ when the condition holds and $0$ otherwise. 
We denote the primal variable by $\x=(\w, s, r)$. 

\section{Verification of Assumption \ref{ass:2}}

To verify the Assumption 2, we plot the bounds of the $\s_{t,i}$ in Figure \ref{fig:verify_ass2}. We set $\beta=0.9, \beta'=0.999, G_0 = 1e^{-7}$ as in common practice of Adam. As in other experiments, we run it for 60 epochs. We can see that within the training process, $\s_{t,i}$ is both upper and lower bounded by moderate constants. Note that when Assumption \ref{ass:2} does not hols, we can clip $\s_{t,i}$ as discussed in Section \ref{sec:adamin}.

\begin{figure}[htbp]
    \centering
    \includegraphics[width=2in]{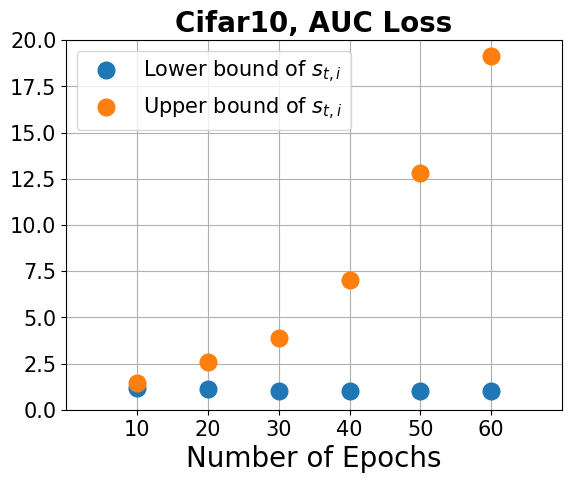}
    \caption{Verification of Assumption \ref{ass:2}} 
    \label{fig:verify_ass2} 
\end{figure}

\end{document}